\documentclass[reqno,12pt]{amsart}

\usepackage{geometry,amsmath,amsfonts,amssymb,amsthm,amscd,mathrsfs,graphicx,enumerate,multicol,wrapfig,subfigure,hyperref,stmaryrd, accents,emptypage}
\usepackage[all]{xy}
\usepackage{setspace}
\usepackage[T1]{fontenc}

\numberwithin{equation}{subsection}

\newcommand{\ra}{\rightarrow}

\newcommand{\lra}{\longrightarrow}
\newcommand{\8}{\infty}
\newcommand{\p}{\prime}
\newcommand{\pt}{\partial}
\newcommand{\eset }{\emptyset}

\newcommand{\e}{\epsilon}
\newcommand{\al}{\alpha}
\newcommand{\Om}{\Omega}
\newcommand{\om}{\omega}
\newcommand{\gam}{\gamma}

\newcommand{\s}{\sigma}

\newcommand{\q}{\theta}

\newcommand{\be}{\beta}
\newcommand{\dt}{\delta}

\newcommand{\Zbb}{\mathbb{Z}}

\newcommand{\Cbb}{\mathbb{C}}

\theoremstyle{plain} 

\newtheorem{THM}{Theorem}[section]
\newtheorem{DEF}[THM]{Definition}

\newtheorem{ILL}[THM]{Illustration}
\newtheorem{PROP}[THM]{Proposition}
\newtheorem{LEM}[THM]{Lemma}
\newtheorem{COR}[THM]{Corollary}
\newtheorem{REM}[THM]{Remark}
\newtheorem*{THMS}{Theorem}

\newcommand{\bt}{\bullet}

\newcommand{\ev}{\mathrm{ev}}

\newcommand{\Ac}{\mathcal{A}}

\newcommand{\img}{\mathrm{im}}

\newcommand{\Uc}{\mathcal{U}}

\newcommand{\Oc}{\mathcal{O}}

\newcommand{\Jc}{\mathcal{J}}

\newcommand{\Gc}{\mathcal{G}}

\newcommand{\Hc}{\mathcal{H}}
\newcommand{\Fc}{\mathcal{F}}

\newcommand{\Xfr}{\mathfrak{X}}
\newcommand{\Ec}{\mathcal{E}}

\newcommand{\Mfr}{\mathfrak{M}}

\newcommand{\Scl}{\mathcal{S}}

\newcommand{\Mcl}{\mathcal{M}}

\newcommand{\Yfr}{\mathfrak{Y}}

\usepackage{xcolor}
\definecolor{airforceblue}{rgb}{0.36, 0.54, 0.66}
\definecolor{burgundy}{rgb}{0.5, 0.0, 0.13}
\definecolor{majorelleblue}{rgb}{0.38, 0.31, 0.86}
\definecolor{darkblue}{rgb}{0.0, 0.0, 0.55}

\hypersetup{urlcolor=burgundy, linkcolor=burgundy,
citecolor=darkblue,
colorlinks=true}

\newcommand{\RNum}[1]{\uppercase\expandafter{\romannumeral #1\relax}}

\usepackage{chngcntr}
\counterwithin{section}{part} 

\title[Kosz. Splt. Thm and At. Clss.]
{Koszul's Splitting Theorem and the Super Atiyah Class
\\}
\author{\small Kowshik Bettadapura}
\date{}

\begin{document}

\begin{abstract} 
In this article we present a self-contained account of two important results in complex supergeometry: (1) Koszul's Splitting theorem and (2) Donagi and Witten's decomposition of the super Atiyah class. These results are related in the same sense that global holomorphic connections on a holomorphic vector bundle are `related' to the Atiyah class of that vector bundle---the latter being the obstruction to the existence of the former. In complex supergeometry: Koszul's theorem pertains to the existence of supermanifold splittings whereas the super Atiyah class accordingly pertains to obstructions to the existence of splittings.
\end{abstract}

\maketitle
\newpage
\setcounter{tocdepth}{1}
\tableofcontents

\onehalfspacing

\newpage
\section*{Introduction}

\noindent
A foundational problem in the area of complex supergeometry concerns the classification of supermanifolds. Berezin gave the first steps toward such a classification, detailed posthumously in \cite[p. 164]{BER}. Subsequently this classification was, in a sense, resolved by Batchelor in \cite{BAT} in the smooth setting; and, in the holomorphic setting, extended by Green in \cite{GREEN}, Manin in \cite{YMAN} and Onishchik in \cite{ONISHCLASS}. In this article we focus on the investigation into this classification by Koszul in \cite{KOSZUL}; and aim to see it as consistent with another investigation offered more recently by Donagi and Witten in \cite{DW2}, en route to their proof of the non-splitness of the supermoduli space of marked curves.
\\\\
Koszul's Theorem, being the focus of Part II of this article, asserts simply: \emph{any global, even, holomorphic connection a supermanifold furnishes it with a unique splitting}. Donagi and Witten's decomposition then, the focus of Part III, asserts contrapositively: \emph{the (affine) Atiyah class of a supermanifold projects onto the primary obstruction to splitting that supermanifold}. More classically on holomorphic vector bundles over complex manifolds, Atiyah in \cite{ATCONN} constructed a cohomology class which measures precisely the failure for the vector bundle in question to admit a global, holomorphic connection. This result has been generalised to the supermanifold setting by Bruzzo et. al. in \cite[p. 161]{BRUZZO}. Specialising to the tangent bundle, we have: \emph{the affine Atiyah class\footnote{By `affine' Atiyah class, we mean the Atiyah class of the tangent bundle.} studied by Donagi and Witten in \cite{DW2} measures the obstruction to the existence of a global holomorphic connection, as studied by Koszul in \cite{KOSZUL}}. Therefore, we might expect a relation between the (affine) Atiyah class and the obstruction classes that goes beyond the primary level. Such a relation was proposed in the author's doctoral thesis \cite{BETTPHD}. 
\\\\
In this article we will be constrained to a largely pedagogical, self-contained account of Koszul's theorem and its proof in Part II; and Donagi and Witten's decomposition and proof in Part III. While these main results  are known, there are two results we obtain en route to our proof of these main results which we might claim is new. The first result concerns the reduction of our proof of Koszul's theorem to that of lifting the Euler vector field in Theorem \ref{djvevcyibeon}. The second result concerns the relation between the affine Atiyah class of split models and supermanifolds in \S\ref{rhf894hf984hfh8033f4}, Theorem \ref{fjbkbfbuecioenice}. More generally, we hope the treatment given in this article of Koszul's theorem and Donagi and Witten's decomposition respectively will complement existing treatments in the literature; in addition to providing insight into the splitting problem for complex supermanifolds. The question of relating the (affine) Atiyah class to the higher obstructions to splitting and applications to the characterisation of algebraic connections on supermanifolds will be deferred to a planned sequel. 

In what follows we present a summary of the essential points in this article.

\subsection*{Article Summary}
This article is divided into three distinguished parts, starting with preliminaries in Part I. 

\subsubsection*{Summary of Part I}
Here one can find key definitions such as that of a model in Definition \ref{rfh9784fh984hf89h8h30}, supermanifold framings in Definition \ref{rfh894hf8hf093j03} and morphisms in Definition \ref{rfj89hg894hf0j3f09j390} leading thereby to categories of `supermanifolds' and `framed supermanifolds'. Proposition \ref{fh794f98hf03jf930} is a general result relating these categories, clarifying that any supermanifold of a prescribed dimension will admit a framing, albeit non-canonically. With the notion of a splitting then established in Definition \ref{rhf794hf89h8f30f093}, the following Lemma \ref{e8y93hf8h3f93} gives equivalent characterisations of splittings, serving as the foundation for deciding when splittings exist. The tangent sheaf of a supermanifold and its strata are then introduced in \eqref{rf74g74hf9380f390}. Unlike the general case, the tangent sheaf of split models (i.e., split supermanifolds) can be decomposed into a $\Zbb$-graded module. Each graded component is related to the modelling data in Lemma \ref{rhf79gf983hf80h03} while Corollary \ref{rgf78gf738fh309f3} concerns the tangent sheaf itself. We conclude with Lemma \ref{rhf78f9hf083j09fj30} relating both the structure sheaf and tangent sheaf of a supermanifold with those of its split model, adapting similar constructions in algebra relating filtered objects to their associated graded objects. The relations in Lemma \ref{rhf78f9hf083j09fj30} will be essential ingredients in the conceptual and technical analyses of supermanifolds to be presented in the sections to follow.

\subsubsection*{Summary of Part II} As the title suggests, Part II concerns Koszul's Splitting Theorem. It begins with Definition \ref{gf7498fh3f30jf39f0}, establishing what is meant by a connection on a supermanifold and in Definition \ref{fg74f94hf80303j0} specialised to what is meant by Koszul in \cite{KOSZUL} (see the following Remark \ref{rhf894hf80h3f0j390}). Theorem \ref{rburghurhoiejipf} is a straightforward adaptation to the supermanifold setting of the affine space of connections. Its proof is deferred to Appendix \ref{rgf78gf7hf983hf30f9j3}. With this preamble established, we state Koszul's Splitting Theorem in Theorem \ref{buie9h08d3jd33f3}. Its proof concerns the subsequent sections in Part II. Our method of proof loosely follows the original, in that we aim to establish a supermanifold splitting based on an inductive argument. Rather than focussing on a global connection directly, we follow the intuitive remarks made by Koszul in the introduction in \cite{KOSZUL}. That is, we consider the problem of lifting a certain vector field---the Euler vector field defined in Definition \ref{djkbkcvjfvjcekjlenklw}, which exists on any split model $\widehat\Xfr$, to a general supermanifold $\Xfr$ with $\widehat\Xfr$ as its split model. With the relations between the structure sheaves and tangent sheaves of $\Xfr$ and $\widehat\Xfr$ from Lemma \ref{rhf79gf983hf80h03}, the problem of lifting the Euler vector field on $\widehat\Xfr$ to $\Xfr$, defined in Definition \ref{f73gf7g498fh308hf03}, becomes a cohomological problem. This leads to the notion of \emph{Euler differential of $\Xfr$} in Definition \ref{hf78gf793hf83hf03j}. Theorem \ref{djvevcyibeon}, whose proof is deferred to \S\ref{rh84hg084g9j4344rf4},  then clarifies the motivation for studying the Euler differential, being: \emph{the Euler differential of $\Xfr$ vanishes, if and only if $\Xfr$ is split}. In \S\ref{dknnjkdbvbdkvdk} we show explicitly how a splitting of $\Xfr$ can be obtained if we can lift the Euler vector field from $\widehat\Xfr$ to $\Xfr$. The afore-mentioned Theorem \ref{djvevcyibeon} will follow as a consequence of the Euler differential being the obstruction to lifting the Euler vector field. 
\\\\
Our proof of Koszul's Theorem is now reduced to the following: \emph{show that any global (even), affine connection on $\Xfr$ can be used to solve $\dt\e_{\widehat\Xfr} = 0$, where  $\dt\e_{\widehat\Xfr}$ denotes the Euler differential of $\Xfr$}. Equivalently, that it will lift the Euler vector field. This is the subject of the concluding \S\ref{8g7fg398fh380fkk3}. The base case of our argument-by-induction is established via a generic, cohomological vanishing property in Proposition \ref{fh7949v7h49f80fh}. It implies the existence of a `1-lift' of the Euler differential (Corollary \ref{tg74gf7h8f3j093}); and hence a `mod-1 lift' of the Euler vector field (see Proposition \ref{8g7fg398fh380fkk3}). The rest of \S\ref{8g7fg398fh380fkk3} is then devoted to showing: if there exists a `mod-$\ell$ lift' of the Euler vector field, then a global (even), affine connection $\nabla$ can be used to define a mod-$(\ell+1)$ lift. This is achieved by the important, shear-like property that $\nabla$ will fix any $1$-lift $H$ of Euler vector vector field and shift the other components of $H$ to a higher stratum in the tangent sheaf. The assumption that $\nabla$ be even is essential and emphasised in the proof of Proposition \ref{rhf893f98h3f80h30}. By induction then, we can lift the Euler vector field and thereby split $\Xfr$, whence Koszul's Theorem.
\\\\
In Appendix \ref{rgf7g7fg93gf083h03} we review earlier work by Onishchik in \cite{ONISHNS} on liftings of vector fields from the split model to supermanifolds more generally. While this work is not directly related to ours, there is an important subtlety to resolve, being Onishchik's Theorem \ref{rfh9hf98hf083h0}\emph{(i)}, the mapping in \eqref{biuebviuboenvoien} and Theorem \ref{djvevcyibeon} concerning the Euler vector field lift. In the comments succeeding Theorem \ref{djvevcyibeon}, we argue that Onishchik's results are consistent with Theorem \ref{djvevcyibeon}; and subsequently deduce a new characterisation of splitting cotangent supermanifolds by reference to the Euler vector field in Theorem \ref{rhf89h38fh430hf03}.

\subsubsection*{Summary of Part III}
Our focus now turns to the \emph{obstructions} to the existence of supermanifold splittings, rather than the splittings directly. As such, it is largely independent of Part II. Before looking at the Atiyah class of supermanifolds directly, we begin with a brief review of one of the main results on Atiyah classes in Theorem \ref{rgf783gf793hf983h}, obtained originally by Atiyah in \cite{ATCONN}. Where its generalisation to the supermanifold setting is concerned, we defer to \cite[p. 161]{BRUZZO}. In contrast to the case of manifolds, the affine Atiyah space of a supermanifold\footnote{see \eqref{rfg7gf93hf8h03}} is naturally $\Zbb_2$-graded. In Lemma \ref{rhf894hf893hf0830} we see that it suffices to only consider the even component of this space, as it is here where the affine Atiyah class will be valued. Theorem \ref{rfh894hf89hf04hf0} concerns its relation to global, even, holomorphic connections. Now before we can state Donagi and Witten's decomposition theorem, whose proof forms the focal point of this Part III, it is necessary to digress on obstruction theory more generally for supermanifolds in \S\ref{rfh894hf94hf0093}. The material presented herein can mostly be found in the literature. Where our treatment of obstruction theory differs however lies in our attempt to preserve `functoriality in supermanifolds'. After reviewing the relevant material on obstruction theory as it can be found in the literature, we look give a more natural treatment of this material in \S\ref{dhbcjdbvjhfbvbuieboi}.\footnote{c.f., Theorem \ref{rfh794f9hf8h380f} in contrast to Theorem \ref{rf748gf9hf8fj93jf9334}} Proposition \ref{rffg874g97h9fh83} confirms that our treatment of obstruction theory in \S\ref{dhbcjdbvjhfbvbuieboi} is consistent with that in the literature, given in \S\ref{rhf84f0jf9fj94jf445545}. 

We include, in Appendix \ref{rh894rhf894hf8f0j30}, a digression on the classification of supermanifolds necessary for understanding the material on obstruction theory given in \S\ref{dhbcjdbvjhfbvbuieboi}. As with our treatment of obstruction theory, our treatment of the classification in Appendix \ref{rh894rhf894hf8f0j30} is tailored toward a viewpoint `functorial in supermanifolds'. See the concluding section of Appendix \S\ref{rh894rhf894hf8f0j30}.
\\\\
Our study of Atiyah classes on supermanifolds is divided into two sections. In \S\ref{rhf894hf984hfh8033f4} we study the Atiyah class of split models (i.e., split supermanifolds); with the more general case, leading to Donagi and Witten's decomposition, deferred to \S\ref{r4g4g949g49g94j904j}. Concerning the split case in \S\ref{rhf894hf984hfh8033f4} we begin a preliminary decomposition, preempting the forthcoming, general decomposition in Proposition \ref{rfh894hf89hf0j903}. That is, we obtain a decomposition of the Atiyah \emph{space} of the split model in \ref{rhf894hf89h3f30}; and subsequently a decomposition of a projection of the Atiyah class. In the spirit of Lemma \ref{rhf78f9hf083j09fj30} on the relation between supermanifolds and their split models, we give a relation between the Atiyah class of split models with supermanifolds more generally in Theorem \ref{fjbkbfbuecioenice}. The body of our article concludes with \S\ref{r4g4g949g49g94j904j}, which is devoted to the statement and proof of Donagi and Witten's decomposition of the affine Atiyah class in Theorem \ref{rfg784gf7hf98h38f03}. The statement of Theorem \ref{rfg784gf7hf98h38f03} is almost identical to Proposition \ref{rfh894hf89hf0j903} concerning split models. A crucial difference here being the appearance of the primary obstruction class in \eqref{kdnjkbjcvecbe}. The proof of Theorem \ref{rfg784gf7hf98h38f03}, which mainly concerns the appearance of this obstruction class, occupies what remains of \S\ref{r4g4g949g49g94j904j}. Our method for proving many of the statements in \S\ref{rhf894hf984hfh8033f4} and \S\ref{r4g4g949g49g94j904j} reduce to forming appropriate, commutative diagrams of sheaves, inducing commutative diagrams on cohomology whence our proposed statements follow.
\\\\
In the concluding remarks we speculate on some future directions in which Koszul's Theorem and Donagi and Witten's decomposition might lead. Among these include the characterisation of connections on supermanifolds; generalising Donagi and Witten's decomposition to the `full, affine Atiyah class', rather than its restriction to the reduced space; and a question on the nature of Atiyah classes of sheaves on supermanifolds more generally. From a broader perspective, the central constructs in this article on which the main results essentially depend are: the Euler differential of supermanifolds in Part II; and the primary obstruction to splitting in Part III. In Appendix \ref{rhf78gf79h38fh380} we include, for completeness, a commentary on these constructs including a relation in Theorem \ref{rfh89hf8hf09j03}.

\newpage
\part{Preliminary Theory}

\section{Definitions and Notation}

\noindent
\subsection{Supermanifolds}
A supermanifold $\Xfr$ is a locally ringed space $(X, \Oc_X)$ together with a sheaf of local, supercommutative rings $\Oc$ augmented over $\Oc_X$, i.e., equipped with an epimorphism $\al: \Oc\ra \Oc_X\ra 0$. We can then write $\Xfr = ((X, \Oc_X), \Oc, \al)$. The sheaf $\Oc$ is referred to as the structure sheaf of the superspace $\Xfr$ and so, accordingly, we denote: $\Oc \stackrel{\Delta}{=} \Oc_\Xfr$ and $\Xfr \stackrel{\Delta}{=} (X, \Oc_\Xfr)$. We say $\Xfr$ is a \emph{supermanifold} if there exists a locally free (l.f.) $\Oc_X$-module $\Ec$ such that $\Oc_\Xfr$ and $\wedge_{\Oc_X}^\bt \Ec$ are locally isomorphic.\footnote{Note, the topology of $\Xfr$ comes from the topological space $X$.} Fixing some l.f., $\Oc_X$-module $\Ec$, a supermanifold $\Xfr = (X,\Oc_\Xfr)$
such that $\Oc_\Xfr$ and $\wedge^\bt_{\Oc_X}\Ec$ are locally isomorphic is said to be \emph{modelled on} $((X, \Oc_X), \Ec)$ or simply $(X, \Ec)$. The data $(X, \Ec)$ is referred to as \emph{modelling data}.

\begin{DEF}\label{rfh9784fh984hf89h8h30}
\emph{Let $(X, \Ec)$ be modelling data. We refer to the space $X$ as the \emph{reduced space}; and the l.f., sheaf $\Ec$ as the \emph{odd contangent bundle}. This motivates the notation $T_{X, -}^* \stackrel{\Delta}{=}\Ec$.} 
\end{DEF}

\noindent
More intrinsically, if we are given a supermanifold $\Xfr$, \emph{its} modelling data will be referred to by: $(|\Xfr|, T_{|\Xfr|, -}^*)$. The dimension of any supermanifold $\Xfr$ modelled on $(X, T^*_{X, -})$ is $(\dim X|\mathrm{rank}~T^*_{X, -})$, with:
\begin{align}
\dim_+\Xfr = \dim X
&&
\mbox{and}
&&
\dim_-\Xfr = \mathrm{rank}~T^*_{X, -}.
\label{fnvurbuoroiepfjpjoe}
\end{align}
With $\Xfr = (X, \Oc_\Xfr)$ a supermanifold modelled on $(X, T^*_{X, -})$, the augmentation $\Oc_\Xfr\twoheadrightarrow \Oc_X$ corresponds, geometrically, to an embedding of spaces $|\Xfr|\subset \Xfr$. The supermanifold $(X, \wedge^\bt_{\Oc_X}T^*_{X, -})$,
where the augmentation is the projection of $\wedge_{\Oc_X}^\bt T^*_{X, -}$ onto its degree zero component, is referred to as \emph{the split model}. We denote this supermanifold by $\widehat\Xfr$. By definition, any supermanifold $\Xfr$ with $\dim\Xfr = \dim\widehat\Xfr$ will be locally isomorphic to $\widehat\Xfr$. Hence we can also refer to the split model $\widehat\Xfr$ intrinsically, by referring to it as the split model associated to a given supermanifold $\Xfr$. 

\subsection{Framings, Models and Splittings}
The structure sheaf $\Oc_\Xfr$, by virtue of being supercommutative is globally $\Zbb_2$-graded, so: $\Oc_\Xfr\cong \Oc_{\Xfr, +}\oplus \Oc_{\Xfr, -}$, where $\Oc_{\Xfr, +}\subset \Oc_\Xfr$ is the sheaf of even subalgebras; and $\Oc_{\Xfr, -}$ is an $\Oc_{\Xfr, +}$-module, referred to as the \emph{fermionic module}.
The kernel of the augmentation $\Oc_\Xfr \twoheadrightarrow \Oc_X$, denoted $\Jc_\Xfr$, is referred to as the \emph{fermionic ideal}. We have the following relations between the fermionic ideal, the reduced space $X$ and the fermionic module:
\begin{align}
\Oc_X
\cong 
\frac{\Oc_\Xfr}{\Jc_\Xfr}
\cong
\frac{\Oc_{\Xfr, +}}{\Jc_\Xfr}
&&
\mbox{and}
&&
\frac{\Jc_\Xfr}{\Jc_\Xfr^2}
\cong 
\frac{\Oc_{\Xfr, -}}{\Jc_\Xfr^2}
\label{duig78fg37h893h03}
\end{align}
where the above isomorphisms are as $\Oc_X$-modules. Now more generally we have, for any $m>0$, a short exact sequence
\begin{align}
0
\lra
\frac{\Jc_\Xfr^m}{\Jc_\Xfr^{m+1}}
\lra
\frac{\Oc_\Xfr}{\Jc_\Xfr^{m+1}}
\lra
\frac{\Oc_\Xfr}{\Jc_\Xfr^m}
\lra 
0.
\label{fg74f93hf83hf03}
\end{align}
The isomorphisms in \eqref{duig78fg37h893h03} reveal that the sequence in \eqref{fg74f93hf83hf03} splits when $m = 1$.

\begin{DEF}\label{rfh894hf8hf093j03}
\emph{A choice of splitting of the sequence in \eqref{fg74f93hf83hf03} for $m = 1$ is referred to as a \emph{framing} for $\Xfr$.}
\end{DEF}

\noindent
In order to give a more categorical understanding of framings we need to establish what is meant by morphisms between supermanifolds. Starting from their definition as locally ringed spaces we have:

\begin{DEF}\label{rfj89hg894hf0j3f09j390}
\emph{Let $\Xfr$ and $\Xfr^\p$ be supermanifolds with respective structure sheaves $\Oc_\Xfr$ and $\Oc_{\Xfr^\p}$ and reduced spaces $X$ and $X^\p$. A \emph{morphism} $f : \Xfr\ra \Xfr^\p$ consists of:
\begin{enumerate}[(i)]
	\item a morphism $|f|: X\ra X^\p$ of reduced spaces;
	\item a morphism of algebras $f^\sharp : |f|^*\Oc_{\Xfr^\p} \ra \Oc_\Xfr$ which commutes with the respective augmentations, i.e., a commutative diagram:
	\[
	\xymatrix{
	|f|^*\Oc_{\Xfr^\p}\ar[d]  \ar[r] & \Oc_\Xfr\ar[d]
	\\
	|f|^*\Oc_{X^\p} \ar[r] & \Oc_{X}
	}
	\]
\end{enumerate}
If $f$ is invertible, it is an isomorphism.}
\end{DEF}

\begin{DEF}\label{rh74hf983hf8j309}
\emph{A morphism of supermanifolds $f: \Xfr\ra\Xfr^\p$ is said to be \emph{even} resp. \emph{odd} if,  with respect to the extant $\Zbb_2$-gradings $\Oc_\Xfr\cong \Oc_{\Xfr,+}\oplus \Oc_{\Xfr, -}$ and $\Oc_{\Xfr^\p}\cong \Oc_{\Xfr^\p, +}\oplus \Oc_{\Xfr^\p, -}$ that
\begin{align*}
\img~f^* \Oc_{\Xfr^\p, \pm}\subset \Oc_{\Xfr, \pm}
&&
\mbox{resp.}
&&
\img~f^* \Oc_{\Xfr^\p, \pm}\subset \Oc_{\Xfr, \mp}
\end{align*}
where the image is taken under the algebra morphism $f^\sharp$ from Definition \ref{rfj89hg894hf0j3f09j390}(ii).}
\end{DEF}

\noindent
We can now form two categories of supermanifolds. The first is of \emph{framed supermanifolds} and the second of \emph{modelled supermanifolds}. As suggested by the name, this first category comprises framed supermanifolds of a given dimension $(p|q)$, denoted $\mathsf{SM}^{\mathrm{fr.}}_{p|q}$. Morphisms in this category are morphisms of supermanifolds which fix the framing. The second category is that of supermanifolds \emph{modelled on $(X, T^*_{X,-})$}. Here again we have a category with supermanifolds modelled on $(X, T^*_{X, -})$ as its objects; and morphisms being those which do not change the modelling data $(X, T^*_{X, -})$. This category is denoted $\mathsf{SM}_{(X, T^*_{X, -})}$.

\begin{PROP}\label{fh794f98hf03jf930}
There exists an equivalence of categories: 
\[
\mathsf{SM}^{\mathrm{fr.}}_{(p|q)}\cong \mathsf{SM}_{(X, T^*_{X, -})}
\] 
for some $(X, T^*_{X, -})$.
\end{PROP}

\begin{proof}
Recall that the modelling data determines the dimension of any supermanifold it models by \eqref{fnvurbuoroiepfjpjoe}. Hence the proposed equivalence makes sense dimensionally if $\dim X = p$ and $\mathrm{rank}~T^*_{X, -} = q$. Now by definition we know that a framing of any supermanifold $\Xfr$ consists of an isomorphism $\Oc_\Xfr\cong \Oc_X\oplus (\Jc_\Xfr/\Jc_\Xfr^2)$, where $\Jc_\Xfr/\Jc_\Xfr^2$ is an l.f., sheaf on $X$. Setting $T_{X, -}^* \stackrel{\Delta}{=} \Jc_\Xfr/\Jc_\Xfr^2$ shows that any supermanifold $\Xfr$ with a fixed framing will be modelled on $(X, T^*_{X, -})$ and so establishes the desired, categorical equivalence.
\end{proof}

\noindent
The notion of supermanifold morphisms from Definition \ref{rfj89hg894hf0j3f09j390} leads to splittings as follows:

\begin{DEF}\label{rhf794hf89h8f30f093}
\emph{A supermanifold $\Xfr$ is split if it is isomorphic to its associated split model $\widehat\Xfr$. Any such isomorphism is referred to as a \emph{splitting}.}  
\end{DEF}

\begin{LEM}\label{e8y93hf8h3f93}
The following statements are equivalent:
\begin{enumerate}[(i)]
	\item $\Xfr$ is split;
	\item the sequence in \eqref{fg74f93hf83hf03} splits for all $\ell$;
	\item taking W.L.O.G.,\footnote{c.f., Proposition \ref{fh794f98hf03jf930}.} $\Xfr$ modelled on $(X, T^*_{X, -})$, that $\Oc_\Xfr\cong \wedge^\bt_{\Oc_X}T^*_{X, -}$;
	\item the following two short exact sequences split:
	\begin{align*}
	0\ra \Jc_\Xfr\ra \Oc_\Xfr\ra \Oc_X\ra0
	&&
	\mbox{and}
	&&
	0\ra \Jc^2_\Xfr\ra \Jc_\Xfr\ra T^*_{X, -}\ra0
	\end{align*}
	where we have identified $T^*_{X, -} = \Jc_\Xfr/\Jc^2_\Xfr$.
\end{enumerate}
\qed
\end{LEM}

\noindent
Lemma \ref{e8y93hf8h3f93}\emph{(ii)} and \emph{(iii)} follow straightforwardly from the definition; while Lemma \ref{e8y93hf8h3f93}\emph{(iv)} is a classical result in supermanifold theory. A proof was provided by the author in \cite[Appendix A]{BETTEMB}. Note that by Lemma \ref{e8y93hf8h3f93}\emph{(ii)}, any splitting of $\Xfr$ modulo $\Jc_\Xfr^2$ will define a framing.

%
%
%
%
%
%
%


\section{The Tangent Sheaf} 
\label{h837fg983hf839j093}

\noindent
Most of the material in this section can be found in \cite{ONISHNS, ONISHCLASS}. We present it here largely to establish notation and key, foundational results to be called upon in later parts of this article.

\subsection{The Adic Filtration}
The \emph{tangent sheaf} of a supermanifold $\Xfr$ is defined as the sheaf of $\Oc_\Xfr$-derivations $T_\Xfr \stackrel{\Delta}{=} Der_{\Oc_\Xfr}\Oc_\Xfr$. With $\Jc_\Xfr\subset \Oc_\Xfr$ the fermionic ideal, powers of this ideal define a filtration of the tangent sheaf $T_\Xfr$. 
With respect to this filtration the $m$-th strata, denoted $T_\Xfr^{(m)}$, is defined as:
\begin{align}
T_\Xfr^{(m)}
\stackrel{\Delta}{=}
\left\{
\nu\in T_\Xfr
\mid
\nu(\Jc_\Xfr^\ell) \subset \Jc_\Xfr^{\ell+ m}
\right\}
\label{rf74g74hf9380f390}
\end{align}
for any $\ell\geq 0$ and where $\Jc_\Xfr^0 = \Oc_\Xfr$. While we have $T_\Xfr^{(m+1)}\subset T^{(m)}_\Xfr$, note that $T_\Xfr \neq T^{(0)}_\Xfr$. Rather, we have the following characterisation:

\begin{LEM}\label{rgf784gf97g9f8h3}
\[
T_\Xfr^{(0)}
\cong
\left\{
\nu\in T_\Xfr
\mid
(\Jc_\Xfr \stackrel{\nu}{\ra}\Oc_\Xfr \ra \Oc_X)=0
\right\}.
\]
\end{LEM}

\begin{proof}
Using firstly that $\Oc_X = \Oc_\Xfr/\Jc_\Xfr$; and secondly that $\nu(\Jc_\Xfr)\subset \Jc_\Xfr$ if and only if $\nu\in T^{(0)}_\Xfr$ it follows that the composition $\Jc_\Xfr \stackrel{\nu}{\ra}\Oc_\Xfr \ra \Oc_X$ must vanish.
\end{proof}

\noindent
We set $T^{(-1)}_\Xfr = T_\Xfr$. By Lemma \ref{rgf784gf97g9f8h3}, non-trivial elements of the quotient $T^{(-1)}_\Xfr/T^{(0)}_\Xfr$ comprise derivations $\nu$ such that the composition $\Jc_\Xfr \stackrel{\nu}{\ra}\Oc_\Xfr \ra \Oc_X$ is non-trivial. As the fermionic ideal $\Jc_\Xfr$ has nilpotence degree\footnote{i.e., that $\Jc_\Xfr^m =(0)$ for $m> \dim_-\Xfr$} $\dim_-\Xfr$ it follows that the adic filtration on $T_\Xfr$ in \eqref{rf74g74hf9380f390} will define a filtration of length $\dim_-\Xfr+1$.

\subsection{Split Tangents}
In the case where $\Xfr =\widehat\Xfr$ is the split model its structure sheaf is $\Zbb$-graded as an $\Oc_X$-module.\footnote{c.f., Lemma \ref{e8y93hf8h3f93}\emph{(iii)}.} Hence this filtration will reduce to a $\Zbb$-grading and we can write $T_{\widehat\Xfr} \cong \bigoplus_{m\geq -1}T^{\{m\}}_{\widehat\Xfr}$ as $\Oc_X$-modules.\footnote{\label{rg7f4gf7g9f8hf03}We use a different bracket style $\{-\}$ in the superscript, in contrast to $(-)$, to distinguish from the generic case. Their relation on the split model is: $T^{(m)}_{\widehat\Xfr} \cong \bigoplus_{\ell\geq m}T^{\{\ell\}}_{\widehat\Xfr}$.} Each summand can subsequently be understood as follows in relation to a model.

\begin{LEM}\label{rhf79gf983hf80h03}
Let $\widehat\Xfr$ be modelled on $(X, T^*_{X, -})$. Then for each $m$ we have a short exact sequence of $\Oc_X$-modules:
\[
0 
\lra 
\wedge^{m+1}T^*_{X, -}\otimes_{\Oc_X} T_{X, -}
\lra 
T_{\widehat\Xfr}^{\{m\}}
\lra 
\wedge^mT^*_{X, -}
\otimes_{\Oc_X}T_X
\lra
0
\]
For $m = -1$ we have: $T^{\{-1\}}_{\widehat\Xfr}\cong T_{X, -}$.
\end{LEM}

\begin{proof}
See Onishchik in \cite[p. 3]{ONISHNS}.
\end{proof}

\begin{COR}\label{rgf78gf738fh309f3}
Let $\widehat\Xfr$ be the split model, modelled on the data $(X, T^*_{X, -})$. Then
\begin{align*}
T_{\widehat\Xfr}\otimes_{\Oc_{\widehat\Xfr}}\Oc_X 
\cong 
T_X
\oplus
T_{X, -}
\end{align*}
\end{COR}

\begin{proof}
By Lemma \ref{rhf79gf983hf80h03} we have:
\begin{align}
0
\ra T^*_{X, -}\otimes T_{X, -}
\ra T^{\{0\}}_{\widehat\Xfr}
\ra T_X
\ra 0
&&
\mbox{and}
&&
0
\ra T_{X, -}
\ra T^{\{-1\}}_{\widehat\Xfr}
\ra 0
\label{fiuveboinvkbe}
\end{align}
The fermionic ideal $\Jc_{\widehat\Xfr}\subset\Oc_{\widehat\Xfr}$ can be identified with $\bigoplus_{m>0}\wedge^mT^*_{X, -}$. Set $T_{\widehat\Xfr, +} = \bigoplus_{m}T^{\{2m\}}_{\widehat\Xfr}$ and $T_{\widehat\Xfr, -} = \bigoplus_{m}T^{\{2m+1\}}_{\widehat\Xfr}$. Note that $T_{\widehat\Xfr, \pm}$ are $\Oc_{\widehat\Xfr, +}$-modules. Now with $\Oc_X = \Oc_{\widehat\Xfr}/\Jc_{\widehat\Xfr}$ and the general rule: for any l.f., $\Oc_{\widehat\Xfr}$-module $\Mcl$ that $\Mcl/\Jc_{\widehat\Xfr}\Mcl\cong \Mcl\otimes_{\Oc_{\widehat\Xfr}}(\Oc_{\widehat\Xfr}/\Jc_{\widehat\Xfr}) =  \Mcl\otimes_{\Oc_{\widehat\Xfr}}\Oc_X$ we can deduce:
\begin{align*}
T_{\widehat\Xfr}\otimes_{\Oc_{\widehat\Xfr}}\Oc_X
&\cong 
\big(T_{\widehat\Xfr, +}\otimes_{\Oc_{\widehat\Xfr, +}}\Oc_X\big)
\oplus 
\big(T_{\widehat\Xfr, -}\otimes_{\Oc_{\widehat\Xfr, +}}\Oc_X\big)
\\
&\cong
\left(
\frac{T^{\{0\}}_{\widehat\Xfr}}{\Jc_{\widehat\Xfr}T_{\widehat\Xfr}\cap T^{\{0\}}_{\widehat\Xfr}}
\oplus \cdots
\right)\oplus
\left(\frac{T^{\{-1\}}_{\widehat\Xfr}}{\Jc_{\widehat\Xfr}T_{\widehat\Xfr}\cap T^{\{-1\}}_{\widehat\Xfr}}\oplus\cdots\right)
\\
&\cong
T_X\oplus T_{X, -} 
&&
\mbox{(by \eqref{fiuveboinvkbe})}
\end{align*}
as required.
\end{proof}

\subsection{Initial Forms}
An important observation for the purposes of this article concerns the relation between the structure sheaves $\Oc_\Xfr$ and $\Oc_{\widehat\Xfr}$; and between their corresponding derivations. The latter, as has been observed earlier, is a $\Zbb$-graded sheaf. On $\Xfr$, with $\Jc_\Xfr\subset\Oc_\Xfr$ the fermionic ideal, its powers define an adic filtration on $\Oc_\Xfr$. With respect to this filtration we can realise $\Oc_{\widehat\Xfr}$ as its associated graded sheaf, i.e., that 
\begin{align}
\Oc_{\widehat\Xfr} \cong \bigoplus_{m\geq0} \frac{\Jc_\Xfr^m}{\Jc^{m+1}_\Xfr}
\label{hgf978fh380fj93}
\end{align}
\noindent 
Appropriating a classical construction from commutative algebra now\footnote{see e.g., \cite[\S5.1, p. 147]{EBUDCOMM}.} we have a mapping of sheaves of sets: $in: \Oc_\Xfr \ra \Oc_{\widehat\Xfr}$ sending any $F\in \Oc_\Xfr$ to its \emph{initial form} defined as follows: for $\ell$ the least integer such that $F\in \Oc_\Xfr^{(\ell)}$ we set $in(F) \stackrel{\Delta}{=} F\mod\Jc_\Xfr^{\ell+1}$, which is an element of $\Oc_{\widehat\Xfr}$ by \eqref{hgf978fh380fj93}.

\begin{REM}\label{hf794gf94hf80h30f3}
\emph{The initial form is a mapping $in: \Oc_\Xfr\ra\Oc_{\widehat\Xfr}$ as sheaves of sets. Both of these sheaves are over $\Cbb$, i.e., that $\Cbb\subset H^0(X, \Oc_\Xfr)$ and $\Cbb\subset H^0(X, \Oc_{\widehat\Xfr})$, meaning we can multiply germs of sections by complex numbers. Multiplication by complex scalars will not change the $\Jc_\Xfr$-adic filtration on $\Oc_\Xfr$ however, so it follows that the initial form will be a $\Cbb$-linear mapping.}
\end{REM}

\noindent
Where the tangent sheaf is concerned we have a parallel description. As has been observed earlier, the tangent sheaf of the split model $T_{\widehat\Xfr}$ is $\Zbb$-graded. Onishchik in \cite[p. 3]{ONISHNS} illustrates how
it can be realised as the associated graded sheaf to the filtration on $T_\Xfr$ from \eqref{rf74g74hf9380f390}. That is:
\begin{align}
T^{\{m\}}_{\widehat\Xfr}\cong \frac{T^{(m)}_\Xfr}{T^{(m+1)}_\Xfr}
&&
\mbox{and so}
&&
T_{\widehat\Xfr} \cong \bigoplus_{m\geq -1}\frac{T^{(m)}_\Xfr}{T^{(m+1)}_\Xfr}.
\label{rfh479f983h8fh03}
\end{align}
On $\Xfr$ now, set $\Oc_\Xfr^{(m)} = \Jc_\Xfr^m$; and denote by $\Oc_\Xfr[\![n]\!]$ resp. $T_\Xfr[\![n]\!]$ the stratum shifted by $n$, e.g., $(\Oc_\Xfr[\![n]\!])^{(m)} = \Oc_\Xfr^{(m+n)}$ and $(T_\Xfr[\![n]\!])^{(m)} = T_\Xfr^{(m+n)}$. We summarise the observations in \eqref{hgf978fh380fj93} and \eqref{rfh479f983h8fh03} in the following:

\begin{LEM}
\label{rhf78f9hf083j09fj30}
There exists a short exact sequences of sheaves of sets on $X$:
\begin{align*}
0
\lra
\Oc_\Xfr[\![1]\!]
\lra
\Oc_\Xfr
\stackrel{in}{\lra}
\Oc_{\widehat\Xfr}
\lra 
0
&&\mbox{and}&&
0
\lra
T_\Xfr [\![1]\!]
\lra
T_\Xfr
\stackrel{in}{\lra}
T_{\widehat\Xfr}
\lra
0.
\end{align*}
The sequences above will be referred to as initial form sequences.
\qed
\end{LEM}

\begin{REM}\label{nfiug4fiu3oifpi3}
\emph{The comments in Remark \ref{hf794gf94hf80h30f3} pertaining to the structure sheaf apply also to the tangent sheaf. That is, the initial form mappings on germs of tangent vectors will be $\Cbb$-linear.}
\end{REM}

\newpage
\part{Koszul's Theorem}

\section{Introduction and Statement}

\noindent
Following classical conventions, an affine connection on a supermanifold $\Xfr$ is a connection $\nabla$ on its tangent bundle. It is \emph{global} if it can be defined as a mapping of sheaves\footnote{
Note, the tensor product is over $\Cbb$ rather than $\Oc_\Xfr$. This is because we cannot identify $u\otimes (gv)$ with $(gu)\otimes v$ for all $g\in \Oc_\Xfr$ as their respective images under $\nabla$ will not coincide, as illustrated by the Leibnitz rule. They will coincide if $g$ is constant however.} $T_\Xfr\otimes_{\Cbb}T_\Xfr \ra T_\Xfr$ subject to $\Oc_\Xfr$-linearity in the first argument and the $\Oc_\Xfr$-Leibnitz rule in its second. More precisely:

\begin{DEF}\label{gf7498fh3f30jf39f0}
\emph{A global, affine connection on a supermanifold $\Xfr$ is a mapping of sheaves $\nabla : T_\Xfr\otimes_\Cbb T_\Xfr \ra T_\Xfr$ such that:
\begin{enumerate}[(i)]
	\item for any $u, v\in T_\Xfr$ and $f, g\in \Oc_\Xfr$ that:
	\begin{align*}
	\nabla(fu\otimes gv) 
&= f\nabla(u\otimes gv) 
\\
&=f \left( u(g)v + (-1)^{|u|\cdot |g|}\nabla(u\otimes v)\right)
	\end{align*}
	where $|u|$ and $|g|$ denote the respective $\Zbb_2$-parities, and;
	\item for any\footnote{for $m = m^\p = -1$ recall that $T_\Xfr^{(-1)} = T_\Xfr$ in which case we recover the connection $T_\Xfr\otimes T_\Xfr \stackrel{\nabla}{\ra} T_\Xfr$} $m^\p, m^{\p\p}\geq 0$ that $\nabla : T_\Xfr^{(m^\p)}\otimes T_\Xfr^{(m^{\p\p})} \ra T_\Xfr^{(m^\p+m^{\p\p})}$.
\end{enumerate}
We write $\nabla(u\otimes v) \stackrel{\Delta}{=}\nabla_uv$.}
\end{DEF}

\noindent
A consequence of the algebra of differential forms on a supermanifold is:

\begin{THM}\label{rburghurhoiejipf}
The space of affine connections on a supermanifold $\Xfr$ is modelled on the sections $H^0(X, \odot^2T_\Xfr^*\otimes T_\Xfr)$, where $`\odot$' refers to the symmetric tensor product. The space of even affine connections is likewise modelled on $H^0(X, (\odot^2T_\Xfr^*\otimes T_\Xfr)_+)$ where $(\odot^2T_\Xfr^*\otimes T_\Xfr)_+$ denotes the even component.\footnote{The sheaves $T_\Xfr$ and $T^*_\Xfr$ are both $\Zbb_2$-graded sheaves of $\Oc_{\Xfr,+}$-modules and so are their tensor and symmetric products. As such it makes sense to form even and odd components of the resultant products.}
\end{THM}

\begin{proof}
A more detailed account of the terminology introduced and proof is deferred to Appendix \ref{rgf78gf7hf983hf30f9j3}.
\end{proof}

\noindent
In addition to the above definition, we will also need recourse to a notion of parity. For any vector field $u$, the mapping $u\mapsto \nabla_u$ sends $T_\Xfr \ra \mathrm{End}_\Cbb T_\Xfr$. The extant $\Zbb_2$-grading of the structure sheaf $\Oc_\Xfr$ induces a $\Zbb_2$-grading of the tangent sheaf $T_\Xfr\cong T_{\Xfr, +}\oplus T_{\Xfr, -}$ and so also a grading of the endomorphisms $\mathrm{End}_\Cbb T_\Xfr\cong (\mathrm{End}_\Cbb T_\Xfr)_+\oplus (\mathrm{End}_\Cbb T_\Xfr)_-$. This leads to the following notion of parity for affine connections.

\begin{DEF}\label{fg74f94hf80303j0}
\emph{A connection $\nabla$ on $T_\Xfr$ is said to be \emph{even} if the linear mapping $u\mapsto \nabla_u$ is valued in $(\mathrm{End}_\Cbb T_\Xfr)_+$ for all $u\in T_\Xfr$. As a tensor product over $\Cbb$ an even connection induces mappings
\begin{align*}
T_{\Xfr, +}\otimes_\Cbb T_{\Xfr, +}\ra T_{\Xfr, +}
&&
T_{\Xfr, +}\otimes_\Cbb T_{\Xfr, -}\ra T_{\Xfr, -}
&&
\mbox{and}
&&
T_{\Xfr, -}\otimes_\Cbb T_{\Xfr, -}\ra T_{\Xfr, +}.
\end{align*}}
\end{DEF}

\begin{REM}\label{rhf894hf80h3f0j390}
\emph{Kozsul in \cite[p. 155]{KOSZUL} considers all connections to be `even' in the sense of Definition \ref{fg74f94hf80303j0}. Hence, in restricting our attention to these even connections, we will not lose the generality inherent in Koszul's theorem.}
\end{REM}

\noindent
In addition to the equivalent conditions for deducing splittings of supermanifolds in Lemma \ref{e8y93hf8h3f93}, Koszul's theorem from \cite{KOSZUL} relates another, sufficient condition for the existence of splittings. We can now state this theorem below, whose proof will occupy Section \ref{8g7fg398fh380fkk3}.

\begin{THM}\label{buie9h08d3jd33f3}
(Koszul's Theorem)
Let $\Xfr$ be a supermanifold and suppose it admits a global, even, affine connection $\nabla$. Then $\nabla$ will define a unique splitting $\Xfr\cong \widehat\Xfr$.
\end{THM}

\section{The Euler Vector Field and Differential}
\label{rhfiugf7g8fh003}

\noindent
A central construct in our proof of Koszul's Theorem concerns the Euler vector field. It is a natural, global object which exists on any split model and the problem of lifting it to a global object on given a supermanifold is equivalent to splitting that supermanifold. This is, in essence, our main observation here. 

\subsection{Construction}
Let $\widehat\Xfr$ be the split model associated to $(X, T^*_{X, -})$. Its tangent sheaf $T_{\widehat\Xfr}$ is $\Zbb$-graded and is related to exterior powers of the fermoinic module $T^*_{X, -}$ as in Lemma \ref{rhf79gf983hf80h03}. We are interested in this sequence in degree zero which we present below for convenience:
\begin{align}
0
\lra 
T^*_{X, -}\otimes_{\Oc_X} T_{X, -}
\stackrel{i}{\lra}
T_{\widehat\Xfr}^{\{0\}}
\lra
T_X
\lra 
0
\label{fnuiebiuveuivbeuo}
\end{align}
Identifying $T^*_{X, -}\otimes_{\Oc_X} T_{X, -}\cong \mathcal End_{\Oc_X}T_{X, -}$, we have on global sections the embedding $H^0(X, \mathcal End_{\Oc_X}T^*_{X, -})\stackrel{i_*}{\subset} H^0(X, T_{\widehat\Xfr}^{\{0\}})$. Hence any global endomorphism of the odd cotangent bundle $T^*_{X, -}$ will define a global, degree zero tangent vector on the split model. 

\begin{DEF}\label{djkbkcvjfvjcekjlenklw}
\emph{Let $\widehat\Xfr$ be the split model associated to $(X, T^*_{X, -})$. The \emph{Euler vector field} on $\widehat\Xfr$, denoted $\e_{\widehat\Xfr}$, is the global, degree zero vector field in the image of the identity ${\bf 1}_{T_{X, -}^*}\in H^0(X, \mathcal End_{\Oc_X}T^*_{X, -})$ under the embedding $H^0(X, \mathcal End_{\Oc_X}T^*_{X, -})\stackrel{i_*}{\subset} H^0(X, T_{\widehat\Xfr}^{\{0\}})$. That is,
\[
\e_{\widehat\Xfr} \stackrel{\Delta}{=}i_*{\bf 1}_{T^*_{X, -}}.
\]}
\end{DEF}

\noindent
Hence, any split model will admit at least \emph{one} global, non-zero vector field---the Euler vector field.

\subsection{The Euler Differential}
In contrast to the split model $\widehat\Xfr$, supermanifolds $\Xfr$ more generally need not admit global, non-zero tangent vectors. To see why, recall the sequence in Lemma \ref{rhf78f9hf083j09fj30} relating the tangent sheaf $T_\Xfr$ to that of $T_{\widehat\Xfr}$. Conjoin this sequence in degree zero with \eqref{fnuiebiuveuivbeuo}, giving:
\begin{align}
\xymatrix{
& & & 0\ar[d] & 
\\
& & & \mathcal End_{\Oc_X}T^*_{X, -}\ar[d]^i
\\
0\ar[r] & T_\Xfr^{(1)} \ar[r] & T^{(0)}_\Xfr \ar[r]_p & T^{\{0\}}_{\widehat\Xfr}\ar[r] & 0
}
\label{rfh94hf98hf03f90j3}
\end{align}
and so, on cohomology:
\begin{align}
\xymatrix{
& & 0\ar[d]& &
\\
& &  H^0(\mathcal End_{\Oc_X}T^*_{X, -}) \ar[d]^{i_*}& &
\\
\cdots\ar[r] & H^0(T^{(0)}_\Xfr) \ar[r]^{p_*} & H^0(T^{(0)}_{\widehat\Xfr})\ar[r]^\dt & H^1(T^{(1)}_\Xfr)\ar[r] & \cdots
}
\label{rhf794f9h30f3}
\end{align}
where $H^j(-) = H^j(X, -)$. We arrive now at important definitions for the purposes of this article.

\begin{DEF}\label{f73gf7g498fh308hf03}
\emph{Let $\Xfr$ be a supermanifold with split model $\widehat\Xfr$. The Euler vector field $\e_{\widehat\Xfr}$ is said to \emph{lift to $\Xfr$} if there exists some global vector field $H\in H^0(X, T_\Xfr^{(0)})$ such that $p_*H = \e_{\widehat\Xfr}$. }
\end{DEF}

\begin{DEF}\label{hf78gf793hf83hf03j}
\emph{The \emph{Euler differential} of $\Xfr$ is the image of the Euler vector field $\e_{\widehat\Xfr}$ under the cohomological boundary mapping $\dt$ in \eqref{rhf794f9h30f3}.}
\end{DEF}

\begin{REM}
\emph{Onishchik in \cite{ONISHNS} also considers the problem of lifting vector fields from the split model to supermanifolds. Accordingly, a more general definition of vector field lifting is given in \cite{ONISHNS}, subsuming our Definition \ref{f73gf7g498fh308hf03}. While Onishchik's paper is not directly related to the present work, it is perhaps tangentially related. For completeness therefore we present a short review of this paper in Appendix \ref{rgf7g7fg93gf083h03} in addition to some comments relating it to the present work.}
\end{REM}

\noindent
Exactness of the bottom row in \eqref{rhf794f9h30f3} ensures that the Euler differential \emph{of} $\Xfr$ measures the obstruction to the existence of a lift of the Euler vector field \emph{to} $\Xfr$.\footnote{c.f., \cite[Proposition 3.1, p. 59]{ONISHNS}} Now if $\Xfr = \widehat\Xfr$ then obviously $\dt\e_{\widehat\Xfr} = 0$. This vanishing will hold in the case where $(\Xfr = \widehat\Xfr)$ is replaced by $(\Xfr \cong \widehat\Xfr)$. Our observation is that that the vanishing of the Euler differential will imply this splitting also. That is, we have:

\begin{THM}\label{djvevcyibeon}
Any supermanifold $\Xfr$ is split if and only if its Euler differential vanishes.
\end{THM}

\noindent
A proof of Theorem \ref{djvevcyibeon} will become apparent following \S\ref{dknnjkdbvbdkvdk} concerning supermanifold diffeomorphisms. 

\subsection{Useful Properties}
In a local coordinate system $(x|\q)$ on $\widehat\Xfr$ the Euler vector field $\e_{\widehat\Xfr}$ can be represented as the derivation:
\begin{align}
\e_{\widehat\Xfr}(x|\q) = \sum_{i=1}^{\dim_-\widehat\Xfr}\q_i\frac{\pt}{\pt \q_i}.
\label{g683f863g793893}
\end{align}
Since $\e_{\widehat\Xfr}$ is global we can equate $\e_{\widehat\Xfr}(x|\q) = \e_{\widehat\Xfr}(y|\eta)$ on intersections of coordinate neighbourhoods $(x|\q)\cap (y|\eta)$. Therefore we can justify deriving certain global statements from the local expression for $\e_{\widehat\Xfr}$ in \eqref{g683f863g793893}. Now with $\Oc_{\widehat\Xfr}$ denoting the structure sheaf of the split model $\widehat\Xfr$ we know from Lemma \ref{e8y93hf8h3f93}\emph{(iii)} that it isomorphic to an exterior algebra. Denote by $\Oc_{\widehat\Xfr}^{\{m\}}$ each graded piece so that, as an $\Oc_X$-module, we have $\Oc_{\widehat\Xfr}\cong \bigoplus_{m\geq0}\Oc_{\widehat\Xfr}^{\{m\}}$. If $\widehat\Xfr$ is associated to $(X, T^*_{X, -})$ then $\Oc_{\widehat\Xfr}^{\{m\}} \cong \wedge^m_{\Oc_X}T^*_{X, -}$. Direct calculation with the local expression for $\e_{\widehat\Xfr}$ in \eqref{g683f863g793893} reveals the following properties which will be essential in forthcoming calculations:

\begin{LEM}\label{rbyurgf4gf233393hf89h3}
The Euler vector field $\e_{\widehat\Xfr}$ satisfies:
\begin{enumerate}[(i)]
	\item $\e_{\widehat\Xfr}(f) = mf$ iff $f\in \Oc_{\widehat\Xfr}^{\{m\}}$ and;
	\item $\mathrm{ad}_{\e_{\widehat\Xfr}}(v) = mv$ iff $v\in T_{\widehat\Xfr}^{\{m\}}$.
\end{enumerate}
where $\mathrm{ad}_u(v) = [u, v]$.
\qed
\end{LEM}

\section{Diffeomorphisms and Splittings}
\label{dknnjkdbvbdkvdk}

\noindent
Before embarking on a proof of Koszul's Theorem (Theorem \ref{buie9h08d3jd33f3}) we will motivate our method 
by explicating the relation between diffeomorphisms and splittings.  

\subsection{The Inverse Function Theorem}
En route to our forthcoming definition of a supermanifold diffeomorphism, we recall the classical Inverse Function Theorem on manifolds for motivation. Let $M$ and $N$ be smooth manifolds and $f: M\ra N$ a differentiable function. It defines a linear map $Tf: TM \ra f^*TN$ of vector bundles over $M$, given by $(Tf)_p : T_pM\ra T_{f(p)}N$. If $(Tf)_p$ is invertible, then $f$ will diffeomorphically map an open neighbourhood of $p$ onto its image. If $Tf$ is an isomorphism of vector bundles over $M$, $f$ is said to be a \emph{local} diffeomorphism of manifolds. Hence $f$ is locally $C^\8$ and invertible. It is globally $C^\8$ and invertible, i.e., a diffeomorphism, if it is, in addition, bijective as a mapping of point-sets. Where supermanifolds are concerned, a generalisation of this theorem to this case can be found in \cite[\S3]{LEI}. We will take it here as the basis for a definition however.

\begin{DEF}\label{rfg8gf7hf983hf03}
\emph{Let $\Xfr$ and $\Yfr$ be supermanifolds and suppose $f: \Xfr\ra\Yfr$ is an even morphism.\footnote{c.f., Definition \ref{rh74hf983hf8j309}.} We say $f$ is a \emph{diffeomorphism} if and only if:
\begin{enumerate}[(i)]
	\item $|f|$ is a bijection and;
	\item $Tf$ is an isomorphism.
\end{enumerate}}
\end{DEF}

\begin{LEM}\label{rh793f93h08h39j3}
Let $f: \Xfr\ra \Yfr$ be a diffeomorphism. Then $|f| : X\ra Y$ is a diffeomorphism and the odd cotangents $T_{X, -}^*$ and $T^*_{Y, -}$ are isomorphic. 
\end{LEM}

\begin{proof}
We know that $|f|$ is a bijection. It remains to argue that $Tf$ will give an isomorphism between the respective tangent bundles $T_X$ and $T_Y$, whence we can apply the classical Inverse Function Theorem. To deduce this now we will use:
\begin{enumerate}[(i)]
	\item let $R$ be a supercommutative ring with fermionic ideal $J\subset R$; and let $M$ and $N$ be $R$-modules. Then if $f: M\ra N$ is an isomorphism, it will define an isomorphism between $M$ and $N$ modulo $J$ (see e.g., \cite[Proposition 1.7.2]{LEI});
	\item With $\Oc_\Xfr$ the structure sheaf of $\Xfr$ and any flat, $\Oc_\Xfr$-module $\Mcl$ we have the identification $\Mcl/\Jc_\Xfr\Mcl\cong \Mcl\otimes_{\Oc_\Xfr}\Oc_X$.
\end{enumerate}
Now with $T_\Xfr$ the tangent sheaf of $\Xfr$ we have the decomposition $T_\Xfr = T_{\Xfr, +}\oplus T_{\Xfr, -}$ induced from the global, $\Zbb_2$-grading on $\Oc_\Xfr$.  The tangent sheaf is locally free (c.f., \cite[\S2]{LEI}) and hence flat so we can use (ii) above. Now from Lemma \ref{rhf78f9hf083j09fj30} we can identify $T_\Xfr$ and $T_{\widehat\Xfr}$ modulo the fermionic ideal, i.e., that there exists an isomorphism $T_\Xfr\otimes_{\Oc_\Xfr}\Oc_X\cong T_{\widehat\Xfr}\otimes_{\Oc_{\widehat\Xfr}}\Oc_X$. Using this, (ii) and Corollary \ref{rgf78gf738fh309f3} gives:
\begin{align}
\frac{T_\Xfr}{\Jc_\Xfr T_\Xfr}
\cong 
T_\Xfr\otimes_{\Oc_\Xfr}\Oc_X
\cong 
T_{\widehat\Xfr}\otimes_{\Oc_{\widehat\Xfr}}\Oc_X
=
T_X\oplus T_{X, -}.
\label{fg874gf793gf3h8fh30}
\end{align} 
Now recall that we are assuming $f: \Xfr\ra\Yfr$ is a diffeomorphism; and therefore that $Tf: T_\Xfr \stackrel{\cong}{\ra}f^*T_\Yfr$. By (i) and \eqref{fg874gf793gf3h8fh30} it follows that 
\begin{align}
T_X\oplus T_{X, -} \cong |f|^*(T_Y\oplus T_{Y, -}). 
\label{endu3hf98h38fh03}
\end{align}
Since $f$ must even by Definition \ref{rfg8gf7hf983hf03}, the isomorphism in \eqref{endu3hf98h38fh03} will be an isomorphism of summands, i.e., it necessitates $T_X\cong |f|^*T_Y$ and $T_{X, -}\cong |f|^*T_{Y, -}$ respectively. Lemma \ref{rh793f93h08h39j3} now follows.
\end{proof}

\noindent
Definition \ref{rfj89hg894hf0j3f09j390} concerned morphisms and isomorphisms of supermanifolds. How do they compare with Definition \ref{rfg8gf7hf983hf03} concerning diffeomorphisms? It may well be routine to show isomorphisms and diffemorphisms coincide. For our purposes in this article however, we will only need to establish a particular case, being the following:

\begin{PROP}\label{yrgf8g47gh498jf093}
Any diffeomorphism between $\Xfr$ and its split model $\widehat\Xfr$ will define a splitting of $\Xfr$.
\end{PROP}

\noindent
Our proof of Proposition \ref{yrgf8g47gh498jf093} will resort to the following general construction.

\subsection{Splitting Operators}
In Lemma \ref{rhf78f9hf083j09fj30} we documented that, for any smooth supermanifold $\Xfr$, its tangent sheaf fits naturally into a graded, short exact sequence of sheaves, referred to as the initial form sequence. That is, for each $m$,  we have a short exact sequence,
\begin{align*}
0 \lra T_\Xfr^{(m+1)}\lra T_\Xfr^{(m)} \lra T_{\widehat\Xfr}^{\{m\}}\lra0.
\end{align*}
More generally now, suppose $\Fc$ is a sheaf over $\Cbb$ equipped with a filtration $\Jc = (\cdots\Fc^{(m+1)}\subset \Fc^{(m)}\cdots)$. Set $\widehat\Fc^{\{m\}} = \Fc^{(m)}/\Fc^{(m+1)}$ as the $m$-th component of the associated graded sheaf so that, in analogy with the tangent sheaf in Lemma \ref{rhf78f9hf083j09fj30}, we have an initial form sequence:
\[
\Scl(\Fc): 0 \lra \Fc[\![1]\!] \lra \Fc \lra \widehat\Fc\lra0.
\]
For each $m$ we denote $\Scl^{(m)}\Fc: 0 \ra \Fc^{(m+1)}\ra \Fc^{(m)} \ra \widehat\Fc^{\{m\}}\ra0$. A \emph{splitting operator} on $(\Fc,\Jc)$, denoted $O^{(m)}(\Fc)$ or more simply $O^{(m)}$ if confusion is unlikely, is then a certain, $\Cbb$-linear operator $O^{(m)} : \Fc^{(m)} \ra \Fc^{(m)}$ contrived precisely to split the sequence $\Scl^{(m)}\Fc$.

\begin{LEM}\label{fvuirguieojpoeke}
Let $O^{(m)} : \Fc^{(m)} \ra \Fc^{(m)}$ be a $\Cbb$-linear operator such that:
\begin{enumerate}[(i)]
	\item $\ker O^{(m)} = \Fc^{(m+1)}$ and;
	\item $O^{(m)}(v)\equiv v \mod \Fc^{(m+1)}$.
\end{enumerate} 
Then $O^{(m)}$ will define a splitting of $\Scl^{(m)}\Fc$.
\end{LEM}

\begin{proof}
We begin by noting the existence of a commutative diagram:
\[
\xymatrix{
0 \ar[r] & \Fc^{(m+1)} \ar[d]\ar[r]& \Fc^{(m)}\ar[d]^{O^{(m)}} \ar[r] & \widehat\Fc^{\{m\}}\ar@{=}[d] \ar[r] & 0
\\
0 \ar[r] & 0\ar[d]  \ar[r]& \img~O^{(m)} \ar@{^{(}->}[d] \ar[r] &  \widehat\Fc^{(m)} \ar[r]\ar@{=}[d] & 0
\\
0 \ar[r] & 0 \ar[d]\ar[r]& \Fc^{(m)}\ar[d]^{O^{(m)}}  \ar[r] &  \widehat\Fc^{\{m\}} \ar@{=}[d] \ar[r] & 0
\\
0 \ar[r] & 0 \ar[r]& \img~O^{(m)}\ar[r] &  \widehat\Fc^{\{m\}} \ar[r] & 0
}
\]
The second, horizontal row of arrows show $\img~O^{(m)}\cong  \widehat\Fc^{\{m\}}$. The third column of arrows shows we have a commutative diagram:
\[
\xymatrix{
& \Fc^{(m)}\ar[dr] &
\\
\ar@{^{(}->}[ur]\img~O^{(m)} \ar@{=}[rr]& & \img~O^{(m)}
}
\]
Hence the short exact sequence $0 \ra \ker O^{(m)} \ra \Fc^{(m)}\ra \img~O^{(m)} \ra 0$ is split, giving $\Fc^{(m)}\cong \ker O^{(m)}\oplus \img~O^{(m)}\cong \Fc^{(m+1)}\oplus  \widehat\Fc^{\{m\}}$, as required.
\end{proof}

\noindent
The following is now a direct consequence of Lemma \ref{fvuirguieojpoeke}.

\begin{COR}\label{rfh9hf983hf803hf03}
If there exist splitting operators $O^{(m)}$ on $(\Fc, \Jc)$ for each $m$, then $\Fc$ is split i.e., there exists an isomorphism $\Fc\cong \bigoplus_m \widehat\Fc^{\{m\}}$.\footnote{Note, we are not assuming any restrictions on $\Jc$ such as finiteness or nilpotency. This result is therefore very general and formal.}
\qed
\end{COR}

\subsection{Proof of Proposition $\ref{yrgf8g47gh498jf093}$}
Let $\Xfr$ be our supermanifold modelled on $(X, T^*_{X, -})$. Assuming $\Xfr$ is diffeomorphic to its split model, Definition \ref{rfg8gf7hf983hf03}(ii) asserts $T_\Xfr$ and $T_{\widehat\Xfr}$ will be isomorphic as sheaves on $X$. Hence the sequence relating $T_\Xfr$ with $T_{\widehat\Xfr}$ in Lemma \ref{rhf78f9hf083j09fj30} will be split. If $spl.$ denotes this splitting in degree zero, then it induces the mapping $spl._*: H^0(X, T_{\widehat\Xfr}^{(0)})\ra H^0(X, T_\Xfr^{(0)})$ on cohomology and evidently will lift the Euler vector field to $\Xfr$. We set:
\begin{align}
\e_\Xfr \stackrel{\Delta}{=} spl._*\e_{\widehat\Xfr}.
\label{vhfgbjdgdjbhfbd}
\end{align}
Our strategy for proving Proposition \ref{yrgf8g47gh498jf093} now will be to construct splitting operators on the structure sheaf $(\Oc_\Xfr, \Jc_\Xfr)$, where $\Jc_\Xfr$ is the fermionic ideal whose powers define a filtration on $\Oc_\Xfr$. We can then deduce splitness by Lemma \ref{e8y93hf8h3f93}\emph{(iv)}. This filtration on $\Oc_\Xfr$ has length $\dim_-\Xfr$ and so Corollary \ref{rfh9hf983hf803hf03} requires the existence of $\dim_-\Xfr$-many splitting operators. However, note by Lemma \ref{e8y93hf8h3f93}\emph{(iv)} that it will be sufficient to only have \emph{two} such operators. Returning now to the lifted Euler vector field $\e_\Xfr$ in \eqref{vhfgbjdgdjbhfbd}, since $spl.$ splits the initial form sequence for the tangent sheaf in Lemma \ref{rhf78f9hf083j09fj30} in degree zero, it splits the bottom row in \eqref{rfh94hf98hf03f90j3} and hence 
\[
p_*\e_\Xfr = p_*spl._*\e_{\widehat\Xfr} = \e_{\widehat\Xfr},
\]
where $p_*$ is the map induced on cohomology in \eqref{rhf794f9h30f3}. By Lemma \ref{rbyurgf4gf233393hf89h3}\emph{(i)} we therefore have for any $F\in \Oc_\Xfr$,
\begin{align}
p_* \e_\Xfr (F) = \e_{\widehat\Xfr}\big( in(F) \big)= |in(F)|~in(F)
\label{rg78gf87f98h3f033}
\end{align}
where $|in(F)|$ is the homogeneous degree of the initial form $in(F)\in \Oc_{\widehat\Xfr}$. Now recall the sequence for $\Oc_\Xfr$ in \eqref{fg74f93hf83hf03}. Using \eqref{hgf978fh380fj93} we obtain the following for each $m$:
\begin{align}
0
\lra 
\Oc_{\widehat\Xfr}^{\{m\}}
\lra 
\Oc_\Xfr/\Jc^{m+1}_\Xfr
\lra 
\Oc_\Xfr/\Jc^m_\Xfr
\lra 
0
\label{fknvkgfbnveoo}
\end{align}
where $\Oc_{\widehat\Xfr}^{\{m\}}$ corresponds to the $m$-th summand in \eqref{hgf978fh380fj93}. On $\Oc_\Xfr$ now we can form the operator,
\[
\widetilde O^{(m)} \stackrel{\Delta}{=} {\bf m} - \mathrm{ad}_{\e_\Xfr}
\]
where ${\bf m} = m{\bf 1}_{\Oc_{\Xfr}}$ and ${\bf 1}_{\Oc_\Xfr}: \Oc_\Xfr=\Oc_\Xfr$ is the identity map. We have:

\begin{LEM}\label{rhf784f9hf803h09j0}
For each $m$ and $\ell$, 
$\widetilde O^{(m)} : \Jc^\ell \ra \Jc^\ell$ is $\Cbb$-linear with $\ker \widetilde O^{(m)} = \Oc_{\widehat\Xfr}^{\{m\}}$.
\end{LEM}

\begin{proof}
That $\widetilde O^{(m)}$ is $\Cbb$-linear follows from $\Cbb$-linearity of its constituent components ${\bf m}$ and $\mathrm{ad}_{\e_\Xfr}$. It remains to verify $\img~\widetilde O^{(m)}|_{\Jc_\Xfr^\ell}\subset\Jc_\Xfr^\ell$ for all $\ell$. But this is also clear upon inspection of these components. To make sense of the statement $\ker \widetilde O^{(m)} = \Oc_{\widehat\Xfr}^{\{m\}}$, observe: since $\Jc_\Xfr^{\ell+1}\subset \Jc_\Xfr^\ell$; and $\widetilde O^{(m)}$ sends $\Jc_\Xfr^{\ell+1}\ra \Jc_\Xfr^{\ell+1}$, it will induce a mapping of the sequence in \eqref{fknvkgfbnveoo} and hence of the quotient $\Jc_\Xfr^\ell/\Jc^{\ell+1}_\Xfr = \Oc_{\widehat\Xfr}^{\{\ell\}} \ra \Oc_{\widehat\Xfr}^{\{\ell\}}$. Now by \eqref{rg78gf87f98h3f033} note that for any $j^\ell\in \Jc^\ell$,
\begin{align}
in\big(\widetilde O^{(m)}(j^\ell)\big) 
&=
in\big({\bf m}(j^\ell)\big) - in\big(\mathrm{ad}_{\e_\Xfr}(j^\ell)\big)
&&
\mbox{(c.f., Remark \ref{hf794gf94hf80h30f3})}
\notag
\\
&= (m-\ell) ~in(j^\ell).
\label{rh79hf983hf93h0}
\end{align}
Thus $\widetilde O^{(m)}$ sends $\Oc_{\widehat\Xfr}^{\{\ell\}} \ra 0$ iff $\ell = n$.
\end{proof}

\noindent
By Lemma \ref{rhf784f9hf803h09j0} and its proof we have the following diagram for each $m$:
\begin{align}
\xymatrix{
0 \ar[r] & \Jc_\Xfr^{m+1}\ar[d]  \ar[r] & \Jc_\Xfr^m \ar[d] \ar[r] & \Oc_{\widehat\Xfr}^{\{m\}}\ar[r] \ar[d] & 0
\\
0 \ar[r] & \img~\widetilde O^{(m)} \ar[r]^\cong & \img~\widetilde O^{(m)} \ar[r] & 0 
}
\label{tg684gg749fh98303}
\end{align}
The isomorphism in the bottom, horizontal row of \eqref{tg684gg749fh98303} asserts: to any $j^m\in \Jc^m_\Xfr$ there will exist some unique $j^{m+1}\in \Jc^{m+1}_\Xfr$ such that $\widetilde O^{(m)}(j^m) = j^{m+1}$. Hence for each $m$ we get the $\Cbb$-linear mapping $\widetilde O^{(m)} : \Jc^m_\Xfr \ra \Jc_\Xfr^{m+1}$. With $q = \dim_-\Xfr$, the composition $O^{(m+1;q+1)} \stackrel{\Delta}{=}\widetilde O^{(q+1)}\circ \widetilde O^{(q-1)}\circ\cdots \widetilde O^{(m+1)}$ maps $\Jc_\Xfr^{m+1} \ra 0$. Hence for any $\ell< m+1$ we obtain the diagram:
\begin{align}
\xymatrix{
0\ar[r] & \Jc^{m+1}_\Xfr \ar[d] \ar[r] &\Jc^\ell_\Xfr \ar[d]  \ar[r] & \Jc^\ell_\Xfr/\Jc_\Xfr^{m+1}\ar[d] \ar[r] & 0
\\
& 0 \ar[r] &\img~O^{(m+1;q+1)} \ar[r] &\img~O^{(m+1;q+1)}\ar[r] & 0
}
\label{rgf784g79h9f8h30}
\end{align}
Specialising to the case $\ell = m$ above, we want to identify $\img~O^{(m+1; q+1)}$ with $\Oc_{\widehat\Xfr}^{\{m\}}$. This will not be the case however since, by \eqref{rh79hf983hf93h0}, for  any $j^m\in \Jc^m$ we have $in~\widetilde O_{m^\p}(j^m) = (m^\p - m)~in(j^m)$. That is, each operator in the composite $O^{(n+1; q+1)}$ will add an integral factor modulo $\Jc_\Xfr^{m+1}$. We illustrate this explicitly below.

\begin{ILL}\label{rf748fh93hf8939}
We will look at the composite $\widetilde O^{(m+2)}\circ \widetilde O^{(m+1)}$.
By $\Cbb$-linearity of the operators involved we have:
\begin{align*}
\widetilde O^{(m+2)}\circ\widetilde O^{(m+1)}
&=
\widetilde O^{(m+2)}\circ \left( {\bf m+1} - \mathrm{ad}_{\e_\Xfr}\right)
\\
&=
\widetilde O^{(m+2)}\circ ({\bf m+1})  - \widetilde O^{(m+2)}\circ \mathrm{ad}_{\e_\Xfr}
\\
&=
{\bf (m+2)(m+1)} - ({\bf m+2})\circ \mathrm{ad}_{\e_\Xfr} + \mathrm{ad}_{\e_\Xfr}\circ \mathrm{ad}_{\e_\Xfr}.
\end{align*}
Now from \eqref{rg78gf87f98h3f033} and using that $\mathrm{ad}_{\e_\Xfr}(F) = \e_\Xfr(F)$ for any $F\in \Oc_\Xfr$, we can deduce:
\[
in\big(\mathrm{ad}_{\e_\Xfr}\circ \mathrm{ad}_{\e_\Xfr}(F)\big) = |in(F)|^2~in(F)
\]
For $j^m\in \Jc_\Xfr^m$ then we find:
\begin{align}
in\big(\widetilde O^{(m+2)}\circ\widetilde O^{(m+1)}\big(j^m)\big)
&=
\left(
(m+2)(m+1) 
-
(m+2)m
+ m^2
\right) in(j^m)
\notag
\\
&=
(m^2 + 2m + 2)~in(j^m)
\label{rfg874g874h9809jf39}
\end{align}
Now define $Q^{(m+1;m+2)} = \frac{1}{m^2 + 2m + 2}\widetilde O^{(m+2)}\circ\widetilde O^{(m+1)}$. By our calculation in \eqref{rfg874g874h9809jf39} we know that $in(Q^{(m+1; m+2)}(j^m)) = in(j^m)$. Hence we have the diagram:
\begin{align}
\xymatrix{
0\ar[r] & \Jc_\Xfr^{m+1} \ar[d] \ar[r] &\Jc^m_\Xfr \ar[d]  \ar[r] & \Oc_{\widehat\Xfr}^{\{m\}}\ar@{=}[d] \ar[r] & 0
\\
& \Jc_\Xfr^{m+3} \ar[r] &\img~Q^{(m+1;m+2)} \ar[r] & \Oc_{\widehat\Xfr}^{\{m\}}\ar[r] & 0
}
\end{align}
In this way we an view $Q^{(m+1;m+2)}$ as a splitting operator on $\Jc^m_\Xfr$ `modulo 2'.
\end{ILL}

\noindent
Following Example \ref{rf748fh93hf8939}, we aim to construct `modulo-$(q+1)$' and `modulo $q$' splitting operators on $\Oc_\Xfr$ and $\Oc_\Xfr[\![1]\!] = \Jc_\Xfr$ respectively. As remarked earlier, Lemma \ref{e8y93hf8h3f93}\emph{(iv)} will guarantee that such operators will give our desired splitting of $\Oc_\Xfr$. 
We begin by noting the following algebraic expansions:
\begin{align}
\prod_{j= 1}^q (j - x)
=
q! + (-1)(q-1)!~x  + \cdots + (-1)^q x^q
\label{rhf783gf73hf803f09jf3}
\end{align}
And
\begin{align}
\prod_{j=2}^q (j - x)
=
q! + (-1) (q-1)! ~x + \cdots + (-1)^{q-1}x^{q-1}.
\label{djnkdjbvkdbkvbdk}
\end{align}
The formal intermediate $x$ represents the operator $\mathrm{ad}_{\e_{\Xfr}}$ and, by \eqref{rg78gf87f98h3f033}, will satisfy:
\begin{align}
in(x^k F) = |in(F)|^k~in(F)
\label{rgf7hf983h030}
\end{align}
for all $F\in \Oc_\Xfr$. Promoting the left-hand sides in \eqref{rhf783gf73hf803f09jf3} and \eqref{djnkdjbvkdbkvbdk} to operators $\widetilde O^{(1; q)}$ and $\widetilde O^{(2;q)}$ on $\Oc_\Xfr$ respectively, see that they will send $\Jc_\Xfr$ resp. $\Jc^2_\Xfr$ to zero. In writing $j^m$ now we will mean an element of $\Jc^m_\Xfr$ so that $in(j^m)\in\Oc_{\widehat\Xfr}^{\{m\}}$.
From the expansion on right hand sides of \eqref{rhf783gf73hf803f09jf3} resp. \eqref{djnkdjbvkdbkvbdk} in addition to \eqref{rgf7hf983h030}, we have:
\begin{align*}
in\big(\widetilde O^{(1;q)}(j^0)\big)
&= 
q! ~in(j^0)~\mbox{and};
\\
in\big(\widetilde O^{(2;q)}(j^1)\big)
&=
\left(q! - (q-1)! + (q-2)! + \cdots + (-1)^{q-1}\right)~ in(j^1)
\\
&=
\sum_{k = 0}^{q-1}(-1)^k (q-k)!~in(j^1).
\end{align*}
Accordingly, set:
\begin{align*}
O^{(1)}
\stackrel{\Delta}{=}
\frac{1}{q!}\widetilde O^{(1;q)}
&&
\mbox{and}
&&
O^{(2)}
\stackrel{\Delta}{=}
\frac{1}{\sum_{k = 0}^{q-1}(-1)^k (q-k)!}\widetilde O^{(2;q)}.
\end{align*}
Then $O^{(1)}$ and $O^{(2)}$ above will define splitting operators for the sequences in Lemma \ref{e8y93hf8h3f93}\emph{(iv)} respectively and hence define a splitting of $\Xfr$. This completes the proof of Proposition \ref{yrgf8g47gh498jf093}.
\qed


\subsection{Further Commentary on Splittings}
While we have not equated diffoeomorphisms, defined in Definition \ref{rfg8gf7hf983hf03}, with isomorphisms, defined independently in Definition \ref{rfj89hg894hf0j3f09j390}; we have nevertheless established in Proposition \ref{yrgf8g47gh498jf093} that for any supermanifold $\Xfr$ with split model $\widehat\Xfr$, 
\[
\mbox{Diffeo.}(\Xfr, \widehat\Xfr) \cong \mbox{Splittings}(\Xfr)
\stackrel{\Delta}{=} 
\mbox{Isom.}(\Xfr, \widehat\Xfr). 
\]
Hence, at least for $\Xfr = \widehat\Xfr$ we have established an equivalence of definitions, i.e., that Diffeo.$(\widehat\Xfr, \widehat\Xfr)\cong$ Isom.$(\widehat\Xfr, \widehat\Xfr)$. Concerning the proof of Proposition \ref{yrgf8g47gh498jf093}, note that we only required the existence of a splitting of the tangent sequence $T_\Xfr$ in Lemma \ref{rhf78f9hf083j09fj30} in degree zero. That is, we only made use of a splitting $spl.$ of the sequence $0\ra T_\Xfr^{(1)}\ra T_\Xfr^{(0)}\ra T_{\widehat\Xfr}^{\{0\}}\ra0$. Furthermore, the splitting $spl.$ was only used to lift the Euler vector field $\e_{\widehat\Xfr}$ as in \eqref{vhfgbjdgdjbhfbd} to some global vector field $\e_\Xfr$ on $\Xfr$. And so we come now to our main observation here being: in the proof of Proposition \ref{yrgf8g47gh498jf093}, all that was important was the initial form formula in \eqref{rg78gf87f98h3f033}. Hence, the same proof given for Proposition \ref{yrgf8g47gh498jf093} will also yield the following:

\begin{THM}\label{rhf9hf983fhfh03}
Let $H$ be a global vector field on $\Xfr$ with initial form the Euler vector field $\e_{\widehat\Xfr}$. Then $H$ will define a splitting of $\Xfr$.\qed
\end{THM}

\noindent
A classical result in supergeometry is Batchelor's theorem, originally appearing in \cite{BAT}, which asserts that any smooth supermanifold splits. We can recover this result from Theorem \ref{rhf9hf983fhfh03} as follows.

\begin{COR}\label{fjvkbvirbvuneinpe}
(Batchelor's Theorem)
Any smooth supermanfold is split.
\end{COR}

\begin{proof}
If $\Xfr$ is smooth, then its tangent sheaf is fine. This means $H^j(X, T_\Xfr^{(\ell)}) = (0)$ for all $\ell$ and all $j>0$. Now from the initial form relation between $T_\Xfr$ and $T_{\widehat\Xfr}$ in Lemma \ref{rhf78f9hf083j09fj30}, note that in degree zero\footnote{c.f., \eqref{rhf794f9h30f3}} we have the following  sequence on cohomology:
\begin{align}
\xymatrix{
\cdots \ar[r] & H^0(X, T_\Xfr^{(0)}) \ar[r] & H^0(X, T_{\widehat\Xfr}^{(0)})\ar[r] & H^1(X, T_{\Xfr}^{(1)}) = (0).}
\label{rfg678gf873hf98303}
\end{align}
Hence $H^0(X, T_\Xfr^{(0)})$ surjects onto $H^0(X, T_{\widehat\Xfr}^{(0)})$ and so there will necessarily exist at least one global vector field on $\Xfr$ with the Euler vector field as its initial form. This vector field will then split $\Xfr$ by Theorem \ref{rhf9hf983fhfh03}. 
\end{proof}

\noindent
Splittings of a smooth supermanifold $\Xfr$ are, in general, `non-canonical'. That is, they need not exist {uniquely}, and so there may be many different splitting maps. We can interpret this statement to mean, by surjectivity in \eqref{rfg678gf873hf98303}, that there may exist many non-identical, global vector fields $H$ on $\Xfr$ with $\e_{\widehat\Xfr}$ as their initial form (and thereby which split $\Xfr$). In the smooth setting then, Koszul's splitting theorem (Theorem \ref{buie9h08d3jd33f3}) can be taken to guarantee the existence of a \emph{unique}, smooth splitting. In the complex analytic setting, splittings need not exist however since globally defined, holomorphic vector fields on arbitrary supermanifolds need not exist. As such, Koszul's theorem amounts to a statement about both existence and uniqueness of splitting maps. 



\section{Proof of Koszul's Theorem}
\label{8g7fg398fh380fkk3}

\noindent
We begin by tying up a loose end leftover at the end of \S\ref{rhfiugf7g8fh003}, being Theorem \ref{djvevcyibeon}   concerning the relation between splittings and the Euler differential.

\subsection{Proof of Theorem $\ref{djvevcyibeon}$}\label{rh84hg084g9j4344rf4}
Let $\Xfr$ be a superanifold with Euler differential $\dt\e_{\widehat\Xfr}$. As observed in the paragraph preceeding the statement of Theorem \ref{djvevcyibeon}, the Euler differential measures precisely the failure for the Euler vector field to lift to $\Xfr$. Now, if $\dt\e_\Xfr \neq0$, then $\Xfr$ cannot be split for the reason that if it were split, then there would exist a lift of $\e_{\widehat\Xfr}$ to $\Xfr$, therefore implying $\dt\e_\Xfr = 0$ and contradicting our assumption. If we assume $\dt\e_\Xfr =0$ now, then there will exist some lift of $\e_{\widehat\Xfr}$ to $\Xfr$, i.e., that there exists some global vector field $H$ on $\Xfr$ with initial form $\e_{\widehat\Xfr}$. We can now use Theorem \ref{rhf9hf983fhfh03} to conclude $\Xfr$ will split. We have therefore established the truth of the following the statements:
\begin{align*}
\dt\e_{\widehat\Xfr} = 0\implies \mbox{$\Xfr$ is split}
&&
\mbox{and}
&&
\dt\e_{\widehat\Xfr}\neq0\implies  \mbox{$\Xfr$ is non-split}
\end{align*}
whence Theorem \ref{djvevcyibeon} follows.\footnote{To more clearly see why Theorem \ref{djvevcyibeon} follows, note from first-order logic: 
\begin{align*}\big((P\implies Q)~\mbox{\emph{and}}~(\neg P\implies \neg Q)\big) \iff\big(P\iff Q\big).
\end{align*}
Apply this to: $P = (\dt\e_{\widehat\Xfr} = 0)$ and $Q = (\mbox{$\Xfr$ is split})$.}
\qed

\subsection{Proof Sketch}
In the same vein as Koszul's proof in \cite{KOSZUL}, ours will also proceed by induction. To describe the inductive step it will be useful to introduce an intermediate notion of the Euler vector field lift.

\subsubsection{The Euler Differential Lift}
Recall that the tangent sheaf $T_\Xfr$ is filtered with strata $(T^{(m)}_\Xfr)_{m\in \Zbb}$ defined as in \eqref{rf74g74hf9380f390}. This filtration is finite with length $\dim_-\Xfr+1$, i.e., that $T^{(m)}_\Xfr = (0)$ for all $m< -1$ and $m> \dim_-\Xfr$. With $T_\Xfr^{(-1)} = T_\Xfr$ the filtration is descending so that $T_\Xfr = T^{(-1)}_\Xfr \supset T^{(0)}_\Xfr\supset T^{(1)}_\Xfr \supset\cdots\supset T^{(q)}_\Xfr\supset 0$. 
Now by definition of the Euler differential in Definition \ref{hf78gf793hf83hf03j}, it is valued in $H^1(X, T^{(1)}_\Xfr)$. Its lifts are defined as follows.

\begin{DEF}\label{rhf794hf98h4f8030}
\emph{Let $\Xfr$ be a supermanifold with Euler differential $\dt\e_{\widehat\Xfr}$. We say that this differential \emph{admits an $\ell$-lift}, for $\ell> 0$, if there exists some $\om\in H^1(X,T^{(\ell+1)}_\Xfr)$ such that $\om^*\mapsto \dt\e_{\widehat\Xfr}$ under the induced map\footnote{this mapping is induced by the inclusion $T_\Xfr^{(\ell)}\subset T_\Xfr^{(1)}$ for any $\ell\geq 1$.}
$H^1(X,T^{(\ell+1)}_\Xfr)\ra H^1(X,T^{(1)}_\Xfr)$. }
\end{DEF}

\noindent
In relation to the Euler vector field lift we have:

\begin{PROP}\label{fbyefreofioenrienc}
Let $\Xfr$ be a supermanifold and suppose its Euler differential admits a $\dim_-\Xfr$-lift. Then the Euler vector field will also lift to $\Xfr$. 
\end{PROP}

\begin{proof}
We will use that $T^{(m)}_\Xfr = (0)$ for all $m> \dim_-\Xfr$ and hence that $T^{(\dim_-\Xfr+1)}_\Xfr = (0)$.
Now if $\dt\e_{\widehat\Xfr}$ admits a $\dim_-\Xfr$-lift, there exists $\om\in H^1(X, T_\Xfr^{(\dim_-\Xfr + 1)})$ mapping onto $\dt\e_{\widehat\Xfr}$; but $H^1(X, T_\Xfr^{(\dim_-\Xfr + 1)}) = (0)$ and so $\om = 0$, giving $\dt\e_{\widehat\Xfr} = 0$ which is precisely the condition for $\e_{\widehat\Xfr}$ to lift to $\Xfr$.\footnote{c.f., the paragraph preceding Theorem \ref{djvevcyibeon}.}
\end{proof}

\noindent 
A consequence of Theorem \ref{djvevcyibeon} and Proposition \ref{fbyefreofioenrienc} is the following.

\begin{COR}\label{hf4ohf8hf09j39f3}
A supermanifold $\Xfr$ is split if and only if its Euler differential admits a $\dim_-\Xfr$-lift.\qed
\end{COR}

\subsubsection{Koszul's Theorem Proof Sketch}
We turn now to a proof sketch of Koszul's theorem (Theorem \ref{buie9h08d3jd33f3}). Our strategy is as follows: under the assumption of a global, affine, even connection $\nabla$ on $\Xfr$ we will show: if its Euler differential admits a $1$-lift, then it will admit a $2$-lift. Hence, upon establishing the existence of a $1$-lift \emph{\'a priori}, we can use $\nabla$ to lift $\dt\e_{\widehat\Xfr}$ indefinitely. Koszul's theorem will then follow from Corollary \ref{hf4ohf8hf09j39f3}.

\subsection{The $1$-Lift}
The initial form sequence from Lemma \ref{rhf78f9hf083j09fj30} is exact in each degree and so gives long exact sequences on cohomology. In order to deduce, then, the existence of a $1$-lift of the Euler differential $\dt\e_{\widehat\Xfr}$ of a supermanifold $\Xfr$, it will suffice to show its image in $H^1(X, T_{\widehat\Xfr}^{\{1\}})$ vanishes. For clarity we present the following composition, extending the diagram in \eqref{rhf794f9h30f3},
\begin{align}
\xymatrix{
0\ar[d]
\\
H^0(\mathcal End_{\Oc_X}T^*_{X, -}) \ar[d]^{i_*} & & 
\\
\ar@{-->}[drr]H^0(T_{\widehat\Xfr}^{(0)})\ar[rr]^\dt & & H^1(T_\Xfr^{(1)})\ar[d]^{p_*}
\\
& & H^1(T_{\widehat\Xfr}^{\{1\}})
}
\label{rfg78gf73hf98h30f303}
\end{align}
where $p_*$ here is induced from $0 \ra T_\Xfr^{(2)} \ra T_\Xfr^{(1)}\stackrel{p}{\ra} T_{\widehat\Xfr}^{\{1\}}\ra0$.
The dashed arrow above interpolates between cohomologies of sheaves of differing weights (i.e., from even to odd) and so ought to vanish. This is the subject of what follows.


\begin{PROP}\label{fh7949v7h49f80fh}
For any $m$ the composition of maps on cohomology induced from the initial form sequence in Lemma \ref{rhf78f9hf083j09fj30}, represented below by the dashed arrow:
\begin{align}
\xymatrix{
H^0(T_{\widehat\Xfr}^{\{m\}})\ar@{-->}[drr] \ar[rr] & & H^{1}(T_\Xfr^{(m+1)})\ar[d] 
\\
& & H^{1}(T_{\widehat\Xfr}^{\{m+1\}})
}
\label{uhfih8fhhf80}
\end{align}
generalising \eqref{rfg78gf73hf98h30f303}, vanishes.
\end{PROP}

\begin{proof}
Recall the initial form sequence from Lemma \ref{rhf78f9hf083j09fj30}. This is a sequence of even morphisms so, with the decomposition $T_\Xfr \cong T_{\Xfr, +}\oplus T_{\Xfr, -}$, we have for each $m$ an exact sequence $T_{\Xfr, p(m)}^{(m+2)} \ra T_{\Xfr, p(m)}^{(m)} \ra T^{\{m\}}_{\widehat\Xfr}$, where $p(m)\in \{+, -\}$ denotes the parity of $m$. This reveals: for any global section $\nu\in H^0(X, T^{\{m\}}_{\widehat\Xfr})$ and with respect to any covering $(\Uc_\al)_\al$, we can always represent $\nu$ on intersections $\Uc_\al\cap\Uc_\be$ by the difference $\nu|_{\Uc_\al\cap\Uc_\be} = \nu_\be- \nu_\al\in T^{(m+2)}_{\widehat\Xfr}(\Uc_\al\cap\Uc_\be)$. In particular, this difference vanishes upon projection onto $T^{\{m+1\}}_{\widehat\Xfr}(\Uc_\al\cap\Uc_\be)$. Hence, on cohomology, it follows that the arrow $H^0(X,T^{\{m\}}_{\widehat\Xfr})\ra H^1(X, T^{\{m+1\}}_{\widehat\Xfr})$ necessarily vanishes.
\end{proof}

\begin{REM}\label{rfh9f98h39fh3fh30}
\emph{Note that Proposition \ref{fh7949v7h49f80fh} does not necessarily imply there do not exist odd morphisms $T^{\{m\}}_{\widehat\Xfr} \ra T^{\{m+1\}}_{\widehat\Xfr}$; only that the composition in \eqref{uhfih8fhhf80} vanishes. Indeed, in Appendix \ref{rgf7g7fg93gf083h03}, it will be clear that non-trivial, odd morphisms indeed exist.}
\end{REM}

\noindent 
Specialising Proposition \ref{fh7949v7h49f80fh} to $m = 0$ now reveals what we originally wanted:

\begin{COR}\label{tg74gf7h8f3j093}
For any supemanifold $\Xfr$ there exists a $1$-lift of its Euler differential, i.e., that
$p_*\dt \e_{\widehat\Xfr} =0$.\qed
\end{COR}

\subsection{Partial Liftings of the Euler Vector Field}
Recall from Theorem \ref{rhf9hf983fhfh03} that a global vector field on $\Xfr$ with $\e_{\widehat\Xfr}$ as its initial form will define a splitting of $\Xfr$. The Euler differential measures the failure for the Euler vector field to lift and thereby also relates information on splitting, albeit more indirectly, in Theorem \ref{djvevcyibeon}. In Proposition \ref{fbyefreofioenrienc} this relation to the Euler vector field lift was clarified. Presently, we look to study a notion of partial liftings of the Euler vector field.

\begin{PROP}\label{rfh874gf94hf8j30}
On any supermanifold $\Xfr$ there exists a smooth, global vector field $H^\8$ such that, modulo $T_\Xfr^{(2)}$, it is holomorphic and maps onto the Euler vector field $\e_{\widehat\Xfr}$.
\end{PROP}

\begin{proof}
Note that we have the following nine-term lattice:
\begin{align}
\xymatrix{
& 0\ar[d] & 0\ar[d]& 0\ar[d] &
\\
0\ar[r] & T_\Xfr^{(2)}\ar@{=}[r] \ar[d]&T_\Xfr^{(2)}\ar[d] \ar[r] & T_\Xfr^{(1)}\ar[d]\ar[r]&0
\\
0 \ar[r] & T_\Xfr^{(1)}\ar[d] \ar[r] & T_\Xfr^{(0)} \ar[d] \ar[r] & T_\Xfr^{(0)}\ar[r]\ar[d] & 0
\\
0 \ar[r] & T_{\widehat \Xfr}^{\{1\}} \ar[r] \ar[d]&  T_\Xfr^{(0)}/T_\Xfr^{(2)}\ar[r] \ar[d] & T_{\widehat \Xfr}^{\{0\}}\ar[r] \ar[d]& 0
\\
&0 & 0& 0& 
}
\label{gf74gf78gf93hf89h03}
\end{align}
The right-most column and the bottom row are exact and induce the following maps between cohomologies:
\begin{align}
\xymatrix{ 
&\ar@{=}[d]  H^0(T_{\widehat\Xfr}^{\{0\}})\ar[r]^\dt & H^1(T_\Xfr^{(1)})\ar[d]^{p_*}
\\
H^0(T_\Xfr^{(0)}/T_\Xfr^{(2)}) \ar[r] & H^0(T_{\widehat\Xfr}^{\{0\}})\ar[r] & H^1(T_{\widehat\Xfr}^{\{1\}}) 
}
\label{rg73gf983h8fh3009j3333}
\end{align}
where $\dt$ and $p_*$ are the maps from the diagram in \eqref{rfg78gf73hf98h30f303}. By Corollary \ref{tg74gf7h8f3j093}, the image of $\e_{\widehat\Xfr}$ in $H^1(X, T_{\widehat\Xfr}^{\{1\}})$ vanishes. Hence by exactness of the bottom row in \eqref{rg73gf983h8fh3009j3333}, there exists a global vector field $\overline H\in H^0(X, T_\Xfr^{(0)}/T_\Xfr^{(2)})$ mapping onto the Euler vector field $\e_{\widehat\Xfr}$. To complete the proof, observe that the middle column in \eqref{gf74gf78gf93hf89h03} is exact and gives on cohomology:
\begin{align*}
\xymatrix{
\cdots \ar[r] & H^0(T_\Xfr^{(0)})\ar[r] & H^0(T_\Xfr^{(0)}/T_\Xfr^{(2)}) \ar[r] & H^1(T_{\Xfr}^{(2)}) \ar[r] & \cdots
}
\end{align*}
We label by an $`\8$' superscript the sheaves of smooth objects---e.g., $T^\8_\Xfr$ denotes the sheaf of smooth vector fields as opposed to $T_\Xfr$ which denotes holomorphic vector fields if $\Xfr$ is assumed to be complex.\footnote{\label{rfg64gfg8f73}As for complex manifolds, there exists a Dolbeault-like operator $\widetilde\pt$ on complex supermanifolds. A smooth function is holomorphic iff $\widetilde\pt f = 0$.}
As holomorphy is a stronger condition than smoothness, we have $T_\Xfr\subset T_\Xfr^\8$. Any sheaf of \emph{smooth} vector fields is fine and so $H^1(X, T_{\widehat\Xfr}^{(2);\8}) = (0)$. Therefore, there will always exist some smooth, global vector field $H^\8\in H^0(X, T_\Xfr^{(0);\8})$ mapping onto any global section in $H^0(X, T_\Xfr^{(0)}/T_\Xfr^{(2)})\subset H^0(X, T^{(0);\8}_\Xfr/T^{(2);\8}_\Xfr)$. We have therefore argued the existence of a smooth, global vector field $H^\8$ on $\Xfr$ such that, modulo $T_\Xfr^{(2);\8}$, it is holomorphic and maps onto $\e_{\widehat\Xfr}$. This completes the proof.
\end{proof}

\noindent
As with the $\ell$-lifts of the Euler differential from Definition \ref{rhf794hf98h4f8030}, Proposition \ref{rfh874gf94hf8j30} is suggestive of a similar notion for the Euler vector field itself. To give this notion recall that we can identify $T_{\widehat\Xfr}^{\{0\}} \cong T_\Xfr^{(0)}/T_\Xfr^{(1)}$ by Lemma \ref{rhf78f9hf083j09fj30}. Hence, from the filtration on $T_\Xfr$ we have more generally a surjection $T_\Xfr^{(0)}/T_\Xfr^{(\ell)} \stackrel{p_\ell}{\ra} T_{\widehat\Xfr}^{\{0\}}\ra 0$ for any $\ell \geq 1$. Upon recalling that $\e_{\widehat\Xfr}\in H^0(X, T^{\{0\}}_{\widehat\Xfr})$ by Definition \ref{djkbkcvjfvjcekjlenklw} we arrive at the following.

\begin{DEF}\label{rh8794g984h0fj09}
\emph{We say the Euler vector field on $\widehat\Xfr$ will admit a \emph{mod $\ell$ lift} to $\Xfr$ if and only if there exists a global section $\overline H\in H^0(X, T_\Xfr^{(0)}/T_\Xfr^{(\ell)})$ which maps onto $\e_{\widehat\Xfr}$ via $H^0(X, T_\Xfr^{(0)}/T_\Xfr^{(\ell)})\stackrel{p_{\ell *}}{\ra} H^0(X, T^{\{0\}}_{\widehat\Xfr})$.}
\end{DEF}

\begin{DEF}\label{uhf93hf89h30fh30}
\emph{Any smooth, global vector field $H^\8\in H^0(X, T_\Xfr^{(0);\8})$ which 
defines a mod $\ell$ lift $\overline H$ will be referred to as a \emph{smooth extension} of $\overline H$.}
\end{DEF}

\noindent
In the language introduced in Definition \ref{rh8794g984h0fj09} and \ref{uhf93hf89h30fh30} then, Proposition \ref{rfh874gf94hf8j30} asserts firstly that the Euler vector field of a split model $\widehat\Xfr$ will admit a mod $2$ lift to \emph{any} supermanifold $\Xfr$ to which $\widehat\Xfr$ is associated; and secondly that this lift furthermore admits a smooth extension. Similarly to Corollary \ref{hf4ohf8hf09j39f3} we now have:
%
\begin{COR}\label{rf783gf873g9f7h38hf03}
Any supermanifold $\Xfr$ which admits a mod $\ell$ lift of the Euler vector field with $\ell > \dim_-\Xfr$ is split.
\end{COR}

\begin{proof}
This follows from Theorem \ref{rhf9hf983fhfh03} upon observing that $T_\Xfr^{(m)} = (0)$ for $m> \dim_-\Xfr$. 
\end{proof}

\subsection{Shear-Like Transformations}
Central to our proof of Koszul's theorem will be the construction of certain transformations on local sections of the tangent sheaf reminiscent of shearing.

\subsubsection{Local Coordinates}

\noindent 
Despite our efforts, we are unable to give a completely coordinate-free proof of Koszul's theorem. For the reader interested in learning about local coordinate formulations on supermanifolds, we refer them to the classical texts \cite{LEI, YMAN}. See also \cite[\S5]{BETTPHD} where Koszul's theorem is studied largely from a local coordinate perspective. We digress here briefly in order to establish notation. Let $\Xfr$ be a supermanifold with tangent sheaf $T_\Xfr$. Recall that it is filtered according to powers of the fermionic ideal $\Jc_\Xfr^m$ as in \eqref{rf74g74hf9380f390}. Let $\Uc\subset \Xfr$ be a local coordinate neighbourhood. Note that $\Uc$ is itself a supermanifold and is split so we can regard it as an open set in $\widehat\Xfr$ also. We denote by $|\Uc|\subset X$ its reduced space. A splitting of $\Uc$ is furnished by choice of local coordinates $(x|\q)$. Hence we can regard $\Oc_\Xfr(\Uc)$ as an $\Oc_X(|\Uc|)$-module. The variables $x$ and $\q$ are even and odd respectively; $(\q)$ generates the fermionic ideal $\Jc_\Xfr(\Uc)$ over $\Oc_X(|\Uc|)$ and so the $m$-th powers $(\q^m)$ generate $\Jc_\Xfr^m(\Uc)$ over $\Oc_X(|\Uc|)$. Where the tangent sheaf is concerned, we have locally an isomorphism\footnote{c.f., footnote \ref{rg7f4gf7g9f8hf03}.}
\begin{align}
T^{(m)}_\Xfr(\Uc) \cong\bigoplus_{\ell\geq m} T^{\{\ell\}}_{\widehat\Xfr}(|\Uc|). 
\label{efh74gf793f98h3f3}
\end{align}
Now the sequence for the sheaf $T_{\widehat\Xfr}^{(\ell)}$ in Lemma \ref{rhf79gf983hf80h03} will split over $\Uc$  so therefore, in coordinates $(x|\q)$, $T^{\{\ell\}}_{\widehat\Xfr}(\Uc)$ is generated as an $\Oc_X(|\Uc|)$-module by the sections $\q^\ell \pt/\pt x$ and $\q^{\ell+1}\pt/\pt\q$ respectively.\footnote{In this way, with $T_{X, -} = T_{\widehat\Xfr}^{\{-1\}}$ we can see the motivation behind referring to $T^*_{X, -}$ as the bundle of odd cotangents as in Definition \ref{rfh9784fh984hf89h8h30}.} Accordingly, from the isomorphism in \eqref{efh74gf793f98h3f3} sections of $T^{(m)}_\Xfr(\Uc)$ are generated by tuples,
\begin{align}
T^{(m)}_\Xfr(\Uc)
= 
\left(
\q^m\frac{\pt}{\pt x},
\q^{m+1}\frac{\pt}{\pt\q},
\q^{m+1}\frac{\pt}{\pt x},
\q^{m+2}\frac{\pt}{\pt \q}, 
\ldots
\right)
\label{g783gf97hf98hf83}
\end{align}
over $\Oc_X(|\Uc|)$. The global $\Zbb_2$-grading of the functions on $\Xfr$ induce a $\Zbb_2$-grading $T_\Xfr\cong T_{\Xfr,+}\oplus T_{\Xfr, -}$, decomposing sections into their even and odd constituents. Over $\Uc$ then and in the notation \eqref{g783gf97hf98hf83} we have over $\Oc_X(|\Uc|)$ the even strata:
\begin{align}
T^{(2m)}_{\Xfr, +}(\Uc)
&= 
\left(
\q^{2m}\frac{\pt}{\pt x}, \q^{2m+1}\frac{\pt}{\pt \q}, \q^{2m+2}\frac{\pt}{\pt x},
\ldots
\right)~\mbox{and};
\label{fjkdbfhdbhjdb}
\\
T^{(2m)}_{\Xfr, -}(\Uc)
&=
\left(
\q^{2m+1}\frac{\pt}{\pt x}, \q^{2m+2}\frac{\pt}{\pt \q}, \q^{2m+3}\frac{\pt}{\pt x},
\ldots
\right).
\label{uiegfghf389fh893hf3}
\end{align}
Regarding the odd strata we have $T_{\Xfr, +}^{(2m+1)} = T_{\Xfr, +}^{(2m+2)}$ and similarly, $T_{\Xfr, -}^{(2m+1)} = T_{\Xfr, -}^{(2m)}$.

\subsubsection{Algebraic Sections}
Let $\nabla$ be a global, even, affine connection on $\Xfr$. With $\nabla$ we will show that the boundary of any mod $2$ lift, $\pt\overline H$, will vanish. Before presenting the relevant calculations we digress to present the following useful definition concerning tangent vectors.\footnote{The terms introduced in Definition \ref{rh73f98389fh03jf3} is adapted from a similar usage in \cite[\S8]{NATDG}.}

\begin{DEF}\label{rh73f98389fh03jf3}
\emph{Let $v\in T_\Xfr$ be a section with initial form $in(v)\in T_{\widehat\Xfr}^{(m)}$. We will refer to this section as \emph{algebraic} if its initial form vanishes modulo $\Jc_{\widehat\Xfr}^{m+1}$. 
}
\end{DEF}

\noindent
Note, in terms of the sequence in Lemma \ref{rhf79gf983hf80h03}, the condition for a section $v\in T^{(m)}_\Xfr$ to be algebraic is for the projection of its initial form under $T^{\{m\}}_\Xfr\ra \wedge_{\Oc_X}^mT^*_{X, -}\otimes_{\Oc_X} T_X$ to vanish, in which case it will be in the image of $\wedge_{\Oc_X}^{m+1}T^*_{X, -}\otimes_{\Oc_X}T_{X, -}$. As a result we have the following.

\begin{LEM}\label{rfh793f97hf893h08f309}
Any mod $\ell$ lift of the Euler vector field is algebraic.
\end{LEM}

\begin{proof}
Recall from Definition \ref{rh8794g984h0fj09} that any mod $\ell$ lift $\overline H$ will have the Euler vector field $\e_{\widehat\Xfr}$ as its initial form and so gives section in $T_\Xfr^{(0)}$. From Definition \ref{djkbkcvjfvjcekjlenklw} the Euler vector field will be a section of $T_{\widehat\Xfr}^{\{0\}}$. It is in the image of $\mathcal End_{\Oc_X}T^*_{X, -}$ and so, upon inspection of the sequence in Lemma \ref{rhf79gf983hf80h03} specialised to $m = 0$, its projection onto $T_X$ vanishes. Hence it vanishes modulo $\Jc_{\widehat\Xfr}$ and so, by Definition \ref{rh73f98389fh03jf3}, the mod $\ell$ lift $\overline H$ will be algebraic .
\end{proof}

\subsubsection{Shearing}
Let $H$ be a global vector field on $\Xfr$ with initial form $\e_{\widehat\Xfr}$. Note that it need not be holomorphic. With a global, affine connection $\nabla$ we can define an $\Oc_\Xfr$-morphism $\nabla H : T_\Xfr \ra T_\Xfr$ given by $v\mapsto (\nabla H)(v) \stackrel{\Delta}{=} \nabla_vH$. 

\begin{PROP}\label{rhf893f98h3f80h30}
Let $\nabla$ be even. Then for any algebraic section $v\in T_\Xfr^{(m)}$,
\[
(\nabla H)(v) \equiv v \mod T_\Xfr^{(m+1)}
\]
where $H$ is a global vector field with initial form $\e_{\widehat\Xfr}$.
\end{PROP}

\begin{proof}
With respect to a covering $(\Uc_\al)$ of $\Xfr$ and the isomorphism between tangent sheaves in \eqref{efh74gf793f98h3f3}, we can write $H$ locally as follows:
\begin{align}
H|_{\Uc_\al} = \e_\al + Q^{\{1\}}_\al + Q^{\{2\}}_\al + \cdots
\label{rh973h893h0j39}
\end{align}
where $Q^{\{j\}}_\al\in T_{\widehat\Xfr}^{\{j\}}(|U_\al|)$ and $\e_\al = \e_{\widehat\Xfr}|_{\widehat\Uc_\al}$.
\footnote{\label{gd783gd79893hd89hd0}The sections $Q^{(j)}_\al$ are holomorphic. Patching the expressions \eqref{rh973h893h0j39} together over $\Xfr$ with a partition of unity recovers $H$. And while $H$ can locally be expressed as a holomorphic section, the resultant need not longer be holomorphic since the partition of unity are not holomorphic functions.}
Let $(x|\q)$ denote coordinates on $\Uc_\al$. Then for any section $v\in T_\Xfr$ note that locally we can always project onto the algebraic component by:
\begin{align}
v\mapsto v(\q)\frac{\pt}{\pt \q}.
\label{fhgf974hf98h80f309}
\end{align}
This merely emphases that the sequence in Lemma \ref{rhf79gf983hf80h03} splits over any coordinate neighbourhood. Now with the local expression for the Euler vector field from \eqref{g683f863g793893} observe that for any $v\in T_\Xfr^{(m)}$,
\begin{align}
(\nabla H)(v)|_{\Uc_\al} &\stackrel{\Delta}{=} \nabla_v H|_{\Uc_\al}
\notag
\\
&=
\nabla_v ( \e_\al + Q^{\{1\}}_\al + Q^{\{2\}}_\al + \cdots)
&&\mbox{(from \eqref{rh973h893h0j39})}
\notag
\\
&\equiv
\nabla_v \e_\al \mod T^{(m+1)}_\Xfr &&\mbox{(c.f., Definition \ref{gf7498fh3f30jf39f0}\emph{(ii)})}
\notag
\\
&=
v(\q)\frac{\pt}{\pt \q} + \q\nabla_v \frac{\pt}{\pt\q} &&\mbox{(from \eqref{g683f863g793893})}.
\label{h73gf793f98h39}
\end{align}
It remains to argue $\q\nabla_v \frac{\pt}{\pt\q} \equiv 0$ modulo $T_\Xfr^{(m+1)}(\Uc_\al)$. To see this it will be essential to use the assumption that $\nabla$ is even and $v$ is algebraic. Indeed, from the notation in \eqref{g783gf97hf98hf83}, if $v\in T^{(m)}_\Xfr$ is algebraic then its image in $\wedge_{\Oc_X}^m T^*_{X, -}\otimes_{\Oc_X} T_X$ will vanish. Hence 
\begin{align}
v|_{\Uc_\al}\in \left(\q^{m+1}\frac{\pt}{\pt \q}, \q^{m+1}\frac{\pt}{\pt x}, \ldots\right). 
\label{biurg94fh39jf3efrfrf}
\end{align}
Now recall from Definition \ref{fg74f94hf80303j0} that if $\nabla$ is even, then $\nabla_v\in \mathrm{End}_\Cbb T_\Xfr$ will be an even endomorphism, hence $\nabla_v : T_{\Xfr, \pm}\ra T_{\Xfr, \pm}$. As a  mapping of tensor products over $\Cbb$ we recall:
\begin{align}
T_{\Xfr, +}\otimes_\Cbb T_{\Xfr, +}\stackrel{\nabla}{\ra} T_{\Xfr, +}
&&
T_{\Xfr, +}\otimes_\Cbb T_{\Xfr, -}\stackrel{\nabla}{\ra} T_{\Xfr, -}
&&
\mbox{and}
&&
T_{\Xfr, -}\otimes_\Cbb T_{\Xfr, -}\stackrel{\nabla}{\ra} T_{\Xfr, +}.
\label{rbfuif9h398fh093jewwe}
\end{align}
From the latter most mapping above and the local characterisation of $T_{\Xfr, \pm}$ in \eqref{fjkdbfhdbhjdb} and \eqref{uiegfghf389fh893hf3} we must then have $\nabla_{\pt/\pt\q}(\pt/\pt\q)\in (\pt/\pt x, \q\pt/\pt\q, \q^2\pt/\pt x, \ldots)$. As a result of this in addition to \eqref{biurg94fh39jf3efrfrf} we are therefore led to,
\[
\q \nabla_v\frac{\pt}{\pt\q}
\in 
\left(
\q^{m+2}\frac{\pt}{\pt \q}, 
\q^{m+2}\frac{\pt}{\pt x},
\ldots
\right)
\in T_\Xfr^{(m+1)}(\Uc_\al).
\]
Hence $\q\nabla_v \frac{\pt}{\pt\q} \equiv 0$ modulo $T_\Xfr^{(m+1)}$
and so with the identification of the algebraic component of $v$ in \eqref{fhgf974hf98h80f309} this lemma now follows from \eqref{h73gf793f98h39}. 
\end{proof}

\noindent
Recall that shear transformations of a vector space are linear transformations which fix a subspace and translate its complement parallel-wise. Analogously, the morphism $\nabla H : T_\Xfr \ra T_\Xfr$ from Proposition \ref{rhf893f98h3f80h30} will fix the subsheaf of algebraic sections modulo $T_\Xfr[\![1]\!]$ and so we might view it as a kind of shearing. From Lemma \ref{rfh793f97hf893h08f309} and Proposition \ref{rhf893f98h3f80h30} now we have immediately:

\begin{COR}
For any global, even, affine connection $\nabla$,
\[
\nabla_{H}H \equiv H \mod T_\Xfr^{(1)}
\]
where $H$ is any global vector field with initial form $\e_{\widehat\Xfr}$.
\qed
\end{COR}

\noindent
In the case where $\Xfr = \widehat\Xfr$ and $H = \e_{\widehat\Xfr}$ we have:

\begin{LEM}\label{rfjiohf4hf89hf90hf0}
For any global, even, affine connection $\nabla$ on $\widehat\Xfr$,
\[
\nabla_{\e_{\widehat\Xfr}}\e_{\widehat\Xfr} = \e_{\widehat\Xfr}.
\]
\end{LEM}

\begin{proof}
We will use the coordinate representation $\e_{\widehat\Xfr}|_\Uc = \q_i\pt/\pt \q_i$ from \eqref{g683f863g793893}, where the index $i$ is implicitly summed. This leads to:
\begin{align}
\nabla_{\e_{\widehat\Xfr}|_\Uc}\e_{\widehat\Xfr}|_\Uc 
&= 
\nabla_{\q_i\frac{\pt}{\pt \q_i}}\left(\q_j\frac{\pt}{\pt \q_j}\right)
\notag
\\
&=
\q_i\left(\frac{\pt}{\pt \q_i}\q_j\right)\frac{\pt}{\pt \q_j}
- 
\q_i\q_j\nabla_{\frac{\pt}{\pt \q_i}}\frac{\pt}{\pt \q_j}
&&
\mbox{(by Definition \ref{gf7498fh3f30jf39f0})\emph{(i)}}
\notag
\\
&=
\q_j\frac{\pt}{\pt \q_j}
-
\sum_{i< j}
\q_i\q_j
\nabla_{\frac{\pt}{\pt \q_i}}\frac{\pt}{\pt \q_j}
+
\sum_{i > j}
\q_i\q_j
\nabla_{\frac{\pt}{\pt \q_i}}\frac{\pt}{\pt \q_j}
\notag
\\
&=
\e_{\widehat\Xfr}|_\Uc
-
\q_i\q_j
\left(
\nabla_{\frac{\pt}{\pt \q_i}}\frac{\pt}{\pt \q_j}
+
\nabla_{\frac{\pt}{\pt \q_j}}\frac{\pt}{\pt \q_i}
\right)
\label{rfhhf983hf083f0j390f}
\\
&= 
\e_{\widehat\Xfr}|_\Uc.
\notag
\end{align}
where we used that the latter, bracketed term in \eqref{rfhhf983hf083f0j390f} vanishes identically by Theorem \ref{rburghurhoiejipf} (c.f., \eqref{fguyegf78hf98hf03}).
\end{proof}

\noindent
We return now to an important relation between global, even, affine connections and algebraic sections.

\begin{LEM}\label{rfh794fg9hf80h30f3}
Let $\nabla$ be a global, even, affine connection on $\Xfr$. Then for any global vector field $H$ on $\Xfr$ with initial form $\e_{\widehat\Xfr}$ and any section $v\in T_\Xfr$, the section $\nabla_v H$ will be algebraic.
\end{LEM}

\begin{proof}
If $v$ is itself algebraic then $\nabla_v H$ will be algebraic by Proposition \ref{rhf893f98h3f80h30}. Assume now that $v$ is \emph{not} algebraic and W.L.O.G., suppose $v\in T^{(m)}_\Xfr$. Then locally on a coordinate neighbourhood $\Uc$ with coordinates $(x|\q)$,
\begin{align}
v|_\Uc
\in 
\left(
\q^m\frac{\pt}{\pt x}, \q^{m+1}\frac{\pt}{\pt \q}, \ldots
\right).
\label{rfg8gf73gf93h980}
\end{align}
Now using the local expansion for $H$ as in \eqref{rh973h893h0j39} gives
\begin{align}
(\nabla_v H)|_\Uc
&\equiv
v(\q)\frac{\pt}{\pt\q} + \q\nabla_v\frac{\pt}{\pt\q} \mod T_\Xfr^{(m+1)}.
\label{rhf98hf83f09j3f}
\end{align}
Since $\nabla$ is even we know from \eqref{fjkdbfhdbhjdb}, \eqref{uiegfghf389fh893hf3}
and \eqref{rbfuif9h398fh093jewwe} that $\nabla_{\pt/\pt \q}(\pt/\pt x)\in (\pt/\pt\q, \q\pt/\pt x,\ldots)$. 
Then from \eqref{rfg8gf73gf93h980}, $\q\nabla_v(\pt/\pt\q)\in (\q^{m+1}\pt/\pt\q, \q^{m+2}\pt/\pt x, \ldots)$ and therefore is algebraic. Furthermore, we know that the term $v(\q)\frac{\pt}{\pt\q}$ is algebraic\footnote{c.f., \eqref{fhgf974hf98h80f309}} and so $\nabla_v H$ in \eqref{rhf98hf83f09j3f} is a sum of algebraic vector fields and so is, in particular, itself algebraic.
\end{proof}

\subsubsection{A Shear-Like Transformation}
Any choice of global, affine connection allows for the formulation of the vector field Lie bracket. And so, on a supermanifold $\Xfr$ with global, affine connection $\nabla$ we can form, for any two sections $u, v\in T_\Xfr$, their bracket:
\begin{align}
[u,v]
\stackrel{\Delta}{=}
\nabla_u v - \nabla_v u + tor.^\nabla(u, v)
\label{rfigf7939f8h38fh30}
\end{align}
where $tor^\nabla : T_\Xfr\otimes_{\Oc_\Xfr}T_\Xfr \ra T_\Xfr$ is a homomorphism referred to as the torsion of $\nabla$.

\begin{LEM}\label{fjkvbvhjvbevne}
Let $\nabla$ be even, $H$ a lift of $\e_{\widehat\Xfr}$ and $v$ any algebraic section of $T^{(m)}_\Xfr$. Then $tor^\nabla(H, v)\equiv0\mod T_\Xfr^{(m+1)}$.
\end{LEM}

\begin{proof}
Since $\nabla$ is even we have the data of mappings from \eqref{rbfuif9h398fh093jewwe}. Its torsion must be a morphism in the same way and so will gives a morphism $T_{\Xfr, -}\otimes_{\Oc_\Xfr}T_{\Xfr, -}\ra T_{\Xfr, +}$ whereby,
\begin{align}
tor^\nabla\left(\frac{\pt}{\pt\q}, \frac{\pt}{\pt\q}\right)
\in \left(\frac{\pt}{\pt x}, \q\frac{\pt}{\pt\q}, \ldots\right).
\label{rugfiifh38fh3}
\end{align}
With a lift $H$ of $\e_{\widehat\Xfr}$ and with $v\in T^{(m)}_\Xfr$ algebraic we have the local expression
\begin{align*}
tor^\nabla(H, v)
&\equiv tor^\nabla\left(\q\frac{\pt}{\pt\q},  \q^m\frac{\pt}{\pt\q}\right)\mod T_\Xfr^{(m+1)}
\\
&= \q^{m+1}tor^\nabla\left(\frac{\pt}{\pt\q},  \frac{\pt}{\pt\q}\right)
\\
&\in \left(\q^{m+1}\frac{\pt}{\pt x}, \q^{m+2}\frac{\pt}{\pt\q}, \ldots\right) = T_\Xfr^{(m+1)}
&&
\mbox{(from \eqref{rugfiifh38fh3}).}
\end{align*}
Hence $tor^\nabla(H, v)\equiv 0$ modulo $T_\Xfr^{(m+1)}$, as required.
\end{proof}

\noindent
The results presented so far culminate in the following.

\begin{PROP}\label{rfh89fh983hf80f0j93}
Let $\nabla$ be a global, even, affine connection on $\Xfr$. For any algebraic section $v\in T^{(m)}_\Xfr$ and lift $H$ of $\e_{\widehat\Xfr}$,
\[
\nabla_Hv \equiv (m+1)v\mod T_\Xfr^{(m+1)}.
\]
\end{PROP}

\begin{proof}
This is immediate from the relation to the Lie bracket in \eqref{rfigf7939f8h38fh30}, the scaling property of $\e_{\widehat\Xfr}$ from Lemma \ref{rbyurgf4gf233393hf89h3}\emph{(ii)}, Proposition \ref{rhf893f98h3f80h30} and Lemma \ref{fjkvbvhjvbevne}.
\end{proof}

\noindent
Unlike the shearing described in Proposition \ref{rhf893f98h3f80h30}, the mapping in Proposition \ref{rfh89fh983hf80f0j93} above also dilates the algebraic sections. Hence, relative to the mapping in Proposition \ref{rhf893f98h3f80h30}, we might refer to $\nabla_H$ in Proposition \ref{rfh89fh983hf80f0j93} as \emph{shear-like}. More generally, if $v \in T^{(m)}_\Xfr$ from Proposition \ref{rfh89fh983hf80f0j93} is \emph{not} algebraic, we have:

\begin{PROP}\label{rhf4f9f9j9f39f34f4f4}
Let $\nabla$ be a global, even, affine connection on $\Xfr$. For any non-algebraic section $v\in T^{(m)}_\Xfr$ and lift $H$ of $\e_{\widehat\Xfr}$,
\[
\nabla_Hv\equiv m v\mod T^{(m)}_\Xfr \cap \{\mbox{algebraic sections}\}.
\]
\end{PROP}

\begin{proof}
The proof of this proposition can be given by direct calculation. With $H$ a lift of $\e_{\widehat\Xfr}$ and $(\Uc_\al)$ a covering of $\Xfr$, write $H$ as in \eqref{rh973h893h0j39}. Then we have
\begin{align*}
(\nabla_Hv)|_{\Uc_\al}
&=
\nabla_{\e_\al + Q_\al^{\{1\}} + \cdots} v
\\
&\equiv
\nabla_{\e_\al}v|_{\Uc_\al} \mod T_\Xfr^{(m+1)}(\Uc_\al)
&&\mbox{(by Definition \ref{gf7498fh3f30jf39f0}\emph{(ii)})}.
\end{align*}
In coordinates $(x|\q)$ on $\Uc_\al$ consider the following local expression:
\begin{align*}
v(x|\q)
\equiv
f_m \q^m\frac{\pt}{\pt x} + g_{m+1}\q^{m+1}\frac{\pt}{\pt\q} \mod T_\Xfr^{(m+1)}(\Uc_\al).
\end{align*}
Then we have
\begin{align}
\nabla_{\e_\al}v|_{\Uc_\al}
=&~
\e_\al(f_m\q^m)\frac{\pt}{\pt x} + \e_\al(g_{m+1}\q^{m+1})\frac{\pt}{\pt\q}
\notag
\\
&~+
f_m\q^m\nabla_{\e_\al}\frac{\pt}{\pt x}
+
g_{m+1}\q^{m+1}\nabla_{\e_\al}\frac{\pt}{\pt\q}
\notag
\\
=&~
m v + g_{m+1}\q^{m+1}\frac{\pt}{\pt\q}
&&\mbox{(by Lemma \ref{rbyurgf4gf233393hf89h3}\emph{(i)})}
\label{rfh974gf9hf8h3003}
\\
&~+
f_m\q^m\nabla_{\e_\al}\frac{\pt}{\pt x}
+
g_{m+1}\q^{m+1}\nabla_{\e_\al}\frac{\pt}{\pt\q}.
\label{hbvevibuibvuio4}
\end{align}
With $\e_\al = \q\pt/\pt\q$ and the property that $\nabla$ is even, the latter most term in \eqref{hbvevibuibvuio4}, being $g_{m+1}\q^{m+1}\nabla_{\e_\al}({\pt}/{\pt\q})$, will be an element of $T_\Xfr^{(m+1)}(\Uc_\al)$; while the former term $f_m\q^m\nabla_{\e_\al}({\pt}/{\pt x})$ will be an algebraic section of $T_\Xfr^{(m)}(\Uc_\al)$. Note moreover that the latter term in \eqref{rfh974gf9hf8h3003}, being $g_{m+1}\q^{m+1}\pt/\pt\q$, is also an algebraic section of $T^{(m)}_\Xfr(\Uc_\al)$. Hence, up to algebraic sections of $T_\Xfr^{(m)}(\Uc_\al)$ we can equate $\nabla_Hv$ with $m v$. This proposition now follows.
\end{proof}

\noindent
We will now turn to our proof of Koszul's theorem.

\subsection{Proof of Theorem $\ref{buie9h08d3jd33f3}$}
Let $\Xfr$ be our supermanifold and suppose it is equipped with a mod $\ell$ lift of the Euler vector field $\overline H$. Denote by $H^\8$ a smooth extension of $\overline H$ to global vector field on $\Xfr$. The argument for why such an extension will always exist follows along very similar lines to that in Proposition \ref{rfh874gf94hf8j30} for the mod $2$ lift. In brief: firstly note the existence of a short exact sequence $T_\Xfr^{(\ell)} \ra T^{(0)}_\Xfr \ra T^{(0)}_\Xfr/T^{(\ell)}_\Xfr$ for any $\ell$ giving the boundary mapping $\pt: H^0(X, T^{(0)}_\Xfr/T^{(\ell)}_\Xfr) \ra H^1(X, T^{(\ell)}_\Xfr)$. Recall for the mod $\ell$ lift $\overline H$ that $\overline H\in H^0(X, T^{(0)}_\Xfr/T^{(\ell)}_\Xfr)$.
Where smooth sections are concerned we know that $\forall j>0,H^j(X, T^{(\ell);\8}_\Xfr)= (0)$; and so we have $H^0(X, T^{(0);\8}_\Xfr) \ra H^0(X, T^{(0);\8}_\Xfr/T^{(\ell);\8}_\Xfr) \stackrel{\pt}{\ra} H^1(X, T^{(\ell)}_\Xfr) = (0)$. From the inclusion $H^0(X, T^{(0)}_\Xfr/T^{(\ell)}_\Xfr)\subset H^0(X, T^{(0);\8}_\Xfr/T^{(\ell);\8}_\Xfr)$ see that $H^0(X, T^{(0);\8}_\Xfr)$ will surject onto $H^0(X, T^{(0)}_\Xfr/T^{(\ell)}_\Xfr)$. We can therefore conclude the existence of some smooth extension $H^\8$ of our given, mod $\ell$ lift.  Now let $(\Uc_\al)$ be an open covering of $\Xfr$. Over $\Uc_\al$ we can generically write 
\begin{align}
H^\8|_{\Uc_\al} &= \e_\al + Q_\al 
\label{rfh8949f83hf0j30}
\\
&=
\e_\al + Q_\al^{\{1\}} + Q_\al^{\{2\}} + \cdots
\label{rhf893hf83hf03j9f333}
\end{align}
where $\e_\al= \e_{\widehat\Xfr}|_{\Uc_\al}$ is a local expression for the Euler vector field; $Q_\al\in T_\Xfr^{(1)}(\Uc_\al)$ and $Q_\al^{\{j\}}\in T_{\widehat\Xfr}^{\{j\}}(\Uc_\al)$.\footnote{Note, while $H^\8$ given by \eqref{rfh8949f83hf0j30} is holomorphic over $\Uc_\al$,this does not imply $H^\8$ will be holomorphic as a global vector field on $\Xfr$. See footnote \ref{gd783gd79893hd89hd0}.} The local isomorphism in \eqref{efh74gf793f98h3f3} justifies the expression in \eqref{rhf893hf83hf03j9f333}. Now recall the boundary mapping $\pt : H^0(X, T_\Xfr^{(0)}/T_\Xfr^{(\ell)}) \ra H^1(X, T_\Xfr^{(\ell)})$. Since $H^\8$ is a smooth extension of a mod $\ell$ lift of $\e_{\widehat\Xfr}$ we have on intersections $\Uc_{\al\be}\stackrel{\Delta}{=}\Uc_\al\cap \Uc_\be$,
\begin{align}
(\pt\overline H)_{\al\be} 
&= \pt (H^\8\mod T_\Xfr^{(\ell)})_{\al\be}
\notag
\\
&= 
(H_\be^\8\mod T_\Xfr^{(\ell)}(\Uc_\be)) - (H^\8_\al \mod T_\Xfr^{(\ell)}(\Uc_\al)) \in T_\Xfr^{(\ell)}(\Uc_{\al\be}).
\label{hd93hf98h30f39}
\end{align}
Comparing this with the expression in \eqref{rfh8949f83hf0j30} and using that $\e_\al = \e_\be$ on $\Uc_{\al\be}$, we see also that $Q_\al  \equiv Q_\be\mod T_\Xfr^{(\ell)}(\Uc_{\al\be})$. As such there exists some smooth, global vector field $Q^\8 \in H^0(X, T_\Xfr^{(1);\8})$ such that $Q^\8|_{\Uc_\al} \equiv Q_\al \mod T_\Xfr^{(\ell)}(\Uc_\al)$ and for all $\al$. Forming the difference $H^\8 - Q^\8$ of smooth vector fields on $\Xfr$ now reveals,
\[
(H^\8 - Q^\8)|_{\Uc_\al} \equiv \e_\al \mod T_\Xfr^{(\ell)}(\Uc_\al)
\]
Hence W.L.O.G., and referencing the expression \eqref{rhf893hf83hf03j9f333}, any smooth extension of a mod $\ell$ lift of $\e_{\widehat\Xfr}$ can be taken to be locally of the form 
\begin{align}
H^\8|_{\Uc_\al} = \e_\al + Q^{\{\ell\}}_\al + Q^{\{\ell+1\}}_\al +\cdots.
\label{rgf74gf97h08f30fj30}
\end{align}
We will suppose now that $H^\8$ is a smooth extension of a mod $\ell$ lift of $\e_{\widehat\Xfr}$ with $\ell> 1$ and is given as in \eqref{rgf74gf97h08f30fj30}. In the statement of Koszul's theorem (Theorem \ref{buie9h08d3jd33f3}) recall that we assume the existence of a global, even, affine connection on $\Xfr$. Consider the construct $\nabla_{H^\8}H^\8$. Over $\Uc_\al$ we have:
\begin{align}
(\nabla_{H^\8}H^\8)|_{\Uc_\al}
&=
\nabla_{H^\8|_{\Uc_\al}}(H^\8|_{\Uc_\al})
\notag
\\
&=
\nabla_{\e_\al}\e_\al+ \nabla_{Q^{(\ell)}_\al}\e_\al + \nabla_{\e_\al}Q^{(\ell)}_\al + \cdots
\notag
\\
&=
\e_\al +  \nabla_{Q^{(\ell)}_\al}\e_\al +\ell Q_\al^{(\ell)} + w^{alg.}_\al + \cdots
\label{rfg87f783gf98hf893}
\end{align}
where \eqref{rfg87f783gf98hf893} follows from Lemma \ref{rfjiohf4hf89hf90hf0} and Proposition \ref{rhf4f9f9j9f39f34f4f4} for $w^{alg.}_\al\in T_\Xfr^{(\ell)}(\Uc_\al)$ some algebraic section. Importantly, the ellipses $`\ldots$' above denote terms in $T^{(\ell+1)}_\Xfr(\Uc_\al)$. By Lemma \ref{rfh794fg9hf80h30f3} the section $ \nabla_{Q^{(\ell)}_\al}\e_\al$ will be algebraic. Hence the only not-necessarily-algebraic section in \eqref{rfg87f783gf98hf893} is $\ell Q^{(\ell)}_\al$. In order to eliminate it, we form $H^{\8;alg.} \stackrel{\Delta}{=}\frac{1}{\ell-1}(\nabla_{H^\8}H^\8 - \ell H^\8)$.\footnote{recall that we are assuming $\ell>1$.} We then have
\[
H^{\8;alg.}|_{\Uc_\al}
\equiv
\e_\al + W^{alg.}_\al
\mod T_\Xfr^{(\ell+1)}(\Uc_\al).
\]
for $W^{alg.}_\al\in T_\Xfr^{(\ell)}(\Uc_\al)$ some algebraic section. Form again now:
\begin{align}
H^\8[\![1]\!]
\stackrel{\Delta}{=}
\frac{1}{\ell-1}\left(\nabla_{H^\8}H^{\8;alg.} - \ell H^{\8;alg.}\right).
\label{rf7gf9hf8309fj30}
\end{align}
From the shear-like transformation in Proposition \ref{rfh89fh983hf80f0j93} then,
\[
(H^\8[\![1]\!])|_{\Uc_\al}
\equiv 
\e_\al
\mod 
T^{(\ell+1)}_\Xfr(\Uc_\al).
\]
We have therefore shown if $H^\8$ is as in \eqref{rgf74gf97h08f30fj30} for any $\ell$ and defines thereby a smooth extension of a mod $\ell$ lift of $\e_{\widehat\Xfr}$, then with the global, even connection $\nabla$ we can modify $H^\8$ to $H^\8[\![1]\!]$ so that now $H^\8[\![1]\!]$ will define a smooth extension of a mod $(\ell+1)$ lift of $\e_{\widehat\Xfr}$. Since $\nabla$ is global, the modification $H^\8[\![1]\!]$ is well defined and so we have established an inductive step. To see how Koszul's theorem can now follow we need to establish the base case. But this is precisely the content of Proposition \ref{rfh874gf94hf8j30} which asserts there will always exist a mod $2$ lift of the Euler vector field to any supermanifold. We can now apply induction to derive a mod $(\dim_-\Xfr+1)$ lift of $\e_{\widehat\Xfr}$. Koszul's theorem then follows from Corollary \ref{rf783gf873g9f7h38hf03}.
\qed
\\\\
\noindent
Our proof of Koszul's theorem above invoked Corollary \ref{rf783gf873g9f7h38hf03} concerning lifts of the Euler vector field. For an alternative perspective, observe from \eqref{hd93hf98h30f39} that we are using the global, even, affine connection to solve the equation $\dt\e_{\widehat\Xfr} = 0$ in $H^1(X, T_\Xfr^{(1)})$. Hence, such a connection represents a solution to this equation and, by Theorem \ref{djvevcyibeon}, any solution to this equation results in a splitting.

\subsection{Further Commentary: A Unique Splitting}
\label{fiubeuifuefohroifjeof}
The Euler vector field exists on any split model $\widehat\Xfr$ and is unique. As we can infer from Theorem \ref{rhf9hf983fhfh03} however, we do not necessarily need the existence of a unique vector field to split a given supermanifold $\Xfr$.  Once we find some global vector field $H$ on $\Xfr$ with initial form $\e_{\widehat\Xfr}$, we can use $H$ to split $\Xfr$. This $H$ need not be unique and any other $H^\p$ with initial form $\e_{\widehat\Xfr}$ will suffice. In the presence of a global, even, affine connection $\nabla$ on $\Xfr$ however we can in fact solve for a \emph{unique} such $H$ on $\Xfr$ with initial form $\e_{\widehat\Xfr}$. This was observed by Koszul in \cite{KOSZUL}. We present a statement and proof below.

\begin{THM}\label{rfh973f983hf9h30}
Let $\nabla$ be a global, even, affine connection on $\Xfr$. Then there exists a unique, degree zero vector field $H^\nabla\in H^0(X, T_\Xfr^{(0)})$ such that 
\begin{align}
\nabla_{H^\nabla}H^\nabla = H^\nabla.
\label{rfg7gf983h8f9h398h}
\end{align}
\end{THM} 

\begin{proof}
We know from our proof of Koszul's theorem that there exists some $H\in H^0(X, T_\Xfr^{(0)})$ mapping onto $\e_{\widehat\Xfr}$. To see that there will exist some unique such $H$ associated to $\nabla$ and satisfying \eqref{rfg7gf983h8f9h398h}, consider starting with some mod $\ell$ lift  $\overline H\in H^0(X, T_\Xfr^{(0)}/T_\Xfr^{(\ell)})$ for some $\ell>1$, mapping onto $\e_{\widehat\Xfr}$ and let $H^\8$ be a smooth extension of $\overline H$. The process\footnote{c.f., \eqref{rf7gf9hf8309fj30}} $H^\8\stackrel{\nabla}{\mapsto}  H^\8[\![1]\!] \stackrel{\nabla}{\mapsto} (H^\8[\![1]\!] )[\![1]\!]\stackrel{\nabla}{\mapsto} \cdots$ will eventually, and after finitely many steps,\footnote{more precisely, after $(\dim_-\Xfr - \ell + 1)$-many steps.} stabilise to some $H^\nabla$ which, in a given atlas $\Uc =(\Uc_\al)$ satisfies, $H^\nabla|_{\Uc_\al} = \e_\al = \e_{\widehat\Xfr}|_{\widehat \Uc_\al}, \forall \al$. In particular, $H^\nabla$ is independent of choice of $\ell$. 
Furthermore, by Lemma \ref{rfjiohf4hf89hf90hf0} we have:
\begin{align}
(\nabla_{H^\nabla}H^\nabla)|_{\Uc_\al} 
=
\nabla_{\e_\al}\e_\al = \e_\al = H^\nabla|_{\Uc_\al}.
\label{rf8f93hf98h8f300}
\end{align}
Since $\nabla$ and $H^\nabla$ are global, \eqref{rf8f93hf98h8f300} implies $\nabla_{H^\nabla}H^\nabla = H^\nabla$. Now suppose $H\in H^0(X, T_\Xfr^{(0)})$ maps onto $\e_{\widehat\Xfr}$ and satisfies $\nabla_HH = H$. Locally we can write $H|_{\Uc_\al} = \e_\al + Q_\al$ for some $Q_\al\in T_\Xfr^{(1)}(\Uc_\al)$.\footnote{c.f., \eqref{rfh8949f83hf0j30}.} As such, by Lemma \ref{rfjiohf4hf89hf90hf0} and Proposition \ref{rhf4f9f9j9f39f34f4f4}:
\begin{align*}
(\nabla_HH)|_{\Uc_\al} &= \nabla_{\e_\al + Q_\al}(\e_\al + Q_\al)
\\
&=
\e_\al + Q^{\{1\}}_\al + w^{alg.}_\al + 2Q^{\{2\}}_\al + w^{\p alg.}+\cdots + \nabla_{Q_\al}Q_\al
\end{align*}
where $Q^{\{j\}}_\al \in T_{\widehat\Xfr}^{\{j\}}(|\Uc_\al|)$ denotes the projection of $Q_\al$ under the local isomorphism \eqref{efh74gf793f98h3f3}.
Evidently, we can equate $\nabla_HH = H$ if and only if $Q_\al = 0$ (which obviously implies $\nabla_{Q_\al}Q_\al = 0$); and in this case we find $H|_{\Uc_\al} = H^{\nabla}|_{\Uc_\al}$ for all $\al$, so thereby we must identify $H$ and $H^\nabla$.
\end{proof}

\noindent
From the proof of Theorem \ref{rfh973f983hf9h30} note that the unique vector field $H^\nabla$ associated to $(\Xfr, \nabla)$, where $\Xfr = \widehat\Xfr$ is the split model, is precisely the Euler vector field $\e_{\widehat\Xfr}$. As a consequence now of Theorem \ref{rfh973f983hf9h30}, we have:

\begin{COR}
Associated to any global, even, affine connection on a supermanifold $\Xfr$ is a unique splitting $\Xfr\stackrel{\cong}{\ra} \widehat\Xfr$.\qed
\end{COR}


\newpage
\part{The Super Atiyah Class}

\section{Atiyah Classes: Preliminaries}

\subsection{Preliminaries: On Manifolds} 
Koszul's theorem relates the existence of a global, even, affine connection on a supermanifold $\Xfr$ with a splitting of $\Xfr$. Hence, the existence of an obstruction class to splitting $\Xfr$ will also obstruct the existence of such a connection. Classically, the Atiyah class of a complex manifold measures precisely the obstruction to the existence of a global, holomorphic connection. This is the central construction in Atiyah's paper \cite{ATCONN} which we paraphrase as follows.

\begin{THM}\label{rgf783gf793hf983h}
Let $E$ be a holomorphic vector bundle on a complex manifold $X$ and with sheaf of holomorphic sections $\Ec$. Associated to $E$ is the Atiyah sequence
\[
0\lra \mathcal Hom_{\Oc_X}(\Ec,\Om^1_X\otimes \Ec) \lra \mathcal Hom_{\Oc_X}(\Ec, J^1\Ec)\lra \mathcal Hom_{\Oc_X}(\Ec,\Ec) \lra 0
\] 
where $J^1\Ec$ is the sheaf of holomorphic $1$-jets of sections of $E$; and $\Oc_X$ is the structure sheaf of holomorphic functions on $X$. The Atiyah sequence induces on cohomology the mapping $H^0(X, \mathcal Hom_{\Oc_X}(\Ec, \Ec)) \stackrel{\dt}{\ra} H^1(X, Hom_{\Oc_X}(\Ec,\Om^1_X\otimes \Ec))$. The following two statements are now equivalent:
\begin{enumerate}[(i)]
	\item $E$ admits a global, holomorphic connection;
	\item $\dt({\bf 1}_\Ec) = 0$.
\end{enumerate}
where ${\bf 1}_\Ec$ is the identity mapping $\Ec\stackrel{=}{\ra} \Ec$. 
\qed
\end{THM}

\noindent
The equivalence of statements Theorem \ref{rgf783gf793hf983h}\emph{(i)} and Theorem \ref{rgf783gf793hf983h}\emph{(ii)} justifies referring to $\dt({\bf 1}_\Ec)$ as a `complete obstruction' to the existence of a global, holomorphic connection on $E$. We introduce now the following notation, to any vector bundle $E$ on $X$ with sheaf of holomorphic sections $\Ec$:
\begin{align}
\mathcal At~\Ec &\stackrel{\Delta}{=} \mathcal Hom_{\Oc_X}(\Ec, \Om^1_X\otimes \Ec);
\label{rfg7gf93hf8h03}
\\
\mathrm{AT}~\Ec &\stackrel{\Delta}{=} H^1(X, \mathcal At~\Ec);
\label{dekjbckveyev}
\\
\mathrm{at}~\Ec
&\stackrel{\Delta}{=}
\dt({\bf 1}_\Ec).
\label{rfg783gf873iu3g7f9}
\end{align}
We refer to $\mathcal At~\Ec$ as the Atiyah \emph{sheaf} of $\Ec$; $\mathrm{AT}~\Ec$ as the Atiyah \emph{space} of $\Ec$; and $\mathrm{at}~\Ec$ as the Atiyah \emph{class} of $\Ec$. 

%

\subsection{The Affine Atiyah Class}\label{rf87gf79h398fh30}
If $\Ec$ is the tangent sheaf on $X$, then the term `Atiyah' is prefixed by the term `affine'. In this special case, the calculus of differential forms will ensure that the Atiyah class $\mathrm{at}~T_X$ will be valued in symmetric $2$-forms.\footnote{\label{rfg78gf7938hf803j90f3}c.f., the proof of Theorem \ref{rburghurhoiejipf} in Appendix \ref{rgf78gf7hf983hf30f9j3}.} And so, in the case $\Ec = T_X$ we modify the Atiyah sheaf in \eqref{rfg7gf93hf8h03} to:
\begin{align}
\mathcal At~T_X \stackrel{\Delta}{=} \odot^2T^*_X\otimes T_X.
\label{rf783f739fh83f03h0}
\end{align}
The sheaf in \eqref{rfg7gf93hf8h03} will be referred to as the \emph{affine} Atiyah sheaf of $X$; and similarly for the corresponding space $\mathrm{AT}~T_X$ and class $\mathrm{at}~T_X$. As for vector bundles generally in Theorem \ref{rgf783gf793hf983h}, the affine Atiyah class of $X$ measures the failure for there to exist a global, affine, holomorphic connection on $X$.
The generalisation to supermanifolds is now immediate. We have already encountered the notion of tangent sheaves for supermanifolds in \S\ref{h837fg983hf839j093}. More generally, Bruzzo et. al. in \cite[p. 161]{BRUZZO} present a generalisation of Theorem \ref{rgf783gf793hf983h} to supermanifolds.
%
%
And so we have straightforwardly the notion of an affine Atiyah sheaf, space and class for any supermanifold $\Xfr$. 

\subsection{The Even Part}
A new feature in supergeometry relative to (commutative) geometry is the presence of a global dichotomy  between \emph{even} and \emph{odd}. On a supermanifold $\Xfr$, the global $\Zbb_2$-grading on its functions $\Oc_\Xfr\cong \Oc_{\Xfr, +}\oplus \Oc_{\Xfr,-}$ induces a grading on its tangents $T_\Xfr \cong T_{\Xfr,+}\oplus T_{\Xfr, -}$ and so also on its Atiyah sheaf $\mathcal At~T_\Xfr\cong (\mathcal At~T_\Xfr)_+\oplus (\mathcal At~T_\Xfr)_-$, realising it as a graded, $\Oc_{\Xfr, +}$-module. Accordingly, we have a decomposition $\mathrm{AT}~T_\Xfr\cong  (\mathrm{AT}~T_{\Xfr})_+\oplus (\mathrm{AT}~T_{\Xfr})_-$ with respect to which:

\begin{LEM}\label{rhf894hf893hf0830}
\[
\mathrm{at}~T_\Xfr \in (\mathrm{AT}~T_\Xfr)_{+}.
\]
\end{LEM}

\begin{proof}
Recall that $\mathcal End_{\Oc_\Xfr}T_\Xfr$ is an $\Oc_\Xfr$-algebra and the section ${\bf 1}_{T_\Xfr}\in H^0(X, \mathcal End_{\Oc_\Xfr}T_\Xfr)$ is invertible, i.e., is in the image of $H^0(X, \mathcal Aut_{\Oc_\Xfr}T_\Xfr)$. It must therefore be in the even component of $\mathcal End_{\Oc_\Xfr}T_\Xfr$, where the decomposition into even and odd components is as $\Oc_{\Xfr,+}$-modules. The Atiyah sequence from Theorem \ref{rgf783gf793hf983h} adapted to supermanifolds will be a sequence of even, $\Oc_\Xfr$-morphisms so therefore, 
\[
\mathrm{at}~T_\Xfr \stackrel{\Delta}{=}\dt({\bf 1}_{T_\Xfr})\in H^1(X, (\mathcal At~T_\Xfr)_+) \stackrel{\Delta}{=} (\mathrm{AT}~T_\Xfr)_+,
\] 
as required.
\end{proof}

\noindent
We will hereafter be concerned specifically with the even Atiyah class of supermanifolds. As might be expected from Theorem \ref{rgf783gf793hf983h}, the even Atiyah class is related to global, even, affine connections thusly:

\begin{THM}\label{rfh894hf89hf04hf0}
The even, affine Atiyah class of any supermanifold $\Xfr$ measures the failure for there to exist a global, even, affine connection on $\Xfr$.
\end{THM}

\begin{proof}
By Theorem \ref{rburghurhoiejipf} the space of global, even, affine connections forms a torsor over the sections $H^0(X, (\mathcal At~T_\Xfr)_+)$. Arguing in analogy with \cite{ATCONN} will reveal that the obstruction to the existence of a global connection will lie in $(\mathrm{AT}~T_\Xfr)_+$. See also the more general argument in \cite[p. 161]{BRUZZO}.
\end{proof}

\begin{REM}
\emph{In dropping the prefix `even', the affine Atiyah class will more generally measure the failure for there to exist a global, affine connection on $\Xfr$ from Definition \ref{gf7498fh3f30jf39f0}.}
\end{REM}


\section{Preliminaries on Obstruction Theory} 
\label{rfh894hf94hf0093}

\subsection{Preliminaries}
\label{rhf84f0jf9fj94jf445545}
The term `obstruction theory' has many connotations in a number of mathematical disciplines and is typically associated with the study of the obstructions to the lifting of some structure. In topology, obstruction theory might pertain the obstructions to lifting topological to smooth structures; smooth structures to orientable structures; and these latter to spin structures, as in \cite[\S6--9]{MILCHAR}. In complex and algebraic geometry, obstruction theory might pertain to the obstructions to the existence of, and liftings of, versal deformations as in \cite[Ch. 5]{KS} and \cite[Ch. 2]{HARTDEF}. In complex supergeometry, following the work of Berezin in \cite[Ch. 4]{BER}, Green in \cite{GREEN} and Manin in \cite[Ch. 4]{YMAN}, obstruction theory concerns the obstructions to the existence of supermanifold splittings. 
\\\\
In this section we present some of the rudiments of obstruction theory for supermanifolds essential for subsequent investigations into the super Atiyah class. Material on the classification of supermanifolds are further prerequisites for the material presented here and so we have included a review of this classification  in Appendix \ref{rh894rhf894hf8f0j30}. In contrast to existing treatments in the literature, our review is tailored toward a viewpoint which is functorial in supermanifolds (see the concluding section of Appendix \S\ref{rh894rhf894hf8f0j30}). To begin now, fix a model $(X, T^*_{X, -})$ and recall the automorphism sheaves $\mathcal G^{(m)}_{(X, T^*_{X, -})}$ from \eqref{noiebu4ifeiofpe}, presented below for convenience:
\begin{align*}
\mathcal G^{(m)}_{(X, T^*_{X, -})}
\stackrel{\Delta}{=}
\left\{
\al\in (\mathcal Aut\wedge^\bt T^*_{X, -})_+
\mid 
\al(u) - u\in \bigoplus_{\ell\geq m}\wedge^\ell T^*_{X,-}
\right\}.
\end{align*}
Green in \cite{GREEN} established the following.

\begin{THM}\label{rf748gf9hf8fj93jf9334}
To any model $(X, T^*_{X, -})$ the automorphism sheaves $(\Gc^{(m)}_{(X, T^*_{X, -})})_{m\in \Zbb}$ satisfy the following:
\begin{enumerate}[(i)]
	\item $\forall m> \mathrm{rank} ~T^*_{X, -}, \Gc^{(m)}_{(X,T_{X, -}^*)} = \{{\bf 1}\}$;
	\item $\forall m < 0,  \Gc^{(m)}_{(X,T_{X, -}^*)} = \eset$;
	\item $\forall m\geq 0,  \Gc^{(m+1)}_{(X,T_{X, -}^*)}\subset  \Gc^{(m)}_{(X,T_{X, -}^*)}$ is normal;
	\item $\forall m\geq 2,  \Gc^{(m)}_{(X,T_{X, -}^*)}/ \Gc^{(m+1)}_{(X,T_{X, -}^*)}$ is abelian.  
\end{enumerate}
\qed
\end{THM}

\noindent
If $\Gc$ is a non-abelian sheaf of groups then its cohomology is only defined in degrees zero and one. In degree zero it is a group and in degree one it is a pointed set. Grothendieck in \cite{GROTHNONAB} nevertheless observed that to any short exact sequence of sheaves of not-necessarily-abelian groups there will be induced a long exact sequence on cohomology as in the case of abelian sheaves, albeit truncated to degrees zero and one and viewed as pointed sets. For any $m\geq 2$ set
\begin{align}
\mathcal Ob^{(m)}_{(X, T^*_{X, -})} \stackrel{\Delta}{=} \frac{ \Gc^{(m)}_{(X,T_{X, -}^*)}}{\Gc^{(m+1)}_{(X,T_{X, -}^*)}}
&&
\mbox{and}
&&
\mathrm{OB}^{(m)}_{(X, T^*_{X, -})}
\stackrel{\Delta}{=}
H^1(X, \mathcal Ob^{(m)}_{(X, T^*_{X, -})}).
\label{rhf89hf89hf093fj3j}
\end{align}
Note that $\mathcal Ob^{(m)}_{(X, T^*_{X, -})}$ makes sense by Theorem \ref{rf748gf9hf8fj93jf9334}\emph{(iii)}; and by Theorem \ref{rf748gf9hf8fj93jf9334}\emph{(iv)}
it will be abelian. 

\begin{DEF}\label{jf90jf904fj9f3333444}
\emph{The sheaf $\Oc b^{(m)}_{(X, T^*_{X, -})}$ in \eqref{rhf89hf89hf093fj3j} will be referred to as an \emph{obstruction sheaf} and the space $\mathrm{OB}^{(m)}_{(X, T^*_{X, -})}$ as an \emph{obstruction space}. In the special case $m = 2$ we refer to this sheaf as the \emph{primary} obstruction sheaf and write $\mathcal Ob^{primary}_{(X, T^*_{X, -})}$. Accordingly, its $1$-cohomology will be referred to as the obstruction \emph{space} and \emph{primary} obstruction space respectively with the latter written $\mathrm{OB}^{primary}_{(X, T^*_{X, -})}$.}
\end{DEF}

\noindent
As a consequence of Theorem \ref{rf748gf9hf8fj93jf9334}\emph{(iii)} we have a long exact sequence of pointed sets which, on cohomology, ends on the piece:
\begin{align}
\xymatrix{
\cdots\ar[r] & 
\mbox{\v H}^1(\Gc^{(m+1)}_{(X, T^*_{X, -})}) \ar[r] & \mbox{\v H}^1(\Gc^{(m)}_{(X, T^*_{X, -})}) \ar[r]^{\om_*} & \mathrm{OB}^{(m)}_{(X, T^*_{X, -})}.
}
\label{rfh794f983hf8003}
\end{align}
Exactness ensures that $\img\{\mbox{\v H}^1(\Gc^{(m+1)}_{(X, T^*_{X, -})}) \ra \mbox{\v H}^1(\Gc^{(m)}_{(X, T^*_{X, -})})\}\cong \ker\om_*$.\footnote{\label{rfg784gf97hf83h}The kernel of a mapping between pointed sets comprise all those elements which, under the mapping in question, map to the base-point. This includes, in particular, the base-point in the former set and so the kernel will not be empty. Now by Theorem \ref{rf748gf9hf8fj93jf9334}\emph{(iv)} the obstruction space $\mathrm{OB}^{(m)}_{(X, T^*_{X, -})}$ will be a complex vector space. Its base-point corresponds then to the zero vector ${\bf 0}$. 
} 
From Green's classification in Theorem \ref{rgf78g37fg39f3h80} any framed supermanifold $(\Xfr, \phi)$ will define an element $[\Xfr, \phi]$ in $\mbox{\v H}^1(\Gc^{(2)}_{(X, T^*_{X, -})})$ and so, to framed supermanifolds $(\Xfr, \phi)$ we have naturally associated an element $\om_*[\Xfr, \phi]\in \mathrm{OB}^{primary}_{(X, T^*_{X,-})}$.

\begin{DEF}\label{rhf894hf89hf80h30}
\emph{To any framed supermanifold $(\Xfr, \phi)$ the element  $\om_*[\Xfr, \phi]$ will be referred to as its \emph{primary obstruction} to splitting. }
\end{DEF}

\noindent
A vanishing primary obstruction class is an almost a meaningless contribution to the question of whether a supermanifold splits. Its non-vanishing is highly instructive however and formed an integral fact in Donagi and Witten's study of the supermoduli space of curves. More precisely, we have the following:

\begin{THM}\label{rfg7f9hf3h80f3}
Let $\Xfr$ be a supermanifold and suppose, with respect to some framing $\phi$, that its primary obstruction $\om_*([\Xfr, \phi])\in \mathrm{OB}^{primary}_{(X, T^*_{X, -})}$ is non-vanishing.\footnote{c.f., footnote \ref{rfg784gf97hf83h}.} Then $\Xfr$ is non-split.
\end{THM}

\begin{proof}
See e.g., \cite[Appx. A]{BETTHIGHOBS}.
\end{proof}

\begin{REM}\label{hfefh9hf98hf03}
\emph{A consequence of Theorem \ref{rfg7f9hf3h80f3} is that the class $\om_*([\Xfr, \phi])$ depends essentially on the isomorphism class of $\Xfr$ and not the particular choice of framing $\phi$. Henceforth, by  abuse of notation, we will simply refer to the primary obstruction class of $(\Xfr, \phi)$ by $\om_*[\Xfr]$ and omit explicit reference to a framing $\phi$.}
\end{REM}

\begin{REM}\label{rgf78gf79fhhf8h034}
\emph{The sequence in \eqref{rfh794f983hf8003} suggests the contrivance of higher obstruction classes to framed supermanifolds. These would be elements in $\mathrm{OB}^{(m)}_{(X, T^*_{X, -})}$ somehow associated to a given, framed supermanifold $(\Xfr, \phi)$ and general $m$. Such an association cannot be naturally associated to $\Xfr$ however due to the existence of `exotic structures'. These were identified by Donagi and Witten in \cite{DW1} as obfuscating the resolution of the splitting problem for supermanifolds by reference to obstruction theory alone. The author in \cite{BETTHIGHOBS} presented a further study in higher obstruction theory concerning cases where such a theory can be naturally associated to supermanifolds. As the setting involving the primary obstruction class will suffice for our purposes in this article, we will not consider `higher' obstruction theory here.}
\end{REM}

\subsection{Naturality}\label{dhbcjdbvjhfbvbuieboi}
In Remark \ref{rgf78gf79fhhf8h034} it was noted that `higher' obstruction classes need not be naturally associated to supermanifolds. This is in contrast to the \emph{primary} obstruction class which, as mentioned in the comment succeeding Theorem \ref{rfg7f9hf3h80f3}, depends essentially on the isomorphism class of supermanifolds. In order to arrive then at a more intrinsic description of the primary obstruction class, we need to appeal to the intrinsic description of the automorphism groups $\Gc_{\Oc_\Xfr}^{(m)}$ in \eqref{rjf84hf803h09fj39fj} in contrast to the extrinsic description $\Gc_{(X, T^*_{X, -})}^{(m)}$ in \eqref{noiebu4ifeiofpe}. And so with $\Gc_{\Oc_\Xfr}^{(m)}$ playing the role of $\Gc_{(X, T^*_{X, -})}^{(m)}$ in Theorem \ref{rf748gf9hf8fj93jf9334} we have:

\begin{THM}\label{rfh794f9hf8h380f}
For any supermanifold $\Xfr$, the automorphism sheaves $(\Gc^{(m)}_{\Oc_\Xfr})_{m\in \Zbb}$ satisfy:
\begin{enumerate}[(i)]
	\item $\forall m> \dim_-\Xfr, \Gc^{(m)}_{\Oc_\Xfr} = \{{\bf 1}\}$;
	\item $\forall m< 0, \Gc^{(m)}_{\Oc_\Xfr} = \eset$;
	\item $\forall m\geq 0, \Gc^{(m+1)}_{\Oc_\Xfr} \subset \Gc^{(m)}_{\Oc_\Xfr}$ is normal;
	\item $\forall m\geq 2, \Gc^{(m)}_{\Oc_\Xfr} /\Gc^{(m+1)}_{\Oc_\Xfr} $ is abelian.
\end{enumerate}
\end{THM}

\begin{proof}
Regarding \emph{(i)},
for $\Jc_\Xfr\subset \Oc_\Xfr$ the fermonic ideal, it is nilpotent with degree $(\dim_-\Xfr+1)$, i.e., that $\Jc_\Xfr^{\dim_-\Xfr+1} = (0)$. Hence $\Gc^{(m)}_{\Oc_\Xfr}$ is the trivial group for $m > \dim_-\Xfr$; \emph{(ii)} is vacuously true since $\Gc^{(m)}_{\Oc_\Xfr}$ is not defined for $m< 0$;\footnote{alternatively, see the proof of Lemma \ref{f783g79fh89fh8h4f}.} \emph{(iii)} the proof given by Green in \cite{GREEN} applies straightforwardly in this general case; and lastly \emph{(iv)},\footnote{in contrast to \emph{(iii)}, Green's proof in \cite{GREEN} cannot be readily adapted for, in \cite{GREEN}, Green makes explicit use of the $\Zbb$-grading on the exterior algebra.} note that if we can express $\Gc^{(m+1)}_{\Oc_\Xfr}$ as the kernel of a homomorphism $\Gc^{(m)}_{\Oc_\Xfr} \stackrel{f}{\ra} \Hc$ for some abelian sheaf $\Hc$, then the quotient, which exists by \emph{(iii)}, will necessarily be abelian. This is because it will be isomorphic to the image $\img\{f: \Gc^{(m)}_{\Oc_\Xfr}\ra \Hc\}$ which, being a subsheaf of an abelian sheaf $\Hc$, must itself be abelian. And so, onto the particulars of the morphism $f$ and sheaf $\Hc$, recall that $\Gc_{\Oc_\Xfr}^{(m)}$ is a sheaf of automorphisms of $\Oc_\Xfr$ of a certain kind. Hence it embeds in the endomorphisms $\Gc^{(m)}_{\Oc_\Xfr}\subset \mathcal End_{\Oc_\Xfr}\Oc_\Xfr$ which is itself a sheaf of $\Oc_\Xfr$-algebras. In $\mathcal End_{\Oc_\Xfr}\Oc_\Xfr$ it makes sense to form binary operations such as the sum or difference of sections of $\Gc_{\Oc_\Xfr}^{(m)}$ and so, similarly to Onishchik in \cite[p. 56]{ONISHCLASS} in the split case, to any $g\in \mathcal G^{(m)}_{\Oc_\Xfr}$ we can use the formal logarithm to map $\Gc^{(m)}_{\Oc_\Xfr} \stackrel{\ln}{\ra} \mathcal End_{\Oc_\Xfr}\Oc_\Xfr$ by 
\begin{align}
g \stackrel{\ln}{\longmapsto} (g - {\bf 1}) + \frac{(g - {\bf 1})^2}{2} - \cdots
\label{rfh9f983hf83hf03}
\end{align} 
where ${\bf 1}\in \Gc_{\Oc_\Xfr}^{(m)}$ is the identity. Observe that, by definition, $g - {\bf 1} : \Oc_\Xfr \ra \Jc^m_\Xfr$ so therefore the formal sum in \eqref{rfh9f983hf83hf03} will only contain finitely many terms. Furthermore it will be additive, i.e., $\ln (g\circ h) = \ln g + \ln h$ for all $g, h\in \Gc^{(m)}_{\Oc_\Xfr}$. Now observe for any $\ell> 0$ that $(g - {\bf 1})^\ell\in \Jc_\Xfr^{m + \ell}$ and so $\ln g\equiv g -{\bf 1} \mod \Jc_\Xfr^{m+1}$. If $g\in \Gc^{(m+1)}_{\Oc_\Xfr}$ then $g - {\bf 1} : \Oc_\Xfr\ra \Jc_\Xfr^{m+1}$ giving therefore $\ln g\equiv 0\mod \Jc^{m+1}_\Xfr$. We have therefore a left exact sequence of sheaves on $X$:
\begin{align}
\xymatrix{
{\bf 1} \ar[r] & \Gc_{\Oc_\Xfr}^{(m+1)} \ar[r] & \Gc_{\Oc_\Xfr}^{(m)} \ar[r] & \mathcal Hom_{\Oc_\Xfr} \left({\Oc_\Xfr}, \frac{\Jc_\Xfr^m}{\Jc_\Xfr^{m+1}}\right).
}
\label{rfh84f409fj43jf4443r4}
\end{align}
Hence $\Gc^{(m)}_{\Oc_\Xfr}/\Gc^{(m+1)}_{\Oc_\Xfr}\subset \mathcal Hom_{\Oc_\Xfr} \left({\Oc_\Xfr}, \frac{\Jc_\Xfr^m}{\Jc_\Xfr^{m+1}}\right)$ will be abelian.
\end{proof}

\noindent
As in \eqref{rhf89hf89hf093fj3j} we set
\begin{align*}
\mathcal Ob^{(m)}_{\Oc_\Xfr} \stackrel{\Delta}{=} \frac{\Gc^{(m)}_{\Oc_\Xfr}}{\Gc^{(m+1)}_{\Oc_\Xfr}}
&&
\mbox{and}
&&
\mathrm{OB}^{(m)}_\Xfr \stackrel{\Delta}{=} H^1(X, \mathcal Ob^{(m)}_{\Oc_\Xfr}).
\end{align*}
By Theorem \ref{rfh794f9hf8h380f}\emph{(iv)}, $\mathrm{OB}^{(m)}_\Xfr$ will be a finite dimensional, complex vector space. Definition \ref{jf90jf904fj9f3333444} now adapts sraightforwardly.

\begin{DEF}
\emph{To any supermanifold $\Xfr$, the sheaf $\mathcal Ob^{(m)}_{\Oc_\Xfr}$ will be referred to as the \emph{$m$-th obstruction sheaf of $\Xfr$} and the space $\mathrm{OB}^{(m)}_\Xfr$ as the $m$-th obstruction \emph{space} of $\Xfr$. In the special case $m = 2$ we refer to these objects as the \emph{primary} obstruction sheaf and space and accordingly denote $\mathcal Ob^{primary}_{\Oc_\Xfr}\stackrel{\Delta}{=}\mathcal Ob^{(2)}_{\Oc_\Xfr}$ and  $\mathrm{OB}^{primary}_\Xfr\stackrel{\Delta}{=}\mathrm{OB}^{(2)}_\Xfr$.}
\end{DEF}

\noindent
Since $\Gc^{(m)}_{\Oc_{\widehat\Xfr}} = \Gc^{(m)}_{(X, T^*_{X, -})}$ for all $m$, we readily recover the constructions from \S\ref{rhf84f0jf9fj94jf445545} upon specialising to the split model, i.e., in taking $\Xfr = \widehat\Xfr$.

\subsection{The Primary Obstructions}
From our classification in Theorem \ref{hf9hf983hf8h30f0}, the isomorphism class of any supermanifold $\Xfr$ will define the base-point in the pointed set $\Mfr_\Xfr$.\footnote{Recall \eqref{hfegf73f8h30f34f}.} If $\Xfr$ is, in addition, equipped with a framing $\phi$, then by the characterisation in \eqref{g64gf784h9f83f0399}, we can identify $\Gc^{(2)}_{\Oc_\Xfr}\cong (\mathcal Aut_{\Oc_\Xfr}\Oc_\Xfr)^\phi$, i.e., that $\Gc^{(2)}_{\Oc_\Xfr}$ will comprise those automorphisms which preserve the framing.\footnote{Recall from Definition \ref{rfh894hf8hf093j03} that a framing is an isomorphism $\Oc_\Xfr/\Jc_\Xfr^2\cong \Oc_X\oplus (\Jc_\Xfr/\Jc_\Xfr^2)$.} Hence, under this identification, the isomorphism class of any supermanifold $\Xfr$ equipped with a framing $\phi$ will define the base-point $[\Xfr, \phi] \in \mbox{\v H}^1(X, (\mathcal Aut_{\Oc_\Xfr}\Oc_\Xfr)^\phi)$. Since $\Gc^{(2)}_{\Oc_\Xfr}\cong (\mathcal Aut_{\Oc_\Xfr}\Oc_\Xfr)^\phi$ we have 
$ \mbox{\v H}^1(X, (\mathcal Aut_{\Oc_\Xfr}\Oc_\Xfr)^\phi) \cong \mbox{\v H}^1(X, \Gc^{(2)}_{\Oc_\Xfr})$ and so  the isomorphism class of any $(\Xfr, \phi)$ will define the base-point in $\mbox{\v H}^1(X, \Gc^{(2)}_{\Oc_\Xfr})$, which we will also denote by $[\Xfr, \phi]$. Now from Theorem \ref{rfh794f9hf8h380f}\emph{(iii)} we have induced a sequence on cohomology,
\[
\xymatrix{
\cdots\ar[r] & 
\mbox{\v H}^1(X, \mathcal G_{\Oc_\Xfr}^{(3)})\ar[r] & \mbox{\v H}^1(X, \mathcal G_{\Oc_\Xfr}^{(2)}) \ar[r]^{\eta_*} & \mathrm{OB}^{primary}_\Xfr.
}
\]
Hence, to any framed supermanifold $(\Xfr, \phi)$ we have associated a class $\eta_*[\Xfr, \phi]$.

\begin{DEF}\label{fbivurviuebfoenie}
\emph{To any framed supermanifold $(\Xfr, \phi)$ the class $\eta_*[\Xfr, \phi]$ will be referred to as the \emph{primary obstruction to splitting $\Xfr$}.}
\end{DEF}

%

\noindent
Now in Definition \ref{rhf894hf89hf80h30} we presented another class associated to framed supermanifolds $(\Xfr, \phi)$ in the obstruction space $\mathrm{OB}^{primary}_{\widehat\Xfr}$. The two classes are related as follows.

\begin{PROP}\label{rffg874g97h9fh83}
For any supermanifold $\Xfr$ with associated split model $\widehat\Xfr$ there is an isomorphism of primary obstruction spaces $\mathrm{OB}^{primary}_{\Xfr}\cong \mathrm{OB}^{primary}_{\widehat\Xfr}$ under which, for any framing $\phi$, that $\eta_*[\Xfr, \phi] \mapsto \om_*[\Xfr]$.\footnote{c.f., Remark \ref{hfefh9hf98hf03}.}
\end{PROP}

\begin{proof}
Recall that a choice of framing $\phi$ on $\Xfr$ is an isomorphism $\Oc_\Xfr/\Jc_\Xfr^2\cong \Oc_X\oplus (\Jc_\Xfr/\Jc_\Xfr^2)$. In particular, $\Oc_\Xfr/\Jc^2_\Xfr\cong \Oc_{\widehat\Xfr}/\Jc^2_{\widehat\Xfr}$. Now recall the left exact sequence in \eqref{rfh84f409fj43jf4443r4} for any $\Xfr$. That this sequence is in fact a short exact sequence of sheaves of groups follows from the characterisations of these group elements in \eqref{g64gf784h9f83f0399} and \eqref{g64gf7833f34h9f83f0399}. Hence,
\begin{align}
\mathcal Ob^{primary}_{\Oc_\Xfr}
\stackrel{\Delta}{=}
\frac{\Gc^{(2)}_{\Oc_\Xfr}}{\Gc^{(3)}_{\Oc_\Xfr}}
\cong  
\mathcal Hom_{\Oc_\Xfr} \left(\frac{\Oc_\Xfr}{\Jc^{2}_\Xfr}, \frac{\Jc_\Xfr^2}{\Jc_\Xfr^{3}}\right).
\label{rfg783gf793hf893hf9}
\end{align}
Using that any framing gives an isomorphism $\Oc_\Xfr/\Jc^2_\Xfr\cong \Oc_{\widehat\Xfr}/\Jc^2_{\widehat\Xfr}$; and furthermore that $\Jc_\Xfr^m/\Jc_\Xfr^{m+1}\cong \Jc_{\widehat\Xfr}^m/\Jc_{\widehat\Xfr}^{m+1}$ for any $m$ allows us to link the isomorphisms in \eqref{rfg783gf793hf893hf9} with:
\[
\mathcal Hom_{\Oc_\Xfr} \left(\frac{\Oc_\Xfr}{\Jc^{2}_\Xfr}, \frac{\Jc_\Xfr^2}{\Jc_\Xfr^{3}}\right)
\cong 
\mathcal Hom_{\Oc_{\widehat\Xfr}} \left(\frac{\Oc_{\widehat\Xfr}}{\Jc^{2}_{\widehat\Xfr}}, \frac{\Jc_{\widehat\Xfr}^2}{\Jc_{\widehat\Xfr}^{3}}\right)
\cong 
\frac{\Gc^{(2)}_{\Oc_{\widehat\Xfr}}}{\Gc^{(3)}_{\Oc_{\widehat\Xfr}}}
\stackrel{\Delta}{=} 
\mathcal Ob^{primary}_{\Oc_{\widehat\Xfr}}.
\]
Hence from any framing $\phi$ we obtain an isomorphism $\mathcal Ob^{primary}_{\Oc_\Xfr}\cong \mathcal Ob^{primary}_{\Oc_{\widehat\Xfr}}$; and therefore a isomorphism of obstruction spaces $\mathrm{OB}^{primary}_{\Xfr}\cong \mathrm{OB}^{primary}_{\widehat\Xfr}$. With regards to the primary obstruction class itself, of which we presently have two incarnations, they can readily be identified by appealing to classifications of supermanifolds. Rather than insisting on an isomorphism of sheaves of groups, observe that we have at least an isomorphism of $1$-cohomologies: \emph{for any supermanifold $\Xfr$ with associated split model $\widehat\Xfr$ that,}
\begin{align}
\mbox{\v H}^1(X, \Gc^{(2)}_{\Oc_\Xfr}) \cong \mbox{\v H}^1(X, \Gc^{(2)}_{\Oc_{\widehat\Xfr}}).
\label{rfj093jf09j90fj30}
\end{align}
The justification behind \eqref{rfj093jf09j90fj30} proceeds as follows: each cohomology is a pointed set and (1) by Green's classification in Theorem \ref{rgf78g37fg39f3h80}; and (2) its variant in Theorem \ref{hf9hf983hf8h30f0}, each $1$-cohomology in \eqref{rfj093jf09j90fj30} will classify the same objects for the reason that $\Xfr$ and $\widehat\Xfr$ are locally isomorphic. As a result they will in turn lie in bijective correspondence and so are isomorphic as sets. Note, they certainly need not be isomorphic as pointed sets, for the base-point in $\mbox{\v H}^1(X, \Gc^{(2)}_{\Oc_\Xfr})$ is $[\Xfr, \phi]$ and, if $\Xfr$ is non-split, cannot be mapped to the base-point in $\mbox{\v H}^1(X, \Gc^{(2)}_{\Oc_{\widehat\Xfr}})$, which corresponds to the split model, since this point will not correspond to $[\Xfr, \phi]$ by Theorem \ref{rgf78g37fg39f3h80} (Green's classification). Now with a choice of framing $\phi$ we can generate:
\begin{enumerate}[(i)]
	\item the base-point $[\Xfr, \phi]\in \mbox{\v H}^1(X, \Gc^{(2)}_{\Oc_\Xfr})$;
	\item a corresponding point $[\Xfr, \phi]^\p\in  \mbox{\v H}^1(X, \Gc^{(2)}_{\Oc_{\widehat\Xfr}})$;
	\item an isomorphism $\mathrm{OB}^{primary}_{\Xfr}\cong \mathrm{OB}^{primary}_{\widehat\Xfr}$;
\end{enumerate}
And so with \eqref{rfj093jf09j90fj30} any framing $\phi$ will generate a commutative diagram:
\begin{align*}
\xymatrix{
\ar[d]^\cong \mbox{\v H}^1(X, \Gc^{(2)}_{\Oc_\Xfr}) \ar[rr]^{\eta_*} & & \mathrm{OB}^{primary}_{\Xfr}\ar[d]^\cong
\\
 \mbox{\v H}^1(X, \Gc^{(2)}_{\Oc_{\widehat\Xfr}})\ar[rr]^{\om_*} & & \mathrm{OB}^{primary}_{\widehat\Xfr}
}
&&
\mbox{with}
&&
\xymatrix{
\ar[d] [\Xfr, \phi]\ar@{|->}[r] & \eta_*[\Xfr, \phi]\ar@{|-->}[d]
\\
[\Xfr, \phi]^\p\ar@{|->}[r] & \om_*[\Xfr, \phi]^\p
}
\end{align*}
Hence $\eta_*[\Xfr, \phi]$ will correspond to $\om_*[\Xfr,\phi]^\p = \om_*[\Xfr]$ as required. 
\end{proof}

\noindent
With Proposition \ref{rffg874g97h9fh83} we can now immediately adapt Theorem \ref{rfg7f9hf3h80f3}.

\begin{THM}\label{rfg873gf97hf893h8f03}
Let $(\Xfr,\phi)$ be a framed supermanifold with non-vanishing, primary obstruction $\eta_*[\Xfr, \phi]\in \mathrm{OB}_{\Xfr}^{primary}$. Then $\Xfr$ is non-split.\qed
\end{THM}

\begin{REM}\label{hf893hf83h409fj3}
\emph{As in Remark \ref{hfefh9hf98hf03} then, a consequence of Theorem \ref{rfg873gf97hf893h8f03} is that the primary obstruction of any supermanifold---in the sense of Definition \ref{fbivurviuebfoenie}, depends essentially on the isomorphism class of $\Xfr$ and not the particular choice of framing. As such, by abuse of notation, we will refer to its primary obstruction without referring to a choice of framing.}
\end{REM}

\section{The Atiyah Class of Split Models}
\label{rhf894hf984hfh8033f4}
\noindent
With the affine Atiyah sheaf, space and class established for supermanifolds in \S\ref{rf87gf79h398fh30}, we can now look to inspect this class in more detail. We begin here with a characterisation of the affine Atiyah class of split models. Subsequently, we establish a relation between the affine Atiyah class of supermanifolds $\Xfr$ more generally and between that of its split model $\widehat\Xfr$.

\subsection{On the Split Model}
Our objective here will be to prove the following.

\begin{PROP}\label{rfh894hf89hf0j903}
Let $\widehat\Xfr$ be the split model associated to $(X, T^*_{X, -})$ and denote by $\iota: X\subset \widehat\Xfr$ the inclusion of the reduced space. Then on $X$ we have a decomposition of the even, affine Atiyah space,
\begin{align}
\big(\iota^\sharp
\mathrm{AT}~T_{\widehat\Xfr} 
\big)_{+}
\cong 
\mathrm{AT}~T_X\oplus \mathrm{AT}~T^*_{X, -}\oplus \mathrm{OB}^{primary}_{{\widehat\Xfr}}.
\label{rhf894hf89h3f30}
\end{align}
Let $pr: \big(\iota^\sharp
\mathrm{AT}~T_{\widehat\Xfr} 
\big)_{+}\ra \mathrm{AT}~T_X\oplus \mathrm{AT}~T^*_{X, -}$ be the map projecting out the primary obstruction space. Then under this projection the affine Atiyah class of $\widehat\Xfr$ satisfies,\footnote{c.f., Lemma \ref{rhf894hf893hf0830}.} 
\[
pr\circ \iota^\sharp \mathrm{at}~T_{\widehat\Xfr}
=
\mathrm{at}~T_X\oplus \mathrm{at}~T^*_{X,-}. 
\]
\end{PROP}

\begin{proof}
With $\iota: X\subset \widehat\Xfr$ the inclusion given in the statement of this theorem, it corresponds to the quotient on function algebras $i^\sharp : \Oc_{\widehat\Xfr} \twoheadrightarrow \Oc_X = \Oc_{\widehat\Xfr}/\Jc_{\widehat\Xfr}$ where $\Jc_{\widehat\Xfr}\subset \Oc_{\widehat\Xfr}$ is the fermonic ideal. Recall from Corollary \ref{rgf78gf738fh309f3} that on tangents we have:
\begin{align}
i^\sharp T_{\widehat\Xfr}
=
T_{\widehat\Xfr}
\otimes_{\Oc_{X}}
\left(\frac{\Oc_{\widehat\Xfr}}{\Jc_{\widehat\Xfr}}\right)
\cong 
T_{X, -} 
\oplus 
T_X.
\label{rgf4gf7g38fh893f3}
\end{align}
Now pullbacks under inclusions will commute with tensor products. Hence,
\begin{align*}
i^\sharp \mathcal At~T_{\widehat\Xfr}
&=
i^\sharp \big(\odot^2\Om^1_{\widehat\Xfr} \otimes T_{\widehat\Xfr}\big)
&&
\mbox{(recall \eqref{rf783f739fh83f03h0})}
\\
&=
\odot^2\big(
T^*_{X, -} 
\oplus 
\Om^1_X
\big)
\oplus 
\big(
T_{X, -} 
\oplus 
T_X
\big)
&&
\mbox{(by \eqref{rgf4gf7g38fh893f3})}.
\end{align*}
To identify the even components now we will use that dualising is an `odd' operation', i.e., for any $\Oc_\Xfr$-module\footnote{note, here we are considering a general supermanifold $\Xfr$.} $\Mcl$ with $\Zbb_2$-decomposition $\Mcl\cong\Mcl_+\oplus \Mcl_-$ that $(\Mcl^*)_{\pm} = \Mcl_{\mp}^*$. In particular 
$\Om^1_{\widehat\Xfr, \pm}
=
T_{\widehat\Xfr, \mp}^*$. This leads to,
\begin{align*}
\odot^2\Om^1_{\widehat\Xfr}
&=
\odot^2\big(
\Om^1_{\widehat\Xfr, +}
\oplus
\Om^1_{\widehat\Xfr, -}
\big)
\\
&=
\odot^2\Om^1_{\widehat\Xfr,+}
\oplus
\wedge^2
\Om^1_{\widehat\Xfr, -}
\oplus 
\big(\Om^1_{\widehat\Xfr,+}\otimes \Om^1_{\widehat\Xfr,-}\big)
\\
&=
\wedge^2T_{\widehat\Xfr, -}^*
\oplus 
\odot^2 T_{\widehat\Xfr, +}^*
\oplus 
\big(
T_{\widehat\Xfr, -}^*
\otimes 
T_{\widehat\Xfr, +}^*
\big).
\end{align*}
Using that the parity distributes additively over tensor products, the even component of the Atiyah sheaf can be identified thusly:
\begin{align}
\big(\mathcal At~T_{\widehat\Xfr}\big)_+
=&~
\big(
\odot^2\Om^1_{\widehat\Xfr}\otimes T_{\widehat\Xfr}
\big)_+
\notag
\\
=&~
\big( 
\odot^2\Om^1_{\widehat\Xfr, +}\otimes T_{\widehat\Xfr, +}
\big)
\oplus 
\big(
\wedge^2\Om^1_{\widehat\Xfr, -}\otimes T_{\widehat\Xfr, +}
\big)
\notag
\\
&~
\oplus
\big(
\Om^1_{\widehat\Xfr, +}
\otimes 
\Om^1_{\widehat\Xfr, -}
\otimes
T_{\widehat\Xfr, -}
\big)
\oplus
\big(
\Om^1_{\widehat\Xfr, -}
\otimes 
\Om^1_{\widehat\Xfr, +}
\otimes 
T_{\widehat\Xfr, -}
\big)
\notag
\\
\cong&~
\big( 
\odot^2\Om^1_{\widehat\Xfr, +}\otimes T_{\widehat\Xfr, +}
\big)
\oplus 
\big(
\wedge^2\Om^1_{\widehat\Xfr, -}\otimes T_{\widehat\Xfr, +}
\big)
\oplus
\big(
\Om^1_{\widehat\Xfr, +}
\otimes \mathcal End~T_{\widehat\Xfr, -}
\big)
\notag
\\
=&~
\big(
\wedge^2T^*_{\widehat\Xfr, -}\oplus \odot^2T^*_{\widehat\Xfr, +}
\big)
\otimes T_{\widehat\Xfr, +}
\oplus
\big(
\Om^1_{\widehat\Xfr, +}
\otimes \mathcal End~T_{\widehat\Xfr, -}
\big).
\notag
\end{align}
Pullbacks will commute with tensor, symmetric and anti-symmetric products and so on $X$ we have by \eqref{rgf4gf7g38fh893f3}:
\begin{align}
i^\sharp
\big(\mathcal At~T_{\widehat\Xfr}\big)_+
&\cong
\big(
\wedge^2i^\sharp T^*_{\widehat\Xfr, -}\oplus \odot^2i^\sharp T^*_{\widehat\Xfr, +}
\big)
\otimes i^\sharp T_{\widehat\Xfr, +}
\oplus
\big(
i^\sharp\Om^1_{\widehat\Xfr, +}
\otimes \mathcal End~i^\sharp T_{\widehat\Xfr, -}
\big)
\notag
\\
&\cong
\big(
\wedge^2T_{X, -}^*\otimes T_X
\big)
\oplus
\big(
\odot^2\Om^1_X
\otimes T_X
\big)
\oplus
\big(
\Om^1_X\otimes \mathcal End~T_{X, -}
\big)
\notag
\\
&=
\mathcal Ob_{\Oc_{\widehat\Xfr}}^{primary}
\oplus 
\mathcal At~T_X
\oplus 
\mathcal At~T_{X, -} 
\label{rgf673gf73h09j3f33434}
\end{align}
Since cohomology commutes with pullbacks and direct sums we can conclude from \eqref{rgf673gf73h09j3f33434} that,
\begin{align*}
i^\sharp
\big(\mathrm{AT}~T_{\widehat\Xfr}\big)_+
&= 
H^1(X, i^\sharp(\mathcal At~T_{\widehat\Xfr})_+)
\cong 
\mathrm{AT}~T_X
\oplus 
\mathrm{AT}~T_{X, -}
\oplus
\mathrm{OB}^{primary}_\Xfr.
\end{align*}
We have arrived now at the desired decomposition of the Atiyah space in \eqref{rhf894hf89h3f30}. Where the decomposition of the Atiyah class is concerned, firstly recall the Atiyah sequences for $T_X$ and $T_{X, -}$ respectively. From Theorem \ref{rgf783gf793hf983h} we set, for notational convenience:
\[
\widetilde{J^1\Ec}
\stackrel{\Delta}{=}
\mathcal Hom_{\Oc_X}(\Ec, J^1\Ec).
\]
The Atiyah sequences for $T_X$ and $T_{X, -}$ are now:
\begin{align*}
0 \ra \mathcal At~T_X \ra \widetilde{J^1T_X} \ra  & \mathcal End_{\Oc_X}T_X \ra 0
&&
\mbox{and}
&&
0 \ra \mathcal At~T_{X, -} \ra \widetilde{J^1T_{X, -}} \ra  & \mathcal End_{\Oc_X}T_{X, -} \ra 0
\end{align*}
The sequences above can be combined to form the sum:
\begin{align*}
\xymatrix{
0
\ar[r] & 
\mathcal At~T_X\oplus \mathcal At~T_{X, -}
\ar[r] & 
 \widetilde{J^1T_X}\oplus  \widetilde{J^1T_{X, -}}
 \ar[r] & 
 \mathcal End~T_X
\oplus 
\mathcal End~T_{X, -}
\ar[r]
& 
0}
\end{align*}
Now note the following:
\begin{enumerate}[(i)]
	\item $i^\sharp \big(\mathcal End~T_{\widehat\Xfr}\big)_+ \cong \mathcal End~T_X
\oplus 
\mathcal End~T_{X, -}$ and;
	\item as $\Oc_X$-modules, $\widetilde{J^1T_X}\cong \mathcal At~T_X\oplus \mathcal End_{\Oc_X}T_X$ and $\widetilde{J^1T_{X, -}}\cong \mathcal At~T_{X, -}\oplus \mathcal End_{\Oc_X}T_{X, -}$.\footnote{see e.g., \cite[\S4, p. 193]{ATCONN}. Indeed, in the notation from Theorem \ref{rgf783gf793hf983h}, $J^1\Ec\cong (\Om^1_X\otimes_{\Oc_X} \Ec)\oplus \Ec$. A holomorphic connection is a splitting of the sequence $0\ra \Om^1_X\otimes\Ec\ra J^1\Ec\ra \Ec\ra0$ with respect to a `twisted' $\Oc_X$-module structure on $J^1\Ec$ given by $f\cdot (\om\otimes s + t) \stackrel{\Delta}{=} f\om\otimes s + df\otimes t + ft$.}
\end{enumerate}
Observation (ii) above generalises to $\widetilde{J^1T_{\widehat\Xfr}}\cong \mathcal At~T_{\widehat\Xfr}\oplus \mathcal End_{\Oc_{\widehat\Xfr}}T_{\widehat\Xfr}$ as $\Oc_{\widehat\Xfr}$-modules. Hence by observation (i) and \eqref{rgf673gf73h09j3f33434}: 
\[
i^\sharp \big(\widetilde{J^1 T_{\widehat\Xfr}}\big)_+\cong \mathcal At~T_X\oplus \mathcal At~T_{X, -}
\oplus
\mathcal Ob^{primary}_{\Oc_{\widehat\Xfr}}
\oplus \mathcal End~T_X\oplus \mathcal End~T_{X, -}.
\]
Therefore, pulling back to $X$ and projecting out $ \mathcal Ob^{primary}_{\Oc_{\widehat\Xfr}}$ yields a morphism of Atiyah sequences:
\begin{align*}
\xymatrix{
0
\ar[r] & 
\big(\mathcal At~T_{\widehat\Xfr}\big)_+
\ar[r] 
\ar[d]^{pr \circ i^\sharp}
& 
\big(\widetilde{J^1T_{\widehat\Xfr}}\big)_+
\ar[r] 
\ar[d]^{pr \circ i^\sharp}
& 
\big(\mathcal End~T_{\widehat\Xfr}\big)_+
\ar[r] 
\ar[d]^{i^\sharp}
& 0
\\
0\ar[r]& 
\mathcal At~T_{X}\oplus \mathcal At~T_{X,-}\ar[r]
& 
\widetilde{J^1T_X}\oplus \widetilde{J^1T_{X, -}}
\ar[r]
& 
\mathcal End~T_X\oplus \mathcal End~T_{X, -}
\ar[r]
&
0
}
\end{align*}
Importantly, we have: $i^\sharp({\bf 1}_{T_{\widehat\Xfr}}) = {\bf 1}_{T_X}\oplus {\bf 1}_{T_{X, -}}$. By Theorem \ref{rgf783gf793hf983h} the Atiyah class of $T_X$ resp. $T_{X, -}$ is $\dt({\bf 1}_{T_X})$ resp. $\dt({\bf 1}_{T_{X, -}})$ where $\dt$ is the boundary map on cohomology induced the respective Atiyah sequences for $T_X$ and $T_{X, -}$. Similarly, the (even) Atiyah class of $T_{\widehat\Xfr}$ is $\dt({\bf 1}_{T_{\widehat\Xfr}})$ where $\dt$ here is the boundary map on cohomology induced by the (even) Atiyah sequence for $T_{\widehat\Xfr}$. Now from our morphism of Atiyah sequences above we have the following commutative diagram on cohomology:
\[
\xymatrix{
\ar[d]^{i^\sharp}
H^0\big(X, \big(\mathcal End~T_{\widehat\Xfr}\big)_\ev\big) \ar[rr]|\dt & & H^1\big(X, \big(\mathcal At~T_{\widehat\Xfr}\big)_\ev\big)
\ar[d]^{i^\sharp}
\\
H^0\big(X, \mathcal End~T_X\big)
\oplus 
H^0\big(X, \mathcal End~T_{X, -}\big)
\ar[rr]|{\dt\oplus\dt}& & 
H^1\big(X, \mathcal At~T_X\big)
\oplus 
H^1\big(X, \mathcal At~T_{X, -}\big)
}
\]
Using $i^\sharp({\bf 1}_{T_{\widehat\Xfr}}) = {\bf 1}_{T_X}\oplus {\bf 1}_{T_{X, -}}$ gives:
\begin{align*}
i^\sharp
\mathrm{at}~T_{\widehat\Xfr}
=
i^\sharp \dt ({\bf 1}_{T_{\widehat\Xfr}})
=
(\dt\oplus \dt)
(i^\sharp {\bf 1}_{T_{\widehat\Xfr}})
=
\dt ({\bf 1}_{T_X})
\oplus 
\dt ({\bf 1}_{T_{X, -}})
=
\mathrm{at}~T_X
\oplus 
\mathrm{at}~T_{X, -}.
\end{align*}
The proposition now follows.
\end{proof}

\subsection{On Supermanifolds More Generally}
Recall that the affine Atiyah class is defined for any supermanifold $\Xfr$, split or otherwise. We intend here to obtain a relation between the Atiyah class of supermanifolds with that of its split model $\widehat\Xfr$. Central to obtaining such a relation is the initial form sequence from  Lemma \ref{rhf78f9hf083j09fj30}. We recall it below for convenience:
\begin{align}
0
\lra
T_\Xfr[\![1]\!]
\lra
T_\Xfr 
\lra 
T_{\widehat\Xfr}
\lra 
0.
\label{rgf7648gf87f93hf30}
\end{align}
Preempting a relation between the classes now, the following lemma details a relation between Atiyah sheaves.

\begin{LEM}
\label{jfbkebvuiebueoroinoifr}
There exists a commutative diagram of sheaves on $X$:
\begin{align}
\xymatrix{
0\ar[r]&\ar[d] \odot^2T^*_{\widehat\Xfr}\otimes T_\Xfr \ar[r] & \mathcal At~T_\Xfr\ar[d]^\pi
\\
0\ar[r] & 
\mathcal At~T_{\widehat\Xfr}
\ar[r]_\iota
& 
\odot^2T^*_{\Xfr}\otimes T_{\widehat\Xfr}
}
\label{rhf7f93hf893f094j393}
\end{align}
\end{LEM}

\begin{proof}
We begin with a general fact on powers of modules (see e.g., \cite[p. 121]{HARTALG}). To any short exact sequence of sheaves of modules $0\ra \Fc^\p\ra \Fc\ra\Fc^{\p\p}\ra0$ the $r$-th symmetric or anti-symmetric power of $\Fc$, denoted $AS^r\Fc$, admits a length-$(r+1)$ filtration:
\[
AS^r\Fc = F^0\supset F^1\supset \cdots F^r\supset 0
\] 
with successive quotients:
\[
F^p/F^{p+1} \cong AS^p\Fc^\p \otimes AS^{r-p}\Fc^{\p\p}.
\]
Returning now to the initial form sequence in \eqref{rgf7648gf87f93hf30} for the tangent sheaf of a supermanifold $T_\Xfr$, dualising gives
\[
0
\lra
T_{\widehat\Xfr}^*
\lra
T_\Xfr^*
\lra
T_{\Xfr}[\![1]\!]^*
\lra
0
\] 
and so by our observation on the symmetric and anti-symmetric powers above we know that $\odot^2T^*_\Xfr$ will be filtered as follows:
\[
\odot^2T^*_\Xfr = F^0 \supset F^1\supset F^2\supset 0
\]
with quotients:
\begin{align*}
F^0/F^1 &= \odot^2T^*_\Xfr/F^1 \cong
\odot^2T^*_\Xfr[\![1]\!];
\\
F^1/F^2 &\cong T^*_{\widehat\Xfr}\otimes T_{\Xfr}[\![1]\!]^*
~~
\mbox{and;}
\\
F^2/F^3 &= F^2 =  \odot^2T_{\widehat\Xfr}^*.
\end{align*}
Hence\footnote{note, this sequence is \emph{not} exact. However $\odot^2T^*_{\widehat\Xfr}$ is in the kernel of $ \odot^2T^*_{\Xfr} \ra \odot^2T_{\Xfr}[\![1]\!]^*\ra 0.$}
\begin{align}
\odot^2T^*_{\widehat\Xfr} \subset \odot^2T^*_{\Xfr} \twoheadrightarrow  \odot^2T_{\Xfr}[\![1]\!]^*
\label{rf87gf93hf83f0j09j3f3}
\end{align}
Now recall that:
\begin{align}
\mathcal At~T_\Xfr
= 
\odot^2T^*_\Xfr \otimes T_\Xfr
&&
\mbox{and}
&&
\mathcal At~T_{\widehat\Xfr}
=
\odot^2T^*_{\widehat\Xfr} \otimes T_{\widehat\Xfr}.
\end{align}
Since tensor products with locally free sheaves will preserve injections, we obtain the top horizontal arrow $0\ra \odot^2T^*_{\widehat\Xfr}\otimes T_\Xfr \ra  \mathcal At~T_\Xfr$ in \eqref{rhf7f93hf893f094j393}. The vertical arrow $\odot^2T^*_{\widehat\Xfr}\otimes T_\Xfr \ra \mathcal At~T_{\widehat\Xfr}$ follows from tensoring the surjection of tangents $T_\Xfr \ra T_{\widehat\Xfr}\ra 0$ from \eqref{rgf7648gf87f93hf30} with $\odot^2T^*_{\widehat\Xfr}$. The injection $0\ra \mathcal At~T_{\widehat\Xfr} \ra \odot^2T^*_\Xfr \otimes T_{\widehat\Xfr}$ on the bottom row of \eqref{rhf7f93hf893f094j393} follows from tensoring the inclusion in \eqref{rf87gf93hf83f0j09j3f3} with $T_{\widehat\Xfr}$. Finally, the vertical map $\mathcal At~T_\Xfr \ra  \odot^2T^*_\Xfr \otimes T_{\widehat\Xfr}$ follows from tensoring the surjection $T_\Xfr \ra T_{\widehat\Xfr}\ra 0$ with $\odot^2T^*_\Xfr$.\footnote{Note that the vertical maps in \eqref{rhf7f93hf893f094j393} are not necessarily surjective.}
\end{proof}

\noindent
The purpose of Lemma \ref{jfbkebvuiebueoroinoifr} is to make sense of the statement of the following theorem which concerns our relation between Atiyah classes.

\begin{THM}\label{fjbkbfbuecioenice}
With $\iota$ and $\pi$ the morphisms in the diagram from Lemma \ref{jfbkebvuiebueoroinoifr}, let $\iota_*$ and $\pi_*$  denote the respective induced morphisms on $1$-cohomology. Then we have the following relation between Atiyah classes:
\[
\pi_* \mathrm{at}~T_\Xfr = \iota_*\mathrm{at}~T_{\widehat\Xfr}.
\]
\end{THM}

\begin{proof}
Recall from \eqref{rgf7648gf87f93hf30} that $T_\Xfr$, the tangent sheaf of $\Xfr$, surjects onto $T_{\widehat\Xfr}$, the tangent sheaf of $\widehat\Xfr$. Not every endomorphism of $T_\Xfr$ will induce an endomorphism of $T_{\widehat\Xfr}$ however. To that end, consider the subsheaf of endomorphisms $\mathcal End^*T_\Xfr\subset \mathcal End~T_\Xfr$ which induce endomorphisms of $T_{\widehat\Xfr}$. Elements of $H^0(X, \mathcal End^*T_\Xfr)$ define endomorphisms of the following short exact sequence, represented by the dashed arrows:
\begin{align}
\xymatrix{
0\ar[r] & T_\Xfr[\![1]\!] \ar@{-->}[d] \ar[r] & T_\Xfr \ar@{-->}[d] \ar[r] & T_{\widehat\Xfr} \ar@{-->}[d]\ar[r] & 0
\\
0\ar[r] & T_\Xfr[\![1]\!] \ar[r] & T_\Xfr \ar[r] & T_{\widehat\Xfr} \ar[r] & 0
}
\label{rfh784gf873hf93h8f30}
\end{align}
Since the identity ${\bf 1}_{T_\Xfr} : T_\Xfr \stackrel{=}{\ra}T_\Xfr$ will induce a morphism as in \eqref{rfh784gf873hf93h8f30}, it follows that\footnote{
Note that for any sheaf $\Fc$ of $\Oc$-modules we have a natural map $H^0(\Fc)\otimes \Oc \ra \Fc$ given by $(\phi\otimes f)_x = \phi(x)\otimes [f]_x$, where $[f]_x$ denotes the germ of $f$ at the point $x$. If $\Oc$ is unital, then we have the mapping $H^0(\Fc)\ra \Fc$ given by $(\phi)_x \mapsto \phi(x)\otimes [1]_x$. Then with $\Oc_\Xfr$ the function algebra for a supermanifold $\Xfr$ with unit the constant function $(x|\q)\mapsto 1\in \Cbb$; and with $\mathcal End~T_\Xfr$ an $\Oc_\Xfr$-module, we can consider ${\bf 1}_{T_\Xfr}$ as an element of $\mathcal End~T_\Xfr$ or $\mathcal End^*T_\Xfr$.} ${\bf 1}_{T_\Xfr}\in \mathcal End^*T_\Xfr$, i.e., that $\mathcal End^*T_\Xfr$ is non-empty. Furthermore, this induced endomorphism will be the identity $T_{\widehat\Xfr}\stackrel{=}{\ra}T_{\widehat\Xfr}$. Hence we can conclude:
\begin{align}
\mbox{under the maps} &&
\xymatrix{
\ar[d] \mathcal End^*T_\Xfr \ar[r] & \mathcal End~T_\Xfr
\\
\mathcal End~T_{\widehat\Xfr}
}
&&
\mbox{that}
&&
\xymatrix{
\ar@{|->}[d] {\bf 1}_{T_\Xfr}^* \ar@{|->}[r] & {\bf 1}_{T_\Xfr}
\\
{\bf 1}_{T_{\widehat\Xfr}}&
}
\label{fcbvyueviuvoienipe}
\end{align}
Now form the sheaf
\[
J^* 
\stackrel{\Delta}{=}
(\odot^2T^*_{\widehat\Xfr}\otimes T_\Xfr)
\oplus
\mathcal End^*T_\Xfr
\]
Then by Lemma \ref{jfbkebvuiebueoroinoifr} and the definition of $\mathcal End^*T_\Xfr$ we have,
\[
\xymatrix{
\ar[d] J^* \ar[r] & \widetilde{J^1T_\Xfr}
\\
\widetilde{J^1{T_{\widehat\Xfr}}}
}
\]
The twisted $\Oc_\Xfr$-module structure on $\widetilde{J^1T_\Xfr}$ can then be induced on $J^*$ yielding therefore a morphism of exact sequences:
\begin{align}
\xymatrix{
&\mathcal At~T_\Xfr \ar[r] & \widetilde{J^1T_\Xfr}\ar[r] & \mathcal End~T_\Xfr
\\
\ar[dd] \odot^2T^*_{\widehat\Xfr}\otimes T_\Xfr \ar[ur] \ar[r] & \ar[dd] J^*\ar[ur] \ar[r] & \ar[dd] \mathcal End^*T_\Xfr\ar[ur]
\\\\
\mathcal At~T_{\widehat\Xfr}\ar[r] &\widetilde{J^1T_{\widehat\Xfr}} \ar[r] & \mathcal End~T_{\widehat\Xfr}
}
\label{rh784gf973hf98hf93jf90}
\end{align}
Importantly, the bottom and top-right sequences in \eqref{rh784gf973hf98hf93jf90} are precisely the Atiyah sequences of $T_\Xfr$ and $T_{\widehat\Xfr}$ respectively.
On cohomology, the boundary maps give:
\begin{align}
\xymatrix{
& H^0(X, \mathcal End~T_\Xfr) \ar[rr]|\dt & & \mathrm{AT}~T_\Xfr
\\
H^0(X, \mathcal End^*T_\Xfr) \ar[dd]|\dt\ar[rr] \ar[ur]& &\ar[ur]|{i_*} H^1(X, \odot^2T^*_{\widehat\Xfr}\otimes T_\Xfr) \ar[dd]|{p_*}
\\\\
 H^0(X, \mathcal End~T_{\widehat\Xfr})\ar[rr]|\dt & & \mathrm{AT}~T_{\widehat\Xfr}
}
\label{jckdvyuevciuebconebcie}
\end{align}
Since the above diagram commutes and the sequences in \eqref{rh784gf973hf98hf93jf90} inducing this diagram are Atiyah sequences, looking at image of ${\bf 1}^*_{T_\Xfr}\in H^0(X, \mathcal End^*T_\Xfr)$ from \eqref{fcbvyueviuvoienipe} implies:
\begin{align}
\xymatrix{
\ar@{|->}[d]|{p_*} \dt{\bf 1}^*_{T_\Xfr} \ar@{|->}[rr]|{i_*} & & \dt{\bf 1}_{T_\Xfr} \stackrel{\Delta}{=}  \mathrm{at}~T_\Xfr
\\
 \dt{\bf 1}_{T_{\widehat\Xfr}} \ar@{=}[r]^\Delta&\mathrm{at}~T_{\widehat\Xfr}
}
\label{dcndbvurvoienivmep}
\end{align}
where $i_*$ and $p_*$ are the induced maps in \eqref{jckdvyuevciuebconebcie}. To deduce Theorem \ref{fjbkbfbuecioenice} now we can `complete the prism' on the right in \eqref{jckdvyuevciuebconebcie} by Lemma \ref{jfbkebvuiebueoroinoifr}, resulting in the commuting square:
\begin{align}
\xymatrix{
\ar[d]_{p_*} H^1(\odot^2T^*_{\widehat\Xfr} \otimes T_\Xfr ) \ar[rr]^{i_*} & & \mathrm{AT}~T_\Xfr\ar[d]_{\pi_*}
\\
\mathrm{AT}~T_{\widehat\Xfr} \ar[rr]^{\iota_*} & & H^1(\odot^2T^*_\Xfr\otimes T_{\widehat\Xfr})
}
\label{fbyrbvuibveioveivpe}
\end{align}
Theorem \ref{fjbkbfbuecioenice} now follows as a consequence of the mappings in \eqref{dcndbvurvoienivmep} and commutativity of \eqref{fbyrbvuibveioveivpe}.
\end{proof}

\section{Donagi and Witten's Decomposition}
\label{r4g4g949g49g94j904j}

\noindent
A key component to Donagi and Witten's second proof of the non-projectedness of the supermoduli space of punctured curves in \cite{DW2} is their decomposition of the affine Atiyah class of supermanifolds. In \cite[Theorem 2.5, p. 14 (arXiv version)]{DW2}, Donagi and Witten show, on any supermanifold $\Xfr$, that
$\mathrm{at}~T_\Xfr$ restricts to the Atiyah classes of the underlying spaces and the \emph{primary} obstruction class to splitting. We have partly derived this decomposition into the underlying Atiyah classes in Proposition \ref{rfh894hf89hf0j903}. The component valued in the obstruction space is given by the following.

\begin{THM}\label{rfg784gf7hf98h38f03}
Let $\Xfr$ be a supermanifold modelled on $(X, T^*_{X, -})$. Then with $\iota: X\subset \Xfr$ the given inclusion of the reduced space, we have on $X$ a decomposition of the even Atiyah space
\begin{align}
\big(\iota^\sharp
\mathrm{AT}~T_{\Xfr} 
\big)_{+}
\cong 
\mathrm{AT}~T_X\oplus \mathrm{AT}~T^*_{X, -}\oplus \mathrm{OB}^{primary}_{{\Xfr}}
\label{rhf8944333r3hf89h3f30}.
\end{align}
Correspondingly, the affine Atiyah class of $\Xfr$ decomposes as follows:
\begin{align}
\iota^\sharp\mathrm{at}~T_\Xfr = \mathrm{at}~T_X\oplus \mathrm{at}~T_{X, -}\oplus \eta_*[\Xfr].
\label{kdnjkbjcvecbe}
\end{align}
where $\eta_*[\Xfr]$ is the primary obstruction to splitting $\Xfr$.\footnote{c.f., Remark \ref{hf893hf83h409fj3}.}
\end{THM}

\noindent
The contents of Theorem \ref{rfg784gf7hf98h38f03} is contained in \cite[Theorem 2.5, p. 14 (arXiv)]{DW2}. In what remains of this section, we will give a proof of Theorem \ref{rfg784gf7hf98h38f03} distinct from the argument by Donagi and Witten in the afore-cited article. We proceed as follows, starting with the decomposition of the Atiyah space.

\begin{LEM}\label{rhf9hf983hf08309f3}
By abuse of notation we will denote by $\iota$ the embeddings $X\subset \Xfr$ and $X\subset \widehat\Xfr$. On $X$ then, there exists an isomorphism of Atiyah spaces $\iota^\sharp \mathrm{AT}~T_\Xfr \cong \iota^\sharp \mathrm{AT}~T_{\widehat\Xfr}$. 
\end{LEM}

\begin{proof}
Note that $\iota^\sharp T_\Xfr[\![m]\!] = \iota^\sharp T^*_\Xfr[\![m]\!] = (0)$ for all $m> 0$ and any supermanifold $\Xfr$.\footnote{This is because, on functions, $\iota^\sharp : \Oc_\Xfr \ra \Oc_\Xfr/\Jc_\Xfr$ is reduction modulo the fermionic ideal $\Jc_\Xfr$. Then for any $\Oc_\Xfr$-module $\Mcl$ see that $\Mcl[\![m]\!] = \Jc^m_\Xfr\Mcl$ so therefore $\forall m>0,\iota^\sharp \Mcl[\![m]\!] = (0)$.} Hence under $\iota^\sharp$ the diagram in Lemma \ref{jfbkebvuiebueoroinoifr} will comprise isomorphisms and so, as sheaves, $\iota^\sharp \mathcal At~T_\Xfr \cong \iota^\sharp \mathcal At~T_{\widehat\Xfr}$. The lemma now follows.
\end{proof}

\noindent
Since $\mathrm{OB}^{primary}_\Xfr\cong \mathrm{OB}^{primary}_{\widehat\Xfr}$ by Proposition \ref{rffg874g97h9fh83}, we recover \eqref{rhf8944333r3hf89h3f30}. Furthermore, the argument in Proposition \ref{rfh894hf89hf0j903} now applies, identifying the component of $\iota^\sharp\mathrm{at}~T_\Xfr$ in the Atiyah spaces $\mathrm{AT}~T_X\oplus \mathrm{AT}~T_{X, -}$ with the Atiyah classes $\mathrm{at}~T_X\oplus \mathrm{at}~T_{X, -}$. It remains to deduce that the component in the primary obstruction space will be the primary obstruction class. 

\subsection{Proof Outline}
We return now to the affine Atiyah sequence for $\Xfr$ which, following Theorem \ref{rgf783gf793hf983h}, we recall:
\[
\xymatrix{
0\ar[r] & \mathcal At~T_\Xfr \ar[r] & \widetilde{J^1T_\Xfr} \ar[r] & \mathcal End_{\Oc_\Xfr}T_\Xfr\ar[r] & 0
}
\]
where $\widetilde{J^1T_\Xfr} = \mathcal Hom_{\Oc_\Xfr}(T_\Xfr, J^1T_\Xfr)$. Our method for completing the proof of Theorem \ref{rfg784gf7hf98h38f03} will be to firstly obtain a `correspondent sequence' between the affine Atiyah sequence above and that for $\Gc^{(2)}_{\Oc_\Xfr}$ from Proposition \ref{fh94hf89hf0j09f}. That is, we will form a morphism of sequences:
\begin{align}
\xymatrix{
&\ar[dd]|{\hole} \Gc^{(2)}_{\Oc_\Xfr} \ar[rr]&  & \Gc^{(1)}_{\Oc_\Xfr} \ar[dd]|{\hole} \ar[rr] & & \mathcal Aut_{\Oc_X}T^*_{X,-}\ar[dd]
\\ 
(\mathcal At~T_\Xfr)_+\ar[dr] \ar[rr] & & (\widetilde{J^1T_\Xfr})_+ \ar[dr] \ar[rr] & & (\mathcal End~T_\Xfr)_+\ar[dr] 
\\
& \mathcal Ob^{primary}_{\Oc_\Xfr} \ar[rr] & & \mathcal H \ar[rr] & & \mathcal End_{\Oc_X}T^*_{X, -}
}
\label{rhf83h0f3f9j39fj39}
\end{align}
where we have identified $T^*_{X, -} = \Jc_\Xfr/\Jc_\Xfr^2$. The idea now is that on cohomology we obtain a commutative diagram:
\begin{align}
\xymatrix{
&H^0(\mathcal Aut_{\Oc_X}T^*_{X, -})\ar[dd]|{\hole} \ar[rr] & & \mbox{\v H}^1(\Gc_{\Oc_\Xfr}^{(2)})\ar[dd]
\\
H^0((\mathcal End_{\Oc_\Xfr}T_\Xfr)_+)\ar[dr] \ar[rr] && (\mathrm{AT}~T_\Xfr)_+\ar[dr] & 
\\
& H^0(\mathcal End_{\Oc_X}T^*_{X, -})\ar[rr] & & \mathrm{OB}^{primary}_\Xfr
}
\label{djkbhdbfjvbdbbcioen}
\end{align}
Commutativity of each square above will guarantee that the projection of the Atiyah class onto the component in the primary obstruction space will  coincide precisely with the primary obstruction class. It remains therefore to construct a morphism of sequences as in \eqref{rhf83h0f3f9j39fj39}.

\subsection{The Automorphism Sheaves}
Recall from Theorem \ref{rfh794f9hf8h380f}\emph{(iii)} and  \emph{(iv)} that for each $m$, $\Gc^{(m+1)}_{\Oc_\Xfr}\subset \Gc^{(m)}_{\Oc_\Xfr}$ is normal has abelian quotient. In analogy with Green in \cite{GREEN} in the split case, we have:

\begin{LEM}\label{rfh984f983hf08h30}
For any $m>1$ there is an isomorphism of abelian sheaves,
\begin{align*}
\frac{\Gc^{(m)}_{\Oc_\Xfr}}{\Gc^{(m+1)}_{\Oc_\Xfr}}
\cong 
\left|
\begin{array}{ll}
\mathcal Hom_{\Oc_\Xfr}\big(\Oc_\Xfr/\Jc_\Xfr, \Jc^m_\Xfr/\Jc_\Xfr^{m+1}\big) & \mbox{$m$ is even};
\\
\mathcal Hom_{\Oc_\Xfr}\big(\Jc_\Xfr/\Jc^2_\Xfr, \Jc^m_\Xfr/\Jc_\Xfr^{m+1}\big) & \mbox{$m$ is odd}.
\end{array}
\right.
\end{align*}
\end{LEM}

\begin{proof}
Note that the automorphisms in $\Gc^{(m)}_{\Oc_\Xfr}$ are even. Hence, the inclusion of the quotient $\Gc^{(m)}_{\Oc_\Xfr}/\Gc^{(m+1)}_{\Oc_\Xfr}\subset \mathcal Hom_{\Oc_\Xfr}(\Oc_\Xfr, \Jc_\Xfr^m/\Jc_\Xfr^{m+1})$ from \eqref{rfh84f409fj43jf4443r4} will lie in the even component. With the global $\Zbb_2$-grading $\Oc_\Xfr\cong \Oc_{\Xfr, +}\oplus \Oc_{\Xfr, -}$ we will use, for any \emph{even} morphism $\om: \Oc_\Xfr \ra \Jc^m_{\Xfr}$, that:
\begin{align}
\om|_{\Oc_{\Xfr, -}} \equiv 0\mod \Jc_\Xfr^{m+1}&&\mbox{if $m$ is even};
\label{ehf893hf983h8fh390}
\\
\om|_{\Oc_{\Xfr, +}} \equiv 0\mod \Jc_\Xfr^{m+1}&&\mbox{if $m$ is odd.~}
\label{efh893hf83f09j33f3}
\end{align}
Now from the sequence $0\ra \Jc_\Xfr\ra \Oc_\Xfr\ra\Oc_\Xfr/\Jc_\Xfr \ra0$ see that we have on hom-sheaves and by contravariance,
\begin{align}
0
\lra 
\mathcal Hom_{\Oc_\Xfr}&\left(\frac{\Oc_\Xfr}{\Jc_\Xfr}, \frac{\Jc^m_\Xfr}{\Jc_\Xfr^{m+1}}\right)
\notag\\
\lra &~
\mathcal Hom_{\Oc_\Xfr} \left(\Oc_\Xfr,  \frac{\Jc^m_\Xfr}{\Jc_\Xfr^{m+1}}\right)
\lra
\mathcal Hom_{\Oc_\Xfr} \left(\Jc_\Xfr,  \frac{\Jc^m_\Xfr}{\Jc_\Xfr^{m+1}}\right)
\lra 
0.
\label{hjhhjvjhcbhjecr}
\end{align}
Recall $\Gc^{(m)}_{\Oc_\Xfr}/\Gc^{(m+1)}_{\Oc_\Xfr}\subset \mathcal Hom_{\Oc_\Xfr}(\Oc_\Xfr, \Jc_\Xfr^m/\Jc_\Xfr^{m+1})$ from \eqref{rfh84f409fj43jf4443r4}. Using $\Oc_\Xfr\cong \Oc_{\Xfr, +}\oplus \Oc_{\Xfr, -}$ in addition to $\Oc_\Xfr/\Jc_\Xfr = \Oc_{\Xfr, +}/(\Jc_\Xfr\cap \Oc_{\Xfr, +})$ see that, if $m$ is even, then by \eqref{ehf893hf983h8fh390} the image of $\Gc^{(m)}_{\Oc_\Xfr}/\Gc^{(m+1)}_{\Oc_\Xfr}$ under the mapping in \eqref{hjhhjvjhcbhjecr} will vanish. Thus when $m$ is even we have $\Gc_{\Oc_\Xfr}^{(m)}/\Gc_{\Oc_\Xfr}^{(m+1)}\cong \mathcal Hom_{\Oc_\Xfr}\left(\frac{\Oc_\Xfr}{\Jc_\Xfr}, \frac{\Jc^m_\Xfr}{\Jc_\Xfr^{m+1}}\right)$ by exactness. When $m$ is odd, \eqref{efh893hf83f09j33f3} will guarantee $\Gc_{\Oc_\Xfr}^{(m)}/\Gc_{\Oc_\Xfr}^{(m+1)}\cap~ \img~\mathcal Hom_{\Oc_\Xfr}\left(\frac{\Oc_\Xfr}{\Jc_\Xfr}, \frac{\Jc^m_\Xfr}{\Jc_\Xfr^{m+1}}\right) = \eset$, so we can conclude $\Gc_{\Oc_\Xfr}^{(m)}/\Gc_{\Oc_\Xfr}^{(m+1)}\subset \mathcal Hom_{\Oc_\Xfr} \left(\Jc_\Xfr,  \frac{\Jc^m_\Xfr}{\Jc_\Xfr^{m+1}}\right)$. Now consider the sequence $\Jc_\Xfr^2\ra \Jc_\Xfr\ra \Jc_\Xfr/\Jc_\Xfr^2$ giving, contravariantly, the sequence of hom-sheaves,
\begin{align}
0
\lra 
\mathcal Hom_{\Oc_\Xfr}&\left(\frac{\Jc_\Xfr}{\Jc^2_\Xfr}, \frac{\Jc^m_\Xfr}{\Jc_\Xfr^{m+1}}\right)
\notag\\
\lra &~
\mathcal Hom_{\Oc_\Xfr} \left(\Jc_\Xfr,  \frac{\Jc^m_\Xfr}{\Jc_\Xfr^{m+1}}\right)
\lra
\mathcal Hom_{\Oc_\Xfr} \left(\Jc^2_\Xfr,  \frac{\Jc^m_\Xfr}{\Jc_\Xfr^{m+1}}\right)
\lra 
0.
\label{hjhhj44444vjhcbhjecr}
\end{align}
The image of $\Gc^{(m)}_{\Oc_\Xfr}/\Gc^{(m+1)}_{\Oc_\Xfr}$ under the mapping \eqref{hjhhj44444vjhcbhjecr} will vanish so therefore by exactness again we have, when $m$ is odd, $\Gc^{(m)}_{\Oc_\Xfr}/\Gc^{(m+1)}_{\Oc_\Xfr} \cong \mathcal Hom_{\Oc_\Xfr}\left(\frac{\Jc_\Xfr}{\Jc^2_\Xfr}, \frac{\Jc^m_\Xfr}{\Jc_\Xfr^{m+1}}\right)$.
\end{proof}

\noindent
Theorem \ref{rfh794f9hf8h380f}\emph{(iii)} and \emph{(iv)} can be generalised to the following.

\begin{LEM}\label{rfg783gf73gf983hf9}
For any $m>0$, the embedding $\Gc^{(m+2)}_{\Oc_\Xfr}\subset \Gc^{(m)}_{\Oc_\Xfr}$ is normal and the quotient $\Gc^{(m)}_{\Oc_\Xfr}/\Gc^{(m+2)}_{\Oc_\Xfr}$ is abelian.
\end{LEM}

\begin{proof}
A subgroup $H< G$ is normal with abelian quotient if $H$ contains the commutator $[G, G]$. We refer to \cite{BETTEMB} for an argument justifying $[\Gc^{(m)}_{\Oc_\Xfr}, \Gc^{(m)}_{\Oc_\Xfr}]\subset \Gc_{\Oc_\Xfr}^{(m+2)}$, from which the lemma follows.
\end{proof}

\noindent
Accepting normality of $\Gc^{(m+2)}_{\Oc_\Xfr}\subset \Gc^{(m)}_{\Oc_\Xfr}$, another way to see that the quotient will be abelian is by inspecting the sequence \eqref{rfh84f409fj43jf4443r4} in the proof of Theorem \ref{rfh794f9hf8h380f}\emph{(iv)}. With $\Gc^{(m)}_{\Oc_\Xfr}\subset \mathcal End_{\Oc_\Xfr}\Oc_\Xfr$, we can form the formal logarithm as in \eqref{rfh9f983hf83hf03}. For any $g\in \Gc^{(m)}_{\Oc_\Xfr}$ and $u\in \Oc_\Xfr$ we have $(g - {\bf 1})^2(u)\in \Jc^{2m}_{\Xfr}\subseteq \Jc_{\Xfr}^{m+2}$ for any $m> 0$. Hence $\Gc^{(m+2)}_{\Oc_\Xfr} \subset \ker\{ \Gc^{(m)}_{\Oc_\Xfr} \ra \mathcal Hom_{\Oc_\Xfr}(\Oc_\Xfr/\Jc_{\Xfr}^{m+2}, \Jc^m_{\Xfr}/\Jc^{m+2}_{\Xfr})\}$ and so, as in the proof of Theorem \ref{rfh794f9hf8h380f}\emph{(iv)}, we find:
\begin{align}
\frac{\Gc^{(m)}_{\Oc_\Xfr}}{ \Gc^{(m+2)}_{\Oc_\Xfr}}
\subset 
\mathcal Hom_{\Oc_\Xfr}\left({\Oc_\Xfr}, \frac{\Jc_\Xfr^m}{\Jc_\Xfr^{m+2}}\right).
\label{rfh794f983hf803f0}
\end{align}
This shows the quotient will be abelian. Arguing as in the proof of Lemma \ref{rfh984f983hf08h30}, the inclusion in \eqref{rfh794f983hf803f0} can be refined to the following:

\begin{LEM}
\label{rjf894hf84f4jf0jf3}
\[
\frac{\Gc^{(m)}_{\Oc_\Xfr}}{ \Gc^{(m+2)}_{\Oc_\Xfr}}
\subset 
\mathcal Hom_{\Oc_\Xfr}\left(\frac{\Oc_\Xfr}{\Jc^2_\Xfr}, \frac{\Jc_\Xfr^m}{\Jc_\Xfr^{m+2}}\right).
\]
\end{LEM}

\begin{proof}
Since any element of $\Gc^{(m)}_{\Oc_\Xfr}$ will map $\Jc_\Xfr^2 \ra \Jc_\Xfr^{m+2}$, the composition $\Gc^{(m)}_{\Oc_\Xfr} \ra \mathcal Hom_{\Oc_\Xfr}(\Oc_\Xfr, \Jc_\Xfr^m) \ra \mathcal Hom_{\Oc_\Xfr}(\Jc^2_\Xfr, \Jc_\Xfr^m)$ vanishes and so vanishes modulo $\Jc_\Xfr^{m+2}$. Hence we have a commuting diagram of inclusions,
\[
\xymatrix{
& & \mathcal Hom_{\Oc_\Xfr}\left(\frac{\Oc_\Xfr}{\Jc^2_\Xfr}, \frac{\Jc^m_\Xfr}{\Jc^{m+2}_\Xfr}\right)\ar@{^{(}->}[d]
\\
\Gc^{(m)}_{\Oc_\Xfr}/ \Gc^{(m+1)}_{\Oc_\Xfr}
 \ar@{^{(}->}[rr] \ar@{.>}[urr] & &\mathcal Hom_{\Oc_\Xfr}\left(\Oc_\Xfr, \frac{\Jc^m_\Xfr}{\Jc^{m+2}_\Xfr}\right).
}
\]
This lemma now follows.
\end{proof}

\noindent
Using $\Gc^{(1)}_{\Oc_\Xfr}/\Gc^{(2)}_{\Oc_\Xfr}\cong \mathcal Aut_{\Oc_X}T^*_{X, -} \subset \mathcal End_{\Oc_X}T^*_{X, -}$ by Proposition \ref{fh94hf89hf0j09f}; Lemma \ref{rfh984f983hf08h30} with $m = 2$; and Lemma \ref{rjf894hf84f4jf0jf3} above, we obtain a morphism of abelian sheaves:
\begin{align}
\xymatrix{
0 \ar[r] & \Gc_{\Oc_\Xfr}^{(2)}/\Gc_{\Oc_\Xfr}^{(3)} \ar[d]_\cong \ar[r] & \Gc^{(1)}_{\Oc_\Xfr}/ \Gc^{(3)}_{\Oc_\Xfr}\ar@{^{(}->}[d] \ar[r] &  \Gc^{(1)}_{\Oc_\Xfr}/ \Gc^{(2)}_{\Oc_\Xfr}\ar[r] \ar@{^{(}->}[d] & 0
\\
& \mathcal Hom_{\Oc_\Xfr}\left(\frac{\Oc_\Xfr}{\Jc_\Xfr}, \frac{\Jc^2_\Xfr}{\Jc_\Xfr^3}\right) \ar[r] & \mathcal Hom_{\Oc_\Xfr}\left(\frac{\Oc_\Xfr}{\Jc^2_\Xfr}, \frac{\Jc_\Xfr}{\Jc_\Xfr^3}\right) \ar[r] & \mathcal End_{\Oc_X} T^*_{X,-} & 
}
\label{hf89rhf98hf083h03}
\end{align}

\subsection{An Exact Lattice}
Unlike the top row in \eqref{hf89rhf98hf083h03}, the bottom row will not be exact. Indeed, for any $m>0$, using the short exact sequences:
\begin{align*}
0\lra \frac{\Jc_\Xfr}{\Jc^2_\Xfr}\lra \frac{\Oc_\Xfr}{\Jc_\Xfr^2}\lra \frac{\Oc_\Xfr}{\Jc_\Xfr}\lra0
&&
\mbox{and}
&&
0
\lra
\frac{\Jc^{m+1}_\Xfr}{\Jc^{m+2}_\Xfr}
\lra 
\frac{\Jc_\Xfr^m}{\Jc_\Xfr^{m+2}}
\lra
\frac{\Jc^{m}_\Xfr}{\Jc_\Xfr^{m+1}}
\lra
0
\end{align*}
we can form the $9$-term, exact lattice:\footnote{an exact lattice is a two-dimensional array of morphisms whose rows and columns are exact}
\begin{align}
\xymatrix{
\mathcal Hom_{\Oc_\Xfr}\left(\frac{\Oc_\Xfr}{\Jc_\Xfr}, \frac{\Jc_\Xfr^{m+1}}{\Jc_\Xfr^{m+2}}\right)\ar[d]  \ar[r] \ar@{.>}[dr] & 
\mathcal Hom_{\Oc_\Xfr}\left(\frac{\Oc_\Xfr}{\Jc^2_\Xfr}, \frac{\Jc_\Xfr^{m+1}}{\Jc_\Xfr^{m+2}}\right) \ar[r]\ar[d]  & \mathcal Hom_{\Oc_\Xfr}\left(\frac{\Jc_\Xfr}{\Jc^2_\Xfr}, \frac{\Jc_\Xfr^{m+1}}{\Jc_\Xfr^{m+2}}\right)\ar[d] 
\\
\mathcal Hom_{|Oc_\Xfr}\left(\frac{\Oc_\Xfr}{\Jc_\Xfr}, \frac{\Jc_\Xfr^m}{\Jc_\Xfr^{m+2}}\right)\ar[d] \ar[r]
& 
\mathcal Hom_{\Oc_\Xfr}\left(\frac{\Oc_\Xfr}{\Jc^2_\Xfr}, \frac{\Jc_\Xfr^{m}}{\Jc_\Xfr^{m+2}}\right)
\ar[r]\ar[d]
\ar@{.>}[dr]
& 
\mathcal Hom_{\Oc_\Xfr}\left(\frac{\Jc_\Xfr}{\Jc^2_\Xfr}, \frac{\Jc_\Xfr^{m}}{\Jc_\Xfr^{m+2}}\right)
\ar[d]
\\
\mathcal Hom\left(\frac{\Oc_\Xfr}{\Jc_\Xfr}, \frac{\Jc_\Xfr^m}{\Jc_\Xfr^{m+1}}\right)
\ar[r]
& 
\mathcal Hom_{\Oc_\Xfr}\left(\frac{\Oc_\Xfr}{\Jc^2_\Xfr}, \frac{\Jc_\Xfr^{m}}{\Jc_\Xfr^{m+1}}\right)
\ar[r]
& 
\mathcal Hom_{\Oc_\Xfr}\left(\frac{\Jc_\Xfr}{\Jc^2_\Xfr}, \frac{\Jc_\Xfr^{m}}{\Jc_\Xfr^{m+1}}\right)}
\label{rjf89hf083f09j30}
\end{align}
For $m = 1$ and with the identification $T^*_{X, -} = \Jc_\Xfr/\Jc_\Xfr^2$, the diagonal morphisms above, indicated by the dotted arrows, are precisely the morphisms in the bottom row of \eqref{hf89rhf98hf083h03}. Note that by \eqref{rfg783gf793hf893hf9}, the sheaf in the top-left  corner of \eqref{rjf89hf083f09j30} is the primary obstruction sheaf $\mathcal Ob^{primary}_{\Oc_\Xfr}$. Regarding the middle term in \eqref{rjf89hf083f09j30} we have the following.

\begin{LEM}
\label{rhf93hf983hf083hf03j0}
For any $m$,
\begin{align*}
\mathcal Hom_{\Oc_\Xfr}\left(\frac{\Oc_\Xfr}{\Jc^2_\Xfr}, \frac{\Jc_\Xfr^{m}}{\Jc_\Xfr^{m+2}}\right)
\cong~&
\mathcal Hom_{\Oc_\Xfr}\left(\frac{\Oc_\Xfr}{\Jc_\Xfr}, \frac{\Jc_\Xfr^{m+1}}{\Jc_\Xfr^{m+2}}\right)\oplus \mathcal Hom_{\Oc_\Xfr}\left(\frac{\Jc_\Xfr}{\Jc^2_\Xfr}, \frac{\Jc_\Xfr^{m+1}}{\Jc_\Xfr^{m+2}}\right)
\\
\oplus~&
\mathcal Hom\left(\frac{\Oc_\Xfr}{\Jc_\Xfr}, \frac{\Jc_\Xfr^m}{\Jc_\Xfr^{m+1}}\right)
\oplus 
\mathcal Hom_{\Oc_\Xfr}\left(\frac{\Jc_\Xfr}{\Jc^2_\Xfr}, \frac{\Jc_\Xfr^{m}}{\Jc_\Xfr^{m+1}}\right)
\end{align*}
\end{LEM}

\begin{proof}
Recall that $\Oc_\Xfr/\Jc_\Xfr^2\cong (\Oc_\Xfr/\Jc_\Xfr)\oplus(\Jc_\Xfr/\Jc_\Xfr^2)$. This gives
\begin{align}
\mathcal Hom_{\Oc_\Xfr}&\left(\frac{\Oc_\Xfr}{\Jc^2_\Xfr}, \frac{\Jc_\Xfr^{m}}{\Jc_\Xfr^{m+2}}\right)
\notag
\\
&\cong 
\mathcal Hom_{\Oc_\Xfr}\left(\frac{\Oc_\Xfr}{\Jc_\Xfr}, \frac{\Jc_\Xfr^{m}}{\Jc_\Xfr^{m+2}}\right)
\oplus 
\mathcal Hom_{\Oc_\Xfr}\left(\frac{\Jc_\Xfr}{\Jc^2_\Xfr}, \frac{\Jc_\Xfr^{m}}{\Jc_\Xfr^{m+2}}\right).
\label{rfh973hf983hf0830}
\end{align}
Regarding the ideal quotient $\Jc_\Xfr^m/\Jc_\Xfr^{m+2}$, assuming $\Jc_\Xfr$ is flat over $\Oc_\Xfr$ allows for the folllowing deductions, which makes use again of the decomposition $\Oc_\Xfr/\Jc_\Xfr^2\cong (\Oc_\Xfr/\Jc_\Xfr)\oplus(\Jc_\Xfr/\Jc_\Xfr^2)$:
\begin{align}
\frac{\Jc_\Xfr^m}{\Jc_\Xfr^{m+2}}
= \frac{\Jc^m_\Xfr}{\Jc_\Xfr^2\Jc_\Xfr^m}
&\cong \Jc_\Xfr^m\otimes_{\Oc_\Xfr} \frac{\Oc_\Xfr}{\Jc_\Xfr^2}
\notag
\\
&\cong 
\Jc^m_\Xfr\otimes_{\Oc_\Xfr}\left( \frac{\Oc_\Xfr}{\Jc_\Xfr} \oplus \frac{\Jc_\Xfr}{\Jc_\Xfr^2}\right)
\notag
\\
&\cong 
\frac{\Jc_\Xfr^m}{\Jc_\Xfr^{m+1}}
\oplus 
\frac{\Jc^{m+1}_\Xfr}{\Jc^{m+2}_\Xfr}.
\label{rhf89e3e3e334hf98hffj43jp}
\end{align}
The lemma follows upon substituting \eqref{rhf89e3e3e334hf98hffj43jp} into \eqref{rfh973hf983hf0830}. 
\end{proof}

\subsection{The Affine Atiyah Sequence}
An important consequence of the decomposition in Lemma \ref{rhf93hf983hf083hf03j0} is the following.

\begin{LEM}\label{rfg63gf873gf9h38f03}
With $\iota: X\subset \Xfr$ the embedding of the reduced space, we have a morphism,
\[
\xymatrix{
0\ar[r] & \iota^\sharp (\mathcal At~T_\Xfr)_+ \ar[r] \ar[d] & \iota^\sharp( \widetilde{J^1T_\Xfr})_+\ar[r] \ar[d] & \iota^\sharp (\mathcal End_{\Oc_\Xfr}T_\Xfr)_+\ar[r]\ar[d]  & 0
\\
& \mathcal Hom_{\Oc_\Xfr}\left(\frac{\Oc_\Xfr}{\Jc_\Xfr}, \frac{\Jc^2_\Xfr}{\Jc_\Xfr^3}\right) \ar[r] & \mathcal Hom_{\Oc_\Xfr}\left(\frac{\Oc_\Xfr}{\Jc^2_\Xfr}, \frac{\Jc_\Xfr}{\Jc_\Xfr^3}\right) \ar[r] & \mathcal End_{\Oc_X} T^*_{X,-} &
}
\]
\end{LEM}

\begin{proof}
By Lemma \ref{rhf9hf983hf08309f3} we know that $\iota^\sharp \mathcal At~T_\Xfr \cong \iota^\sharp\mathcal At~T_{\widehat\Xfr}$; from \eqref{rgf673gf73h09j3f33434} that $\mathcal Ob^{primary}_{\Oc_{\widehat\Xfr}}\subset \iota^\sharp\mathcal At~T_{\widehat\Xfr}$; from the proof of Proposition \ref{rffg874g97h9fh83} that $\mathcal Ob^{primary}_{\Oc_{\widehat\Xfr}}\cong \mathcal Ob^{primary}_{\Oc_{\Xfr}}$; and finally by \eqref{rfg783gf793hf893hf9} that $\mathcal Ob^{primary}_{\Oc_{\Xfr}}\cong \mathcal Hom_{\Oc_\Xfr}\left(\frac{\Oc_\Xfr}{\Jc_\Xfr}, \frac{\Jc_\Xfr^2}{\Jc_\Xfr^3}\right)$. Hence $\mathcal Hom_{\Oc_\Xfr}\left(\frac{\Oc_\Xfr}{\Jc_\Xfr}, \frac{\Jc_\Xfr^2}{\Jc_\Xfr^3}\right)$ will be a direct summand of $\iota^\sharp (\mathcal At~T_\Xfr)_+$ and, accordingly, we can define a projection $\iota^\sharp (\mathcal At~T_\Xfr)_+\ra \mathcal Hom_{\Oc_\Xfr}\left(\frac{\Oc_\Xfr}{\Jc_\Xfr}, \frac{\Jc_\Xfr^2}{\Jc_\Xfr^3}\right)$. Similarly, $\mathcal End_{\Oc_X}T^*_{X, -}$ will be a direct summand of $\iota^\sharp (\mathcal End_{\Oc_\Xfr}T_\Xfr)_+$ leading thereby to the projection $\iota^\sharp (\mathcal End_{\Oc_\Xfr}T_\Xfr)_+\ra \mathcal End_{\Oc_X}T^*_{X, -}$. Finally, using that $\widetilde{J^1T_\Xfr}\cong  (\mathcal At~T_\Xfr)_+\oplus (\mathcal End_{\Oc_\Xfr}T_\Xfr)_+$ as untwisted, $\Oc_\Xfr$-modules leads to $\iota^\sharp (\widetilde{J^1T_\Xfr})_+ \ra \mathcal Hom_{\Oc_\Xfr}\left(\frac{\Oc_\Xfr}{\Jc_\Xfr}, \frac{\Jc_\Xfr^2}{\Jc_\Xfr^3}\right)\oplus \mathcal End_{\Oc_X}T^*_{X, -}\subset \mathcal Hom_{\Oc_\Xfr}\left(\frac{\Oc_\Xfr}{\Jc^2_\Xfr}, \frac{\Jc_\Xfr}{\Jc_\Xfr^3}\right)$, where the latter inclusion comes from the decomposition in Lemma \ref{rhf93hf983hf083hf03j0} for $m = 1$. We therefore obtain the desired morphism in the statement of this lemma.
\end{proof}

\subsection{Proof Completion of Theorem $\ref{rfg784gf7hf98h38f03}$}
With \eqref{hf89rhf98hf083h03} and Lemma \ref{rfg63gf873gf9h38f03} see that we obtain a correspondent sequence between the affine Atiyah sequence and that of the automorphism sheaves as in \eqref{rhf83h0f3f9j39fj39}---this correspondent sequence is the bottom row in \eqref{hf89rhf98hf083h03}, or equivalently in Lemma \ref{rfg63gf873gf9h38f03}. As mentioned in the remarks following \eqref{hf89rhf98hf083h03}, this sequence need not be exact and hence we need-not-necessarily obtain the diagram on cohomology as in \eqref{djkbhdbfjvbdbbcioen}. That we will in fact obtain such a diagram is a consequence of the $9$-term, exact lattice in \eqref{rjf89hf083f09j30} specialised to $m = 1$. In this case, the boundary map $H^0(\mathcal End_{\Oc_X}T^*_{X, -})\ra \mathrm{OB}^{primary}_\Xfr$ can be defined by composing the top-horizontal boundary with the right-vertical boundary; or equivalently the bottom-horizontal boundary with the left-vertical boundary. Either way, with this composition we arrive at the diagram on cohomology in \eqref{djkbhdbfjvbdbbcioen}, commutativity of which justifies the decomposition of the Atiyah class in \eqref{kdnjkbjcvecbe}. Theorem \ref{rfg784gf7hf98h38f03} now follows.
\qed

\subsection{On the Split Model}
Recall in Proposition \ref{rfh894hf89hf0j903} that we obtained an expression for the Atiyah class of split models with the obstruction space projected out. With the general decomposition in Theorem \ref{rfg784gf7hf98h38f03} along with Theorem \ref{rfg873gf97hf893h8f03} which implies: \emph{for any split model $\widehat\Xfr$ that its primary obstruction satisfies $\eta_*[\widehat\Xfr] = 0$}; we can thus conclude: 

\begin{COR}
Let $\widehat\Xfr$ be a split model with reduced space $X$, embedding $\iota: X\subset \widehat\Xfr$; and odd cotangent bundle $T_{X, -}^*$. Then on $X$, 
\[
\iota^\sharp \mathrm{at}~T_{\widehat\Xfr} = \mathrm{at}~T_X\oplus \mathrm{at}~T_{X, -}^*.
\]\qed
\end{COR}

\noindent

\newpage
\addtocontents{toc}{\vspace{\normalbaselineskip}}
\section*{Concluding Remarks}

\subsection*{The Full Atiyah Class}
An immediate corollary of Donagi and Witten's decomposition in Theorem \ref{rfg784gf7hf98h38f03} is the following:

\begin{THMS}
Let $\Xfr$ be a complex supermanifold whose affine Atiyah class vanishes upon restriction to the reduced space. Then its reduced space and odd cotangent bundle admit global, holomorphic connections.\footnote{By `global holomorphic connection on the reduced space' it is meant a global, affine connection. If $\Xfr$ is modelled on $(X, T^*_{X, -})$, then $X$ is the reduced space and $T^*_{X, -}$ is its odd cotangent bundle.}\qed
\end{THMS}

\noindent  
For $(X, T^*_{X, -})$ the model for $\Xfr$, note in particular that the existence of global holomorphic connections on $T_X$ and $T^*_{X, -}$ will be insufficient to define a global, affine, even connection on $\Xfr$. This is for the reason that, even if $\iota^\sharp \mathrm{at}~T_\Xfr$ vanishes, $\Xfr$ need not be split.\footnote{c.f., the remarks following Definition \ref{rhf894hf89hf80h30}.} Regarding the full Atiyah class however, we know from Theorem \ref{rfh894hf89hf04hf0} that: \emph{if $\mathrm{at}~T_\Xfr =0$, then there will exist a global, even, affine connection on $\Xfr$}. Hence by Koszul's theorem (Theorem \ref{buie9h08d3jd33f3}) and the theorem stated above it follows:

\begin{THMS}
Let $\Xfr$ be a complex supermanifold with vanishing affine Atiyah class. Then:
\begin{enumerate}[(i)]
	\item the reduced space and odd cotangent bundle of $\Xfr$ admit global, holomorphic connections and;
	\item $\Xfr$ is split.
\end{enumerate}
\qed
\end{THMS}

\noindent
Part \emph{(ii)} of the above theorem strongly suggests the existence of a relation between the affine Atiyah class and the higher obstructions to splitting or, equivalently, the Euler differential. In the author's doctoral thesis \cite{BETTPHD} a relation along the lines of Theorem \ref{rfg784gf7hf98h38f03}, utilising thickenings\footnote{For more on thickenings in supergeometry, see \cite{BETTOBSTHICK}.} was proposed, being: \emph{suppose $\Xfr$ is endowed with a covering in which an $\ell$-th obstruction to splitting, $\eta^{(\ell)}$ is defined. Let $\iota^{(\ell)} : \widehat\Xfr^{(\ell)}\subset \Xfr$ be the associated thickening. Then $\iota^{(\ell)\sharp}\mathrm{at}~T_\Xfr = \mathrm{at}~T_X\oplus \mathrm{at}~T_{X, -}^*\oplus \eta^{(\ell)}$}. In a planned sequel to this article, tentatively containing a `Part IV', we will investigate the relation between the full affine Atiyah class, higher obstructions; and the Euler differential.\footnote{Note in particular that with such a relation, Koszul's splitting theorem will be immediate.} For completeness and for the interested reader we clarify, in Appendix \ref{rhf78gf79h38fh380}, the relation between the Euler differential and primary obstructions of complex supermanifolds, as ought to be expected from Theorem \ref{djvevcyibeon}.

\subsection*{Beyond Affine}
Classically, as established by Atiyah in \cite{ATCONN}, the Atiyah class of a vector bundle contains important information about the base manifold, being the Chern classes of that bundle. In this article we have seen how the affine Atiyah class, generalised to complex supermanifolds, will contain important information pertaining to the classification of its complex structure via splitting. The Atiyah class itself however can be generalised to any sheaf of modules on a space as explained in \cite[Tag 09DF]{STACKPROJ}; and to any supermanifold as in \cite[p. 161]{BRUZZO}. Given the results in affine case then, it is natural to speculate as to what kind of information will be captured by the Atiyah class of a sheaf of modules more generally. Where the complex structure is concerned we conclude with the following open question: 
\begin{align*}
\begin{array}{l}
\mbox{\emph{let $\Xfr$ be a complex supermanifold and $\Fc$ a sheaf of $\Oc_\Xfr$-modules.}}
\\
\mbox{\emph{Then what information, if any, pertaining to the classification}}
\\
\mbox{of \emph{$\Xfr$ is contained in the Atiyah class of $\Fc$?}}
\end{array}
\end{align*}

\newpage
\addtocontents{toc}{\vspace{\normalbaselineskip}}
\appendix
\numberwithin{equation}{section}

\section{Proof of Theorem $\ref{rburghurhoiejipf}$}
\label{rgf78gf7hf983hf30f9j3}

\noindent
Recall that an affine connection $\nabla$ on $\Xfr$ is the $\Cbb$-linear mapping of sheaves $\nabla: T_\Xfr \otimes T_\Xfr\ra T_\Xfr$. It is $\Oc_\Xfr$-linear in its first argument; and is an $\Oc_\Xfr$-derivation in its second argument. It can be represented locally by $\nabla \sim d + A$ where $d$ is the universal de Rham differential and $A\in \mathcal Hom_{\Oc_\Xfr}(T_\Xfr\otimes T_\Xfr, T_\Xfr)$.  Now for tensor squares we have the decomposition into symmetric 
$`\odot$'
and anti-symmetric 
$`\wedge$'
tensors, giving:
\begin{align}
\mathcal Hom_{\Oc_\Xfr}(T_\Xfr\otimes T_\Xfr, T_\Xfr)
\cong
\big(\odot^2 T_\Xfr^*
\oplus
\wedge^2T_\Xfr^*\big)\otimes T_\Xfr.
\label{rfhuief983h9fh3h0f3}
\end{align}
We need to argue the connection form $A$ will necessarily be valued in $\odot^2T^*_\Xfr\otimes T_\Xfr$.
This is easiest to see in local coordinates $(x|\q)$.
%
Since the coordinates $x$ and $\q$ differ in parity, so do their differentials. Hence for vector fields $H$ and $Z$ we have:
\begin{align}
(dx{\odot\over\wedge} dy)(H, Z)
&=
H(x)Z(y) \pm H(y)Z(x)~
\mbox{and};
\label{4fh893hf8h3f0j309jf3}
\\
(d\q{\odot\over\wedge} d\eta)(H, Z)
&=
H(\q)Z(\eta) \mp H(\eta)Z(\q)
\label{rf83f93hf8893f09f3}
\end{align}
where by $\odot\over \wedge$ it is meant $\odot$ and $\wedge$ respectively; and $(y|\eta)$ are other even-odd coordinates. 
Now let $A$ denote the connection form for $\nabla$. In coordinates we can write $A$ in tensor components,
\begin{align*}
A(x|\q)
&=
\al_{\mu\nu}^{~~\s}~
dx^\mu\otimes dx^\nu\otimes \frac{\pt}{\pt x^\s}
+
\al_{\mu\nu k}~
dx^\mu\otimes dx^\nu\otimes \frac{\pt}{\pt \q_k}
+
\al_{\mu }^{~j\s}~
dx^\mu\otimes d\q_j\otimes \frac{\pt}{\pt x^\s}
\\
&~+
\al_{\mu~~k}^{~j}~
dx^\mu\otimes d\q_j\otimes \frac{\pt}{\pt \q_k}
+
\al^{ij\s}~
d\q_i\otimes d\q_j\otimes \frac{\pt}{\pt x^\s}
+
\al^{ij}_{~~k}~
d\q_i\otimes d\q_j\otimes \frac{\pt}{\pt \q_k}
\end{align*}
In order to see $A$ will be symmetric in its differentials we will need to recall the following: for any two vector fields $H$ and $Z$, their bracket can be defined over the whole space by setting:
\[
[H, Z]  \stackrel{\Delta}{=} \nabla_HZ - \nabla_ZH + Tor.^\nabla(H, Z).
\]
The above expression generalises to supermanifolds by replacing the bracket by the graded-commutative bracket giving (c.f., Definition \ref{gf7498fh3f30jf39f0}(i)),
\begin{align}
[H, Z] \stackrel{\Delta}{=} \nabla_HZ - (-1)^{|Z|\cdot|H|} \nabla_ZH + Tor.^\nabla(H, Z)
\label{rfg7gf874gf93h98fh3fh3}
\end{align}
where $|Z|$ and $|H|$ refers to the $\Zbb_2$-parity of $Z$ and $H$ as vector fields on $\Xfr$. Now in a coordinate system $(x|\q)$, the frame for the tangent bundle $(\pt/\pt x|\pt/\pt\q)$ will be torsion free. Hence by \eqref{rfg7gf874gf93h98fh3fh3} we find:
\begin{align}
\nabla_{\frac{\pt}{\pt x^\rho}}\frac{\pt}{\pt x^\gam} - \nabla_{\frac{\pt}{\pt x^\gam}}\frac{\pt}{\pt x^\rho} 
&=
\left[\frac{\pt}{\pt x^\rho}, \frac{\pt}{\pt x^\gam}\right] = 0~\mbox{and};
\label{fvurygfug9h38f3}
\\
\nabla_{\frac{\pt}{\pt\q_k}}\frac{\pt}{\pt\q_l}
+
\nabla_{\frac{\pt}{\pt\q_l}}\frac{\pt}{\pt\q_k} 
&=
\left[
\frac{\pt}{\pt \q_k}, \frac{\pt}{\pt \q_l}
\right]
=0.
\label{fguyegf78hf98hf03}
\end{align}
And so where the connection form is concerned, we have:
\begin{align}
A(x|\q)\left(\frac{\pt}{\pt x^\rho}, \frac{\pt}{\pt x^\gam}\right)
&=
\odot^2A\left(\frac{\pt}{\pt x^\rho}, \frac{\pt}{\pt x^\gam}\right)
+
\wedge^2A\left(\frac{\pt}{\pt x^\rho}, \frac{\pt}{\pt x^\gam}\right)
\notag
\\
&=
\left(
\left(
\al_{\rho\gam}^{~~\s} + \al_{\gam\rho}^{~~\s}
\right)
+
\left(\al^{~~\s}_{\rho\gam} - \al^{~~\s}_{\gam\rho}
\right)
\right)
\otimes\frac{\pt}{\pt x^\s}
\notag
\\
&~
+
\big(
\big(\al_{\rho\gam k} + \al_{\gam\rho k}\big)
+
\big(
\al_{\rho\gam k} - \al_{\gam\rho k}
\big)
\big)
\otimes
\frac{\pt}{\pt \q_k}
\notag
\\
&
\stackrel{\scriptsize\mathrm{set}}{=} A(x|\q)\left(\frac{\pt}{\pt x^\gam}, \frac{\pt}{\pt x^\rho}\right)
&&
\mbox{(by \eqref{fvurygfug9h38f3})}
\notag
\\
&= 
\left(
\left(
\al_{\gam\rho}^{~~\s} + \al_{\rho\gam}^{~~\s}
\right)
+
\left(\al^{~~\s}_{\gam\rho} - \al^{~~\s}_{\rho\gam}
\right)
\right)
\otimes\frac{\pt}{\pt x^\s}
\notag
\\
&~+
\big(
\big(
\al_{\gam\rho k} + \al_{\rho\gam k}
\big)
+
\big(
\al_{\gam\rho k} - \al_{\rho\gam k}
\big)
\big)
\otimes
\frac{\pt}{\pt \q_k}
\notag
\\
\iff
\al^{~~\s}_{\gam\rho} - \al^{~~\s}_{\rho\gam}&=0~\mbox{and}
~\al_{\rho\gam k} - \al_{\gam\rho k} = 0;
\notag
\\
\iff 
\wedge^2A\left(\frac{\pt}{\pt x^\rho}, \frac{\pt}{\pt x^\gam}\right)  &=0
\label{4h73fh983hf030f9}
\end{align}
A similar calculation for the odd derivations, but making use \eqref{rf83f93hf8893f09f3} and \eqref{fguyegf78hf98hf03} now, will reveal:
\begin{align}
\odot^2A
\left(\frac{\pt}{\pt\q_k}, \frac{\pt}{\pt\q_l}\right) = 0.
\label{rhf893f89hf03f039f}
\end{align}
Viewed as a supermatrix\footnote{
Following convention in \cite{YMAN}, for a module $M = M_+\oplus M_-$ over a supercommutative ring $R$ we have: $\wedge^\bt M\cong \wedge^\bt_{R_+} M\otimes \odot^\bullet_{R_+}M_-$ and $\odot^\bt M\cong \odot^\bt_{R_+}M_+\otimes \wedge^\bullet_{R_+}M_-$.}
 then, the constraints in \eqref{4h73fh983hf030f9} and \eqref{rhf893f89hf03f039f} can be summarised by $\wedge^2A = 0$. Hence, by \eqref{rfhuief983h9fh3h0f3}, it follows that $A$ is valued in $\odot^2T^*_\Xfr\otimes T_\Xfr$. 
\qed

\newpage
\section{Cotangent Supermanifolds and Liftings}
\label{rgf7g7fg93gf083h03}

\subsection*{Onishchik's Characterisation}
Green's classification in Theorem \ref{rgf78g37fg39f3h80} applies to supermanifolds generally. Onishchik in \cite{ONISHNS} specialises the class of supermanifolds considered to the following kind, which we refer to as `cotangent supermanifolds'.

\begin{DEF}
\emph{Any supermanifold modelled on a complex manifold $X$ and its cotangent bundle $T^*_X$ will be referred to as a \emph{cotangent supermanifold}.}
\end{DEF}

\noindent
Let $(X,T_X^*)$ be a model for contangent supermanifolds and denote by $\widehat\Xfr$ the split model. By Lemma \ref{rhf79gf983hf80h03} we have an exact sequence for the tangent sheaf for each $m$,
\begin{align}
0
\lra 
\Om^{m+1}_X\otimes T_X
\stackrel{i}{\lra}
T^{\{m\}}_{\widehat\Xfr}
\lra
\Om^m_X\otimes T_X
\lra
0.
\label{rfh89hf830f9j3933434}
\end{align}
Now let $d_X: \Oc_X\ra \Om^1_X$ be the universal, K\"ahler derivation. Note that it defines a global vector field on the split model of degree one, i.e., $d^X\in H^0(X, T^{\{1\}}_{\widehat\Xfr})$. Accordingly, we will refer to $d^X$ as the \emph{de Rham vector field}. Note furthermore that the inclusion $i$ in \eqref{rfh89hf830f9j3933434} is defined for any $m$, i.e., that $i : \Om^m_X\otimes T_X\ra T^{\{m-1\}}_{\widehat\Xfr}$. With these observations and the (super) Lie bracket on $T_{\widehat\Xfr}$, we can form a mapping: $\ell: \Om^m_X \otimes T_X \ra T^{\{m\}}_X$ given by $\psi\stackrel{\ell}{\mapsto} [i(\psi), d^X]$. In contrast with the general case now, the mapping $\ell$ defines a holomorphic splitting of the sequence \eqref{rfh89hf830f9j3933434}. Hence, for any $m$, we have:
\begin{align}
T^{\{m\}}_{\widehat\Xfr}
\cong 
i(\Om^{m+1}_X\otimes T_X)
\oplus  
\ell (\Om^{m}_X\otimes T_X).
\label{rhf74gf93hf80h309}
\end{align}
With this splitting we have a mapping $\Om^n_X\otimes T^{\{m\}}_{\widehat\Xfr} \stackrel{\cong}{\ra} i(\Om^n_X\otimes \Om^{m+1}_X\otimes T_X)\oplus \ell (\Om^n_X\otimes \Om^{m}_X\otimes T_X) \ra i(\Om_X^{n+m+1}\otimes T_X)\oplus \ell(\Om^{n+m}_X\otimes T_X)\stackrel{\cong}{\ra} T^{\{m+n\}}_{\widehat\Xfr}$, where we used the wedge product on differential forms, $\otimes \mapsto \wedge$. Onishchik in \cite{ONISHNS} then studies the mapping $\mu: \Om^1_X \ra \Om^1_X\otimes T^{\{1\}}_{\widehat\Xfr}\ra T^{\{2\}}_{\widehat\Xfr}$ given by $\psi \mapsto \psi\otimes d\mapsto \psi d^X$.\footnote{Products here are wedge products, so $\psi d^X = \psi\wedge d^X$.} Now for any $m> 0$, the formal exponential of vector fields in $T^{\{m\}}_{\widehat\Xfr}$, for any $\widehat\Xfr$, will be a finite sum and defines a mapping of sheaves of sets $\exp: T^{\{m\}}_{\widehat\Xfr}\ra \mathcal G_{\Oc_{\widehat\Xfr}}^{(m)}$. Upon specialising now to the case where $\widehat\Xfr$ is a split, cotangent supermanifold with model $(X, T^*_X)$, we can compose with $\mu$ to get a mapping of sheaves of sets:
\begin{align}
\exp\mu: \Om^1_X \lra \mathcal G^{(2)}_{\widehat\Xfr}
&&
\mbox{given by}
&&
\psi \longmapsto \exp\psi d^X.
\label{rhf893hf8hf093j9}
\end{align}
Hence we have induced on cohomology $(\exp\mu)_*: H^1(X, \Om^1_X) \ra \mbox{\v H}^1(X, \mathcal G^{(2)}_{\Oc_{\widehat\Xfr}})$. Applying Green's classification in Theorem \ref{rgf78g37fg39f3h80} reveals, when $\widehat\Xfr$ is a split, cotangent supermanifold, that any element of $H^1(X, \Om^1_X)$ will define a cotangent supermanifold with $\widehat\Xfr$ as its split model.\footnote{That is, up to (framed) isomorphism.}
Onishchik in \cite[\S2]{ONISHNS} derived the following characterisation.

\begin{THM}\label{rgf74gff8h330f3j9}
Let $\om\in H^1(X,\Om^1_X)$ and denote by $\Xfr_\om \stackrel{\Delta}{=}(\exp\mu)_*(\om)$ the associated, cotangent supermanifold. If $\om\neq0$, $\Xfr_\om$ is non-split.
\qed
\end{THM}

\noindent
From the explicit construction in the proof of Theorem \ref{rgf74gff8h330f3j9} in \cite{ONISHNS}, the primary obstruction $\eta_*[\Xfr_\om]$ of the cotangent supermanfold $\Xfr_\om$ can be identified with $pr \mu_*(\om)$, where $\mu_*: H^1(X, \Om^1_X)\ra H^1(X, T^{\{2\}}_{\widehat\Xfr})$ is the induced mapping on cohomology; and $pr$ is the projection $H^1(X, T^{\{2\}}_{\widehat\Xfr}) \ra \mathrm{OB}^{primary}_{\widehat\Xfr} \cong H^1(X, \Om_X^2\otimes T_X)$. Note that, having established just this, Theorem \ref{rgf74gff8h330f3j9} will follow from the more general Theorem \ref{rfg7f9hf3h80f3}. With respect to a covering $(\Uc_{\al})_{\al}$, we can see from \eqref{rhf893hf8hf093j9} that a cocycle representative of the primary obstruction $\eta_*[\Xfr_\om]$ is given by $(\om_{\al\be}d^X)_{\al,\be}$, where $(\om_{\al\be})_{\al,\be}$ is a cocycle representation of the given class $\om$.

\subsection*{Relation to the Euler Vector Field}
\label{rfh894hf8hf0j39j343}
Observe that the sequence in \eqref{rfh89hf830f9j3933434} can continued further, giving:
\[
\xymatrix{
& & & 0\ar[d] &
\\
0\ar[r] & \Om^{m+1}_X\otimes T_X\ar[r] & T^{\{m\}}_{\widehat\Xfr} \ar@{.>}[dr] \ar[r] & \Om^m_X\otimes T_X\ar[r] \ar[d]& 0
\\
& & & T^{\{m-1\}}_{\widehat\Xfr}\ar[d]& 
\\
& & &\Om^{m-1}_X\otimes T_X\ar[d] & 
\\
& & & 0& 
}
\]
The composition $T^{\{m\}}_{\widehat\Xfr} \ra T^{\{m-1\}}_{\widehat\Xfr}$ indicated above is, by construction, non-zero. 
There is a natural isomorphism of sheaves of $\Oc_X$-modules, $T_X \stackrel{\cong}{\ra} T_{\widehat\Xfr}^{\{-1\}}$ given by $\pt/\pt x \mapsto \pt/\pt(dx)$. Evidently, the morphism $T_{\widehat\Xfr}^{\{m\}}\ra T_{\widehat\Xfr}^{\{m-1\}}$ is defined by reference to this isomorphism, i.e., is the composition $T_{\widehat\Xfr}^{\{m\}} \ra \Om^m_X\otimes T_X \stackrel{\cong}{\ra} \Om^m_X\otimes T^{\{-1\}}_{\widehat\Xfr} \subset T_{\widehat\Xfr}^{\{m-1\}}$.\footnote{Recall that in writing $T_{\widehat\Xfr} = \bigoplus_{m\geq-1} T_{\widehat\Xfr}^{\{m\}}$, we have decomposed $T_{\widehat\Xfr}$ into a $\Zbb$-graded sheaf of $\Oc_X$-modules with $\Om^j_X\otimes T_{\widehat\Xfr}^{\{j^\p\}} \subset  T^{\{j+j^\p\}}_{\widehat\Xfr}$.} For $m = 1$ we have the mapping $T^{\{1\}}_{\widehat\Xfr} \ra T^{\{0\}}_{\widehat\Xfr}$. Comparing local coordinate descriptions of the de Rham vector field and the Euler vector field in \eqref{g683f863g793893}, and using that they are global sections, it is evident that:
\begin{align}
d^X \longmapsto \e_{\widehat\Xfr},
\label{biuebviuboenvoien}
\end{align}
under the induced mapping $H^0(X, T_{\widehat\Xfr}^{\{1\}}) \ra H^0(X, T_{\widehat\Xfr}^{\{0\}})$.

\subsection*{Vector Field Liftings}
In Definition \ref{f73gf7g498fh308hf03} the notion of a vector field lift was adapted to the Euler vector field. Such a notion can be made sense of more generally however with the initial form sequence in Lemma \ref{rhf78f9hf083j09fj30}. That is:

\begin{DEF}
\emph{Let $\Xfr$ be a supermanifold with split model $\widehat\Xfr$. For a given $m$, any global vector field on $\widehat\Xfr$ of degree $m$, i.e., any element of $H^0(X, T_{\widehat\Xfr}^{\{m\}})$, is said to \emph{lift} to a vector field on $\Xfr$ if it is in the image of the mapping on global sections $H^0(X, T_\Xfr^{(m)}) \ra H^0(X, T_{\widehat\Xfr}^{\{m\}})$.}
\end{DEF}

\noindent
Our proof of Koszul's theorem in this article relied essentially on Theorem \ref{djvevcyibeon}, which identified liftings of the Euler vector field with splittings.\footnote{c.f., \S\ref{fiubeuifuefohroifjeof}.} Onishchik in \cite{ONISHNS} also studied the splitting problem, specialised to cotangent supermanifolds; and from the viewpoint of vector field liftings. We summarise these results in the following.

\begin{THM}\label{rfh9hf98hf083h0}
For $\om\in H^1(X,\Om^1_X)$ let $\Xfr_\om$ be the cotangent supermanifold from Theorem \ref{rgf74gff8h330f3j9}. Then:
\begin{enumerate}[(i)]
	\item the de Rham vector field $d^X\in H^0(X, T^{\{1\}}_{\widehat\Xfr})$ will lift to $\Xfr_\om$;
	\item Recall the splitting of $T_{\widehat\Xfr}^{\{m\}}$ in \eqref{rhf74gf93hf80h309} for any $m$. Specialising to $m = 0$, for any $u\in H^0(X, T_X)$, if 
	\begin{align}
	\ell(u)\smile \mu_*\om = 0,\label{rh73hf9h38fh380h}
	\end{align}
	then $\ell(u)$ lifts to $\Xfr_\om$.\footnote{Recall that on cohomology the cup product is the mapping $H^j(X, \Fc)\otimes H^{j^\p}(X, \Gc) \ra H^{j+j^\p}(X, \Fc\otimes_{\Oc_X} \Gc)$. The equation in \eqref{rh73hf9h38fh380h} is evidently in $H^1(X, T^{\{0\}}_{\widehat\Xfr}\otimes T^{\{2\}}_{\widehat\Xfr})$.}
\end{enumerate}
\qed
\end{THM}

\noindent
We wish to remark here that Theorem \ref{rfh9hf98hf083h0} is indeed consistent with, and independent of, Theorem \ref{djvevcyibeon}, which concerns the Euler vector field. Indeed, on a first reading, Theorem  \ref{rfh9hf98hf083h0}\emph{(i)}, the mapping in \eqref{biuebviuboenvoien} and Theorem \ref{djvevcyibeon} might jointly imply every cotangent supermanifold $\Xfr_\om$ splits, contradicting Theorem \ref{rgf74gff8h330f3j9}. But such a conclusion would be based on the error of forgetting the parity, i.e., how the vector fields act on functions. With $\widehat\Xfr$ the split, cotangent supermanifold modelled on $(X, T^*_X)$, its structure sheaf is $\Oc_{\widehat\Xfr} = \Om^\bt_X$. The de Rham vector field acts by $d^X : \Om^\bt_X \ra \Om^\bt[\![1]\!]$. In contrast, the Euler vector field acts as $\e_{\widehat\Xfr} : \Om^\bt_X \ra \Om^\bt_X[\![0]\!]$. Hence we cannot identify $d^X$ and $\e_{\widehat\Xfr}$ as global vector fields on $\widehat\Xfr$, despite the mapping in \eqref{biuebviuboenvoien}.\footnote{Note moreover, for any $f\in \Oc_X$ and $\om\in \Om^\bt_X$, that $d^X(f\om) = (d^Xf)\om + fd^x\om$; while by Lemma \ref{rbyurgf4gf233393hf89h3}\emph{(ii)}, $\e_{\widehat\Xfr}(f\om) = f\e_{\widehat\Xfr}(\om)$. That is, $\e_{\widehat\Xfr}$ is an $\Oc_X$-linear derivation of $\Om^\bt_X$ while $d^X$ is certainly not $\Oc_X$-linear.} In particular, for any cotangent supermanifold $\Xfr$ with split model $\widehat\Xfr$, the image of $d^X$ in $H^1(X, T^{(2)}_{\Xfr})$ will not correspond to the Euler differential.\footnote{Recall by Proposition \ref{fh7949v7h49f80fh} and Corollary \ref{tg74gf7h8f3j093} that the Euler differential will also be valued in  $H^1(X, T^{(2)}_{\Xfr})$.} Now with Theorem \ref{rfh9hf98hf083h0}\emph{(i)} taken to mean: \emph{the image of $d^X$ in $H^1(X, T_\Xfr^{(2)})$ vanishes for any $\Xfr$}; in stating that Theorem \ref{rfh9hf98hf083h0} is consistent with Theorem \ref{djvevcyibeon}, we observe that Theorem \ref{rfh9hf98hf083h0}\emph{(i)} need \emph{not} imply vanishing of the Euler differential.

\subsection*{Splitting Cotangent Supermanifolds}
The result on lifting in Theorem \ref{rfh9hf98hf083h0}\emph{(i)} is suggestive and we might wonder under what condition it will imply the existence of a supermanifold splitting. This is the subject of the following and might be viewed as an analogue of Theorem \ref{rfh9hf98hf083h0}\emph{(ii)} for the Euler vector field.

\begin{THM}\label{rhf89h38fh430hf03}
Let $\Xfr_\om$ be a cotangent supermanifold as in Theorem \ref{rgf74gff8h330f3j9} and suppose there exists a global, degree-$(-1)$ vector field $\gam\in H^0(X, T^{\{-1\}}_{\widehat\Xfr})$ such that,\footnote{in $(i)$, $[-,-]$ refers to the super-Lie bracket on $T_{\widehat\Xfr}$.}
\begin{enumerate}[(i)]
	\item $[\gam, d^X] = \e_{\widehat\Xfr}$ and;
	\item $\gam$ lifts to $\Xfr_\om$.
\end{enumerate}
Then $\Xfr_\om$ is split.
\end{THM}

\begin{proof}
Our proof makes use of the following claim: \emph{for vector fields $u, v\in H^0(X, T_{\widehat\Xfr})$, if $u$ and $v$ lift to some supermanifold $\Xfr$ with split model $\widehat\Xfr$, then so does their Lie bracket.} To see why this claim holds, recall that if $u$ and $v$ are homogeneous with parities $p(u)$ and $p(v)$ respectively, then $[u, v]\in H^0(X, T_{\widehat\Xfr}^{\{p(u) + p(v)\}})$. Note that also makes sense to form Lie brackets of vector fields on $\Xfr$. Now let $U$ and $V$ denote the respective liftings of $u$ and $v$ to $\Xfr$. Since $U$ and $V$ are globally defined, it suffices to consider them locally wherein we can write $U = u + u^{(p(u)+1)}$ and $V = v + v^{(p(v)+ 1}$, for sections $u^{(p(u)+1)}\in T_\Xfr^{(p(u)+1)}$ and $v^{(p(v)+ 1}\in T_\Xfr^{(p(v)+1)}$. Evaluating their bracket gives,
\begin{align}
[U,V]
=&~
[u, v]\notag
\\
&~+ 
[u, v^{(p(v)+1)}] + [u^{(p(u)+1)}, v] +[u^{(p(u)+1)}, v^{(p(v)+ 1}]
\label{rhf983hf983h0fh3h90}
\end{align}
See that \eqref{rhf983hf983h0fh3h90} is a section of $T_\Xfr^{(p(u) + p(v) + 1)}$. Hence $[U,V]\mod T_\Xfr^{(p(u) + p(v) + 1)} = [u, v]$, which proves the claim. Applying this to our case, recall from Theorem \ref{rfh9hf98hf083h0}\emph{(ii)} that $d^X$ will always lift to $\Xfr_\om$ and so, with our assumption in Theorem \ref{rhf89h38fh430hf03}\emph{(ii)} we know that $[\gam, d^X]$ will be a degree-zero vector field on $\widehat\Xfr$ which lifts to $\Xfr_\om$. Assuming Theorem \ref{rhf89h38fh430hf03}\emph{(i)}, we see that the Euler vector field lifts to $\Xfr_\om$ and hence its differential vanishes. This theorem now follows from Theorem \ref{djvevcyibeon}.
\end{proof}

\begin{REM}
\emph{Note that if we know Theorem \ref{rhf89h38fh430hf03}\emph{(i)} and \emph{(ii)}, then by Theorem \ref{rgf74gff8h330f3j9} it follows that $\om = 0$. Hence Theorem \ref{rhf89h38fh430hf03}\emph{(i)} and \emph{(ii)} might be viewed as `supergeometric conditions' guaranteeing the vanishing of cohomology classes in $H^1(X, \Om_X^1)$, or equivalently, the $(1,1)$-Dolbeault classes in $H^{1, 1}(X, \Cbb)$.}
\end{REM}

\section{On The Classification of Supermanifolds}
\label{rh894rhf894hf8f0j30}

\noindent
The title of this section coincides largely with that of a paper by Onishchik \cite{ONISHCLASS} and not coincidentally so. In \cite{ONISHCLASS}, Hodge-theoretic methods are adapted to establish the moduli problem for complex supermanifolds, following the set-theoretic underpinnings provided by Green in \cite{GREEN}. Similarly, we will also begin from the foundational work by Green and look to present a moduli-theoretic extension, with a view toward functoriality.

We begin by reviewing this work by Green below.

\subsection*{Green's Classification}
Let $(X, T^*_{X, -})$ be a model on which we can form a category of supermanifolds.\footnote{c.f., Proposition \ref{fh794f98hf03jf930}.}
Associated to this model is following sheaf of groups on $X$:
\begin{align}
\mathcal G^{(m)}_{(X, T^*_{X, -})}
\stackrel{\Delta}{=}
\left\{
\al\in (\mathcal Aut\wedge^\bt T^*_{X, -})_+
\mid 
\al(u) - u\in \bigoplus_{\ell\geq m}\wedge^\ell T^*_{X,-}
\right\}.
\label{noiebu4ifeiofpe}
\end{align}
The sheaf of groups $\Gc^{(m)}_{X, T^*_{X, -}}|_{m=2}$ is of particular importance and plays the central role in Greens's classification. In this case we have a short exact sequence of sheaves of groups on $X$,\footnote{
Green in \cite{GREEN} refers to automorphisms of \emph{connected sheaves} $\Ac$, which are $\Zbb$-graded sheaves with degree zero component $\Oc_X$. As automorphisms of such sheaves preserve $\Oc_X$ it follows, in the case where $\Ac = \wedge^\bt T^*_{X, -}$, that $\Gc^{(1)}_{(X, T^*_{X, -})}$ will be its sheaf of automorphisms.}
\begin{align}
\xymatrix{
\{{\bf 1}\} \ar[r] & \Gc_{(X, T^*_{X,-})}^{(2)} \ar[r] & \Gc_{(X, T^*_{X,-})}^{(1)} \ar[r] & \mathcal Aut_{\Oc_X}T^*_{X, -}\ar[r] & \{{\bf 1}\}
}
\label{fuirghf89h0fj309j3}
\end{align}
from whence we obtain a group action $H^0(X,\mathcal Aut_{\Oc_X}T^*_{X, -})$ on $\mbox{\v H}^1\big(X,\Gc_{(X, T^*_{X,-})}^{(2)}\big)$, an observation first documented by Grothendieck in \cite{GROTHNONAB} in a much more general setting. Green in \cite{GREEN} establishes:

\begin{THM}
\label{rgf78g37fg39f3h80}
(Green's Theorem)
The set of supermanifolds modelled on $(X, T^*_{X, -})$ up to isomorphism lies in bijective correspondence with the set of cosets:
\begin{align*}
\frac{\mbox{\emph{\v H}}^1\big(X,\Gc_{(X, T^*_{X,-})}^{(2)}\big)}{H^0(X,\mathcal Aut_{\Oc_X}T^*_{X, -})}.
\end{align*}
We will refer to the above orbit set as Green's orbit set.
\qed
\end{THM}

\noindent
The set $\mbox{\v H}^1\big(X,\Gc_{(X, T^*_{X,-})}^{(2)}\big)$ is itself a pointed set and the $H^0(X,\mathcal Aut_{\Oc_X}T^*_{X, -})$-orbit of this point defines the base-point in Green's orbit set from Theorem \ref{rgf78g37fg39f3h80} above. Under the correspondence with isomorphism classes of supermanifolds, this base-point corresponds to the isomorphism class of the split model associated to $(X, T^*_{X, -})$.

\subsection*{A Generalised Classification}
Green's classification in Theorem \ref{rgf78g37fg39f3h80} does not involve supermanifolds intrinsically, i.e., it is predicated on the exterior algebra. Recall however from Lemma \ref{e8y93hf8h3f93}\emph{(iii)} that the exterior algebra $\wedge^\bt_{\Oc_X} T^*_{X, -}$ forms the structure sheaf of the split model $\widehat\Xfr$. Indeed, the definition of the automorphism sheaves $\mathcal G^{(m)}_{(X, T^*_{X, -})}$ in \eqref{noiebu4ifeiofpe}, central to Green's classification, only requires the data of the structure sheaf and fermionic ideal. More generally therefore we can form, for any supermanifold $\Xfr$ with structure sheaf $\Oc_\Xfr$ and fermionic ideal $\Jc_\Xfr$, the sheaf of groups:
\begin{align}
\mathcal G_{\Oc_\Xfr}^{(m)}
\stackrel{\Delta}{=}
\big\{
\al\in (\mathcal Aut_{\Oc_\Xfr}\Oc_\Xfr)_+
\mid
\al(u) - u\in \Jc_\Xfr^m
\big\}.
\label{rjf84hf803h09fj39fj}
\end{align}
As already observed, in the case where $\Xfr  = \widehat\Xfr$ we recover \eqref{noiebu4ifeiofpe} by Lemma \ref{e8y93hf8h3f93}\emph{(iii)}. In analogy with \eqref{fuirghf89h0fj309j3} now we have:

\begin{PROP}\label{fh94hf89hf0j09f}
For any supermanifold $\Xfr$ there exists a short exact sequence of sheaves of groups,
\[
\xymatrix{
\{{\bf 1}\} \ar[r] & \Gc_{\Oc_\Xfr}^{(2)} \ar[r] & \Gc_{\Oc_\Xfr}^{(1)} \ar[r] & \mathcal Aut_{\Oc_X}T^*_{X, -}\ar[r] & \{{\bf 1}\}
}
\]
where we have identified $\Jc_\Xfr/\Jc^2_\Xfr$ with $T^*_{X, -}$. 
\end{PROP}

\noindent
Our proof of Proposition \ref{fh94hf89hf0j09f} will involve the following preliminary results.

\begin{LEM}\label{f783g79fh89fh8h4f}
Any automorphism $\al : \Oc_\Xfr\ra \Oc_\Xfr$ sends $\Jc_\Xfr^m \ra \Jc_\Xfr^m$, for any $m$.
\end{LEM}

\begin{proof}
Suppose there exists an automorphism $\al$ which reduces the power of the fermionic ideal, say $\Jc_\Xfr^m \ra \Jc_\Xfr^{m-m_\al}$ for some integer $m_\al$. Then the $n$-fold composition $\al^n$ sends $\Jc_\Xfr^m \ra \Jc_\Xfr^{m - nm_\al}$. Since $\Jc_\Xfr^{m - nm_\al} = (0)$ for $m - nm_\al< 0$, there will exist some $n$ such that $\al^n = 0$, i.e., $\al$ will be nilpotent. But then $\al$ cannot be invertible and so cannot be an automorphism. Thus any automorphism preserves the power of the fermionic ideal.
\end{proof}

\begin{COR}\label{rgf874g7fh398fh803jf3933}
Any automorphism of $\Oc_\Xfr$ will induce an automorphism of $\Oc_\Xfr/\Jc_\Xfr^m$ for any $m$. \qed
\end{COR}
%
%
%
\noindent
Recall the sequence of sheaves modulo powers of the fermionic ideal from \eqref{fg74f93hf83hf03}. By Corollary \ref{rgf874g7fh398fh803jf3933}, note that any $\al\in \Gc^{(m)}_{\Oc_\Xfr}$ will induce the following diagram:
\begin{align}
\xymatrix{
0\ar[r] & \Jc_\Xfr^m/\Jc_\Xfr^{m+1}\ar[d]^{\overline \al}  \ar[r] & \Oc_\Xfr/\Jc_\Xfr^{m+1}\ar[r] \ar[d]^{\overline\al}& \Oc_\Xfr/\Jc^m_\Xfr\ar[r] \ar@{=}[d]& 0
\\
0\ar[r] & \Jc^m_\Xfr/\Jc^{m+1}_\Xfr \ar[r] & \Oc_\Xfr/\Jc_\Xfr^{m+1}\ar[r] & \Oc_\Xfr/\Jc^m_\Xfr\ar[r] & 0
}
\label{g64gf784h9f83f0399}
\end{align}
Hence we obtain a map $\Gc_{\Oc_\Xfr}^{(m)} \ra \mathcal Aut_{\Oc_X}\Jc_\Xfr^m/\Jc_\Xfr^{m+1}$ given by: $\al\longmapsto \overline \al|_{\Jc_\Xfr^m/\Jc_\Xfr^{m+1}}$. The kernel of this restriction comprise those automorphisms $\be\in \Gc^{(m)}_{\Oc_\Xfr}$ inducing:
\begin{align}
\xymatrix{
0\ar[r] & \Jc_\Xfr^m/\Jc_\Xfr^{m+1}\ar@{=}[d]  \ar[r] & \Oc_\Xfr/\Jc_\Xfr^{m+1}\ar[r] \ar[d]^{\overline\be}& \Oc_\Xfr/\Jc^m_\Xfr\ar[r] \ar@{=}[d]& 0
\\
0\ar[r] & \Jc_\Xfr^m/\Jc_\Xfr^{m+1} \ar[r] & \Oc_\Xfr/\Jc_\Xfr^{m+1}\ar[r] & \Oc_\Xfr/\Jc_\Xfr^m\ar[r] & 0
}
\label{g64gf7833f34h9f83f0399}
\end{align}
We now arrive at the proof of the proposition motivating these subsequent considerations.
\\\\
\emph{Proof of Proposition $\ref{fh94hf89hf0j09f}$}. We will argue that:
\[
\ker\left\{\Gc_{\Oc_\Xfr}^{(1)} \ra \mathcal Aut_{\Oc_X}T^*_{X, -}\right\}
\cong 
\Gc_{\Oc_\Xfr}^{(2)}.
\]
Since $\Oc_\Xfr$ is globally $\Zbb_2$-graded we know that $\Oc_\Xfr/\Jc_\Xfr^2$ splits. That is, $\Oc_\Xfr/\Jc_\Xfr^2\cong (\Oc_\Xfr/\Jc_\Xfr)\oplus (\Jc_\Xfr/\Jc_\Xfr^2)$. Recall furthermore that any automorphism in $\Gc^{(1)}_{\Oc_\Xfr}$ will induce the identity modulo $\Jc_\Xfr$. Now on inspection of the diagram in \eqref{g64gf7833f34h9f83f0399}, specialised to $m = 1$, splitness of the middle term necessitates $\overline \be = {\bf 1}$, the identity. From our characterisation of automorphisms in \eqref{g64gf784h9f83f0399}, automorphisms in $\Gc^{(1)}_{\Oc_\Xfr}$ inducing the identity on $\Oc_\Xfr/\Jc_\Xfr^2$ are precisely elements in $\Gc^{(2)}_{\Oc_\Xfr}$, whence Proposition \ref{fh94hf89hf0j09f} follows.
\qed
\\\\
Applying Grothendieck's observation\footnote{c.f., remarks succeeding \eqref{fuirghf89h0fj309j3}} now to the sequence in Proposition \ref{fh94hf89hf0j09f}, there will exist an action $H^0(X, \mathcal Aut_{\Oc_X}T^*_{X, -})$ on the pointed set $\mbox{\v H}^1(X, \mathcal G_{\Oc_\Xfr}^{(2)})$. For the orbit set under this action we write:
\begin{align}
\Mfr_\Xfr 
\stackrel{\Delta}{=}
\frac{\mbox{\v H}^1(X, \mathcal G_{\Oc_\Xfr}^{(2)})}{H^0(X, \mathcal Aut_{\Oc_X}T^*_{X, -})}.
\label{hfegf73f8h30f34f}
\end{align}
In the case $\Xfr = \widehat\Xfr$ we recover Green's orbit set. In analogy then with Green's classification in Theorem \ref{rgf78g37fg39f3h80}, we have the following generalisation.

\begin{THM}\label{hf9hf983hf8h30f0}
For any supermanifold $\Xfr$ the set $\Mfr_\Xfr$ is a pointed set which classifies\footnote{i.e., is in bijective correspondence with the set of isomorphism classes of supermanifolds $\Xfr^\p$ which are locally isomorphic to $\Xfr$}  isomorphism classes of supermanifolds $\Xfr^\p$ which are locally isomorphic to $\Xfr$. The base-point in $\Mfr_\Xfr$ corresponds to the isomorphism class of $\Xfr$.
\end{THM}

\begin{proof}
More generally, to see that $\mbox{{\v H}}^1\big(X, \Gc^{(m)}_{\Oc_\Xfr}\big)$ is a pointed set for any $m$, its base-point can be described as follows: since $\Gc^{(m)}_{\Oc_\Xfr}$ is a sheaf of groups there exists a natural homomorphism from the sheaf of trivial groups $\{{\bf 1}\}\ra \Gc^{(m)}_{\Oc_\Xfr}$, given by ${\bf 1} \mapsto {\bf 1}_{\Oc_\Xfr}$. The \v Cech cohomology $\mbox{\v H}^1(X, \{{\bf 1}\})$ is a one-point set and the base-point in $\mbox{{\v H}}^1\big(X, \Gc^{(m)}_{\Oc_\Xfr}\big)$ is the image of this point under the inclusion $\{pt\} = \mbox{\v H}^1(X, \{{\bf 1}\})\ra \mbox{{\v H}}^1\big(X, \Gc^{(m)}_{\Oc_\Xfr}\big)$, which is induced by the morphism $\{{\bf 1}\}\ra \Gc^{(m)}_{\Oc_\Xfr}$.
Upon identifying $\Gc_{\Oc_\Xfr}^{(2)}$ with the sheaf of automorphisms of $\Xfr$ with preserve a given framing $\phi$, the set $\mbox{{\v H}}^1\big(X, \Gc^{(2)}_{\Oc_\Xfr}\big)$ can be identified with `twisted forms'. These are precisely supermanifolds $\Xfr^\p$ with framing $\phi$ and which are locally isomorphic to $\Xfr$. The base-point corresponds to the trivially twisted form $\Xfr$ itself. For further details on twisted forms more generally we refer to \cite[p. 134]{MILNE}. The group $H^0(X, \mathcal Aut_{\Oc_X} T^*_{X, -})$ acts by permuting the choice of framing $\phi$ and so the orbit set will classify isomorphism classes of supermanifolds, locally isomorphic to $\Xfr$. 
\end{proof}

\subsection*{Toward a Functor of Supermanifolds}
In analogy with work in algebraic geometry pertaining to classification and moduli,\footnote{see e.g., \cite[\S4]{HARTDEF}} it is desirable to form a functor from supermanifolds to sets, $\mathsf{SM}\Rightarrow \mathsf{Set}$. By Proposition \ref{fh794f98hf03jf930} we know that any supermanifold $\Xfr$ will have \emph{some} model $(X, T^*_{X, -})$; and Green's theorem (Theorem \ref{rgf78g37fg39f3h80}) affirms that its isomorphism class will define an element in a set, Green's orbit set, which itself only uses information about the model $(X, T^*_{X, -})$. The determination of a model $(X, T^*_{X, -})$ from the data of a supermanifold $\Xfr$ is not unique however since it requires the choice of a framing which, while extant, is by Definition \ref{rfh894hf8hf093j03} non-canonical. Hence we cannot functorially assign points in Green's orbit set from the data of a supermanifold. The purpose behind presenting the abstraction $\Gc_{\Oc_\Xfr}^{(m)}$ in \eqref{rjf84hf803h09fj39fj} is to emphasise a construction which is intrinsic to the data of a supermanifold and does not require a choice of framing. The notation in \eqref{hfegf73f8h30f34f} is therefore suggestive of our desired, $\mathsf{Set}$-valued functor of supermanifolds. At the level of objects we have, to each $\Xfr\in \mathrm{Ob}(\mathsf{SM})$ the assignment $\Xfr \stackrel{F}{\mapsto} \Mfr_\Xfr$. Modelling $\mathsf{SM}$ as a discrete category then, with any morphism the identity mapping we obtain a functor $F : \mathsf{SM} \Rightarrow \mathsf{Set}$. We will not further develop this line of thought now however and so will conclude on these open remarks. Theorem \ref{hf9hf983hf8h30f0} will suffice for our purposes in this article.

\section{Primary Obstructions and the Euler Differential}
\label{rhf78gf79h38fh380}

\noindent
Recall from Theorem \ref{rfg7f9hf3h80f3} that the primary obstruction to splitting a supermanifold $\Xfr$ measures the failure for $\Xfr$ to split. As 
mentioned in the succeeding comments however, its vanishing does not necessarily contribute any useful information on this splitting question. In contrast, by Theorem \ref{djvevcyibeon}, the Euler differential $\dt\e_{\widehat\Xfr}$ is a complete measure of splitting. Therefore, we might expect a natural relation between the Euler differential and the primary (and higher) obstructions and indeed such a relation exists. It is this relation which we will establish presently. To formulate the statement, recall Proposition \ref{fh7949v7h49f80fh} on the vanishing of the composition $H^0(X, T^{\{m\}}_{\widehat\Xfr})\ra H^{1}(X, T^{\{m+1\}}_{\widehat\Xfr})$. It follows then, for any $m$, that we obtain a mapping on cohomology $\dt^\p: H^0(X, T^{\{m\}}_{\widehat\Xfr})\ra H^{1}(X, T^{\{m+2\}}_{\widehat\Xfr})$ represented in the following diagram by the dotted arrow:\footnote{Note, the existence of the dashed arrow $\widetilde\dt: H^0(X, T^{\{m\}}_{\widehat\Xfr})\ra H^{1}(X, T^{(m+2)}_\Xfr)$ follows from exactness of the column.}
\begin{align}
\xymatrix{ 
& & H^{1}(T^{(m+2)}_{\Xfr})\ar[d]|{\hole}\ar[r] & H^{1}(T^{\{m+2\}}_{\widehat\Xfr})
\\
\cdots \ar[r] & H^0(T^{\{m\}}_{\widehat\Xfr}) \ar@{-->}[ur]^{\widetilde\dt} \ar@{.>}[urr]_>{\dt^\p} \ar[dr]_{\mathrm{zero}} \ar[r]_\dt  & H^{1}(T_\Xfr^{(m+1)})\ar[d] \ar[r] & \cdots
\\
& & H^{1}(T^{\{m+1\}}_{\widehat\Xfr})
}
\label{rhf97hf983hf03}
\end{align}
Where the obstruction sheaves are concerned, we have the general characterisation from Green in \cite{GREEN}:

\begin{LEM}\label{fh894hf98h3f8h30}
Let $\widehat\Xfr$ be the split model with reduced space $X$ and odd cotangent bundle $T^*_{X, -}$. Then for any $m$, the obstruction sheaf of the split model satisfies,
\[
\mathcal Ob^{(m)}_{\Oc_{\widehat\Xfr}}
\cong 
\left\{
\begin{array}{ll}
\wedge^mT^*_{X, -}\otimes_{\Oc_X}T_X & \mbox{if $m = 2\ell$ is even}
\\
\wedge^mT^*_{X, -}\otimes_{\Oc_X}T_{X, -}& \mbox{if $m = 2\ell+1$ is odd}.
\end{array}
\right.
\]\qed
\end{LEM}

\noindent
Hence Lemma \ref{rhf79gf983hf80h03} and Lemma \ref{fh894hf98h3f8h30} above give: $T_{\widehat\Xfr}^{\{2\}}/\mathcal Ob^{(3)}_{\Oc_{\widehat\Xfr}}\cong \mathcal Ob^{primary}_{\Oc_{\widehat\Xfr}}$ and so the exact piece on cohomology:
\begin{align}
\xymatrix{
\cdots
\ar[r]& 
\mathrm{OB}^{(3)}_{{\widehat\Xfr}}
\ar[r] & 
H^1(X, T^{\{2\}}_{\widehat\Xfr})
\ar[r]^{q_*} &  
\mathrm{OB}^{primary}_{{\widehat\Xfr}}
\ar[r] &  
\cdots
}
\label{rfyu4f8gfwwww38fh83hf03}
\end{align}
Upon specialising \eqref{rhf97hf983hf03} to $m= 0$ and utilising $q_*$ from \eqref{rfyu4f8gfwwww38fh83hf03}, we have:

\begin{THM}\label{rfh89hf8hf09j03}
For any supermanifold $\Xfr$,
\[
q_*\dt^\p \e_{\widehat\Xfr} = \eta_*[\Xfr]
\]
where $\e_{\widehat\Xfr}$is the Euler vector field and $\eta_*[\Xfr]$ is the primary obstruction to splitting $\Xfr$.
\end{THM}

\begin{REM}
\emph{Recall from Definition \ref{hf78gf793hf83hf03j} that the Euler differential is the image of $\e_{\widehat\Xfr}$ in $H^1(X, T^{(1)}_{\Xfr})$. From the diagram in \eqref{rhf97hf983hf03} we can, W.L.O.G., identify the Euler differential with $\dt^\p\e_{\widehat\Xfr}$. Hence Theorem \ref{rfh89hf8hf09j03} can be considered a relation between the Euler differential and the primary obstruction.}
\end{REM}

\begin{proof}
As with many of the other arguments provided in this article, our method for proving Theorem \ref{rfh89hf8hf09j03} will be to deduce it from commutativity of an appropriate diagram of cohomologies. To that extent we begin by observing, for any $m, \ell, \ell^\p$ with $m \leq \ell\leq \ell^\p$, that there will exist a short exact sequence:
\[
0
\lra 
\frac{T^{(\ell)}_\Xfr}{T^{(\ell^\p)}_\Xfr}
\lra
\frac{T^{(m)}_\Xfr}{T^{(\ell^\p)}_\Xfr}
\lra
\frac{T^{(m)}_\Xfr}{T^{(\ell)}_\Xfr}
\lra 
0
\]
Appropriately specialising the above sequence gives the following diagram:
\begin{align}
\xymatrix{
0\ar[r]& T^{\{2\}}_{\widehat\Xfr} \ar[d] \ar[r] & T^{(0)}_\Xfr/T^{(3)}_\Xfr \ar[r] \ar@{=}[d] & T^{(0)}_\Xfr/T^{(2)}_\Xfr\ar[d]\ar[r] &0 
\\
0\ar[r] & T^{(1)}_\Xfr/T^{(3)}_\Xfr\ar[d] \ar[r] & T^{(0)}_\Xfr/T^{(3)}_\Xfr \ar[r] & T^{\{0\}}_{\widehat\Xfr}\ar[r] & 0
\\
& T^{\{1\}}_{\widehat\Xfr}
}
\label{rfh793hf80f9j3233}
\end{align}
where we have used that $T^{\{m\}}_{\widehat\Xfr} \cong T^{(m)}_\Xfr/T^{(m+1)}_\Xfr$. Regarding the middle row in \eqref{rfh793hf80f9j3233} observe that it comes from reducing the initial form sequence in Lemma \ref{rhf78f9hf083j09fj30} by $T_\Xfr^{(3)}$. That is, we have:
\begin{align}
\xymatrix{
T^{(1)}_\Xfr \ar[r] \ar[d] & T^{(0)}_\Xfr \ar[r] \ar[d]  & T^{\{0\}}_{\widehat\Xfr}\ar@{=}[d]
\\
T^{(1)}_\Xfr/T^{(3)}_\Xfr\ar[r] & T^{(0)}_\Xfr/T^{(3)}_\Xfr \ar[r]  & T^{\{0\}}_{\widehat\Xfr}
}
&&
\implies
&&
\xymatrix{
\ar@{=}[d] H^0(T^{\{0\}}_{\widehat\Xfr}) \ar[r]^\dt & H^1(T^{(1)}_\Xfr) \ar[d] 
\\
H^0(T^{\{0\}}_{\widehat\Xfr})  \ar[r]^{\pt} & H^1(T^{(1)}_\Xfr/T^{(3)}_\Xfr) 
}
\label{hfjvjvbjhdbvjkdnk}
\end{align}
where the latter diagram above on cohomology commutes. On considering a morphism of tangent sheaves analogous to the former in \eqref{hfjvjvbjhdbvjkdnk} shifted in degree we can form another commuting diagram, essentially reducing that in cohomology in \eqref{hfjvjvbjhdbvjkdnk} by $T^{(2)}_\Xfr$. Combining this with \eqref{hfjvjvbjhdbvjkdnk} results in: 
\[
\xymatrix{
H^0(T^{\{0\}}_{\widehat\Xfr})
\ar@/_/[ddr]_{\mathrm{zero}} \ar[dr]^\dt \ar@/^/[drr]^\pt & & 
\\
& H^1(T^{(1)}_\Xfr)  \ar[d]^{p_*} \ar[r] &  H^1(T^{(1)}_\Xfr/T^{(3)}_\Xfr)\ar[d]^{p_*} 
\\
& H^1(T^{\{1\}}_\Xfr) \ar@{=}[r] & H^1(T^{\{1\}}_{\widehat\Xfr})
}
\]
Hence,
\begin{align}
\dt   = \pt  \mod T_\Xfr^{(3)}
&&
\mbox{and}
&&
0 = p_*\dt  =  p_*\pt.
\label{rfh974hf8hf093j93333}
\end{align}
We return now to the diagram in \eqref{rfh793hf80f9j3233}. It induces the following on cohomology:
\begin{align}
\xymatrix{
H^0(T^{(0)}_\Xfr/T^{(2)}_\Xfr) \ar[rr] \ar[d] & & H^1(T^{\{2\}}_{\widehat\Xfr})\ar[d]
\\
H^0(T^{\{0\}}_{\widehat\Xfr})\ar[drr]_{\mathrm{zero}} \ar@{.>}[urr]^{\pt^\p} \ar[rr]^\pt & & H^1(T^{(1)}_\Xfr/T^{(3)}_\Xfr)\ar[d]^{p_*} 
\\
& & H^1(T^{\{1\}}_{\widehat\Xfr})
}
\label{rfh89hf83hf09fj30}
\end{align}
Since $p_*\pt = 0$ from \eqref{rfh974hf8hf093j93333}, the existence of $\pt^\p$, represented above by the dotted arrow, is guaranteed. With the former relation in \eqref{rfh974hf8hf093j93333} then, we therefore have:
\begin{align}
\dt^\p\e_{\widehat\Xfr}= \pt^\p \e_{\widehat\Xfr}.
\label{rfhu9ehf89hf0j309fj03}
\end{align}
To complete the proof of this theorem, we will need to recall from Lemma \ref{rfg783gf73gf983hf9} that, for any $m$, $\Gc^{(m+2)}_{\Oc_\Xfr}\subset \mathcal \Gc^{(m)}_{\Oc_\Xfr}$ is normal. Onishchik in \cite{ONISHCLASS} shows:
\begin{align}
\frac{\Gc^{(2m)}_{\Oc_\Xfr}}{\Gc^{(2m+2)}_{\Oc_\Xfr}} \cong T^{\{2m\}}_{\widehat\Xfr}. 
\label{dcoienoeoimoievior}
\end{align}
Now for any supermanifold $\Xfr$ the formal exponential of derivations of positive degree will be finite. This leads to an isomorphism of sheaves of \emph{sets}, $T_\Xfr^{(m)} \stackrel{\cong}{\ra} \Gc^{(m)}_{\Oc_\Xfr}$, given by $\nu \mapsto \exp\nu = {\bf 1} + \nu + \frac{1}{2!}\nu^2 +\cdots$. We obtain therefore a morphism of sheaves, where it should be remembered that only the middle, vertical morphism is as sheaves of sets:
\[
\xymatrix{
\ar[d] \Gc^{(2)}_{\Oc_\Xfr} \ar[r] & \Gc^{(1)}_{\Oc_\Xfr}\ar[dd] \ar[r] & \mathcal Aut_{\Oc_X}T^*_{X, -}\ar@{^{(}->}[d]
\\
T^{\{2\}}_{\widehat\Xfr} \ar@{^{(}->}[d] &  & \mathcal End_{\Oc_X}T^*_{X, -}\ar@{^{(}->}[d]
\\
T^{(1)}_\Xfr/T^{(3)}_\Xfr \ar[r] & T^{(0)}_\Xfr/T^{(3)}_\Xfr \ar[r] & T^{\{0\}}_{\widehat\Xfr}
}
\]
Note that even though the middle-vertical arrow is as sheaves of sets, the left- and right-most, vertical morphisms are as sheaves of groups and so we have induced a commuting diagram on cohomology:
\begin{align}
\xymatrix{
\ar[dd] H^0(\mathcal Aut_{\Oc_X}T^*_{X, -}) \ar[rr] & & \mbox{\v H}^1(\Gc_{\Oc_\Xfr}^{(2)})\ar[d]^{\eta^\p} \ar[dr]^{\eta_*}
\\
& & H^1(T^{\{2\}}_{\widehat\Xfr})\ar[d]\ar[r]_{q_*} & \mathrm{OB}^{primary}_{\Oc_\Xfr}
\\
H^0(T^{\{0\}}_{\widehat\Xfr})\ar@{.>}[urr]^{\pt^\p} \ar[rr]^\pt & & H^1(T^{(1)}_\Xfr/T^{(3)}_\Xfr)
}
\label{rbf973f93hf08j309fj3}
\end{align}
where, by construction, $\pt$ is the boundary mapping from \eqref{hfjvjvbjhdbvjkdnk} with lift $\pt^\p$ from \eqref{rfh89hf83hf09fj30}.
Note that we have also appended the triangle of mappings $\eta^\p, \eta_*$ and $q_*$. By Definition \ref{djkbkcvjfvjcekjlenklw}, the Euler vector field $\e_{\widehat\Xfr}$ is in the image of ${\bf 1}_{T^*_{X, -}}\in H^0(X, \mathcal Aut_{\Oc_X}T^*_{X, -})$. Furthermore the class of $\Xfr$, $[\Xfr]\in \mbox{\v H}^1(X, \Gc_{\Oc_\Xfr}^{(2)})$, also lies in the image of ${\bf 1}_{T^*_{X, -}}$ since it corresponds to the base-point (see e.g., Theorem \ref{hf9hf983hf8h30f0} and the proof of Proposition \ref{rffg874g97h9fh83}).
Now since $\pt^\p$ lifts $\pt$, it follows that the sub-diagram in \eqref{rbf973f93hf08j309fj3} formed by $\pt^\p$ and $\eta^\p$ commute. Hence, and by \eqref{rfhu9ehf89hf0j309fj03}, we have
\begin{align}
\dt^\p\e_{\widehat\Xfr} = \pt^\p \e_{\widehat\Xfr} = \eta^\p[\Xfr].
\label{rhf9h3f98309fj39jf}
\end{align}
Finally, the appended triangle of mappings $(\eta^\p, \eta_*, q_*)$ in \eqref{rbf973f93hf08j309fj3} will commute since it is obtained from the following commuting diagram
\[
\xymatrix{
\Gc^{(2)}_{\Oc_\Xfr}\ar[d]\ar[drr] & & 
\\
\Gc^{(2)}_{\Oc_\Xfr}/ \Gc^{(4)}_{\Oc_\Xfr} \ar[rr] & & \Gc^{(2)}_{\Oc_\Xfr}/ \Gc^{(3)}_{\Oc_\Xfr} 
}
\]
where we used that $\Gc^{(2)}_{\Oc_\Xfr}/ \Gc^{(4)}_{\Oc_\Xfr} \cong T_{\widehat\Xfr}^{\{2\}}$ by \eqref{dcoienoeoimoievior}; and $ \Gc^{(2)}_{\Oc_\Xfr}/ \Gc^{(3)}_{\Oc_\Xfr}\stackrel{\Delta}{=} \mathcal Ob^{primary}_{\Oc_\Xfr}$ by \eqref{rhf89hf89hf093fj3j}. Theorem \ref{rfh89hf8hf09j03} now follows from the definition of the primary obstruction in Definition \ref{fbivurviuebfoenie}, the identities in \eqref{rhf9h3f98309fj39jf}; and commutativity of the triangle of mappings $(\eta^\p, \eta_*, q_*)$ in \eqref{rbf973f93hf08j309fj3}.
\end{proof}

\newpage


\bibliographystyle{alpha}
\bibliography{Bibliography}

\hfill
\\
\noindent
\small
\textsc{
Kowshik Bettadapura 
\\
\emph{Yau Mathematical Sciences Center} 
\\
Tsinghua University
\\
Beijing, 100084, China}
\\
\emph{E-mail address:} \href{mailto:kowshik@mail.tsinghua.edu.cn}{kowshik@mail.tsinghua.edu.cn}

\end{document}